\documentclass[12pt]{article}
\usepackage{latexsym, amssymb, amsmath, amscd, amsfonts, epsfig, graphicx, colordvi,verbatim,ifpdf,extarrows}
\usepackage{amsfonts, amsmath, amssymb}
\usepackage{amssymb,amsfonts,amsmath,latexsym,epsfig,cite, psfrag,eepic,color}
\usepackage{amscd,graphics}
\usepackage{latexsym, amssymb,  amsmath,amscd, amsfonts, epsfig, graphicx, colordvi,amsthm}

\usepackage{graphicx}
\usepackage{epstopdf}
\usepackage{color}
\usepackage{ifpdf}
\usepackage{fancybox}
\usepackage[font=small,labelfont=bf,labelsep=none]{caption}
\usepackage{float}

\newtheorem{thm}{Theorem}[section]
\newtheorem{prop}[thm]{Proposition}

\newtheorem{rem}[thm]{\it Remark}
\newtheorem{defi}[thm]{Definition}
\newtheorem{lem}[thm]{Lemma}
\newtheorem{core}[thm]{Corollary}

\def\pf{\noindent{\it Proof.} }
\def\qed{\nopagebreak\hfill{\rule{4pt}{7pt}}
\medbreak}

\numberwithin{equation}{section}

\def\qed{\nopagebreak\hfill{\rule{4pt}{7pt}}
\medbreak}

\newlength{\boxedparwidth}
  {\begin{center} \begin{tabular}{|@{\hspace{.315in}}c@{\hspace{.15in}}|}
                  \hline \\ \begin{minipage}[t]{\boxedparwidth}
                  \setlength{\parindent}{.25in}}%
  {\end{minipage} \\ \\ \hline \end{tabular} \end{center}}

\begin{document}

\begin{center}

 {\Large \bf New companions to the generations of the G\"ollnitz-Gordon identities}
\end{center}

\begin{center}
 {Thomas Y. He}$^{1}$ and
  {Alice X.H. Zhao}$^{2}$ \vskip 2mm

$^{1}$ School of Mathematical Sciences, Laurent Mathematics Center and V.C. \&  V.R. Key Lab, Sichuan Normal University, Chengdu 610066, P.R. China\\[6pt]

  $^{2}$ College of Science, Tianjin University of Technology, Tianjin 300384, P.R. China

   \vskip 2mm

    $^1$heyao@sicnu.edu.cn,  $^2$zhaoxh@email.tjut.edu.cn
\end{center}

\vskip 6mm   {\noindent \bf Abstract.} The G\"ollnitz-Gordon identities were found by G\"ollnitz and Gordon independently. In 1967, Andrews obtained a combinatorial generalization
of the  G\"ollnitz-Gordon identities, called the Andrews-G\"ollnitz-Gordon theorem. In 1980, Bressoud extended the Andrews-G\"ollnitz-Gordon theorem to even moduli, called the Bressoud-G\"ollnitz-Gordon theorem. Furthermore, Bressoud gave the generating functions for the generalizations of the G\"ollnitz-Gordon identities. In this article, we will give new companions to the generalizations of the G\"ollnitz-Gordon identities.

\noindent {\bf Keywords}: G\"ollnitz-Gordon identities, Bailey pairs, G\"ollnitz-Gordon marking

\noindent {\bf AMS Classifications}: 05A17, 11P81, 11P84

\section{Introduction}

A partition of a positive integer $n$ is a finite non-increasing sequence of positive integers $\lambda_1\geq\lambda_2\geq\cdots\geq\lambda_\ell>0$ such that $\lambda_1+\lambda_2+\cdots+\lambda_\ell=n.$ The $\lambda_s$ are called the parts of the partition.  We sometimes write $\lambda=(1^{f_1}2^{f_2}3^{f_3}\cdots)$, where $f_t(\lambda)$  denotes the number of parts equal to $t$ in $\lambda$. Let $\ell(\lambda)$ be the number of parts of $\lambda$ and  $|\lambda|$ be the sum of parts of $\lambda$.
The  G\"ollnitz-Gordon identities were found by G\"ollnitz \cite{Gollnitz-1960,Gollnitz-1967} and Gordon \cite{Gordon-1962,Gordon-1965} independently.

\begin{thm}[First G\"ollnitz-Gordon identity]\label{first-G-G-thm}
The number of partitions of $n$ into distinct nonconsecutive
parts with no even parts differing by exactly $2$ equals the number of partitions of $n$ into parts $\equiv1,4,7\pmod 8$.
\end{thm}

\begin{thm}[Second G\"ollnitz-Gordon identity]\label{second-G-G-thm}
 The number of partitions of $n$ into distinct nonconsecutive
parts with no even parts differing by exactly $2$ and no parts equal to $1$ or $2$ equals the number of partitions of $n$ into parts $\equiv3,4,5\pmod 8$.
\end{thm}

The generating function version of the G\"ollnitz-Gordon identities was given as follows.
 \begin{equation}\label{first-G-G-I}
\sum_{n\geq0}\frac{q^{n^2}(-q;q^2)_n}{(q^2;q^2)_n}=\frac{1}{(q,q^4,q^7;q^8)_\infty},
\end{equation}
and
\begin{equation}\label{second-G-G-I}
\sum_{n\geq0}\frac{q^{n^2+2n}(-q;q^2)_n}{(q^2;q^2)_n}=\frac{1}{(q^3,q^4,q^5;q^8)_\infty}.
\end{equation}

Here and in the sequel, we assume that $|q|<1$ and adopt the standard notation \cite{Andrews-1976}:
\[(a;q)_\infty=\prod_{i=0}^{\infty}(1-aq^i), \quad (a;q)_n=\frac{(a;q)_\infty}{(aq^n;q)_\infty},\]
and
\[(a_1,a_2,\ldots,a_m;q)_\infty=(a_1;q)_\infty(a_2;q)_\infty\cdots(a_m;q)_\infty.
\]

In 1967, Andrews \cite{Andrews-1967} found the following combinatorial generalization of the G\"ollnitz-Gordon identities, which can be called the   Andrews-G\"ollnitz-Gordon theorem.

\begin{thm}[Andrews-G\"ollnitz-Gordon identities]\label{Gordon-Gollnitz-odd}
For $k\geq i\geq1$, let $C_1(k,i;n)$ denote the number of partitions $\lambda$ of $n$ of the form $(1^{f_1}\,2^{f_{2}}\,3^{f_3}\cdots)$, such that
\begin{itemize}
\item[{\rm (1)}] $f_1(\lambda)+f_2(\lambda)\leq i-1${\rm;}

\item[{\rm (2)}] $f_{2t+1}(\lambda)\leq1${\rm;}

\item[{\rm (3)}] $f_{2t}(\lambda)+f_{2t+1}(\lambda)+f_{2t+2}(\lambda)\leq k-1$.
\end{itemize}
For $k\geq i\geq1$, let $D_1(k,i;n)$ denote the number of partitions of $n$ into parts $\not\equiv2\pmod4$ and $\not\equiv0,\pm(2i-1)\pmod{4k}$. Then for $k\geq i\geq1$ and  $n\geq0$,
\[C_1(k,i;n)=D_1(k,i;n).\]
\end{thm}

Setting $k=2$ and $i=2$ or $1$,  Theorem \ref{Gordon-Gollnitz-odd} reduces to Theorem \ref{first-G-G-thm} and Theorem \ref{second-G-G-thm}, respectively. In 1980, Bressoud \cite{Bressoud-1980} extended the Andrews-G\"ollnitz-Gordon identities to even moduli, which has been called the  Bressoud-G\"ollnitz-Gordon theorem.

\begin{thm}[Bressoud-G\"ollnitz-Gordon identities]\label{Gordon-Gollnitz-even}
For $k> i\geq1$, let $C_0(k,i;n)$ denote the number of partitions of $n$ of the form $\lambda=(1^{f_1} 2^{f_{2}}3^{f_3}\cdots)$ such that
\begin{itemize}
\item[{\rm (1)}] $f_1(\lambda)+f_2(\lambda)\leq i-1$\rm{;}

\item[{\rm (2)}] $f_{2t+1}(\lambda)\leq1$\rm{;}

\item[{\rm (3)}] $f_{2t}(\lambda)+f_{2t+1}(\lambda)+f_{2t+2}(\lambda)\leq k-1$\rm{;}

\item[{\rm (4)}] if $f_{2t}(\lambda)+f_{2t+1}(\lambda)+f_{2t+2}(\lambda)=k-1$, then $tf_{2t}(\lambda)+tf_{2t+1}(\lambda)+(t+1)f_{2t+2}(\lambda)\equiv {O}_\lambda(2t+1)+i-1\pmod{2}$, where ${O}_\lambda(N)$ denotes the number of odd parts not exceeding $N$ in $\lambda$.
\end{itemize}

\noindent For $k>i\geq1$, let $D_0(k,i;n)$ denote the number of partitions of $n$ into parts $\not\equiv2\pmod4$ and $\not\equiv0,\pm(2i-1)\pmod{4k-2}$. Then, for $k>i\geq1$ and $n\geq0$,
\[C_0(k,i;n)=D_0(k,i;n).\]
\end{thm}

Bressoud \cite{Bressoud-1980} also gave the generating function versions of Theorems \ref{Gordon-Gollnitz-odd} and \ref{Gordon-Gollnitz-even}.

\begin{thm}[Bressoud, 1980]\label{Gollnitz-3-e}
For  $j=0$ or $1$ and $(2k+j)/2> i\geq1$,
\begin{equation}\label{Gollnitz-A2}
\begin{split}
&\sum_{N_1\geq \cdots \geq N_{k-1}\geq0}\frac{(-q^{1-2N_1};q^2)_{N_1}
q^{2(N^2_1+\cdots+N^2_{k-1}+N_i+\cdots+N_{k-1})}}
{(q^2;q^2)_{N_1-N_2}\cdots(q^2;q^2)_{N_{k-2}-N_{k-1}}(q^{4-2j};q^{4-2j})_{N_{k-1}}}\\
&=\frac{(q^2;q^4)_\infty(q^{2i-1},q^{4k-2i-1+2j},q^{4k-2+2j};q^{4k-2+2j})_\infty}
{(q;q)_\infty}.
\end{split}
\end{equation}
\end{thm}

In this article, we will give a new generating function version of Theorems \ref{Gordon-Gollnitz-odd} and \ref{Gordon-Gollnitz-even} with the aid of Bailey pairs, which is stated as follows.

\begin{thm}\label{Gollnitz-3-ee} For $j=0$ or $1$, $k\geq i\geq2$ and $(2k+j)/2>2$,
\begin{equation}\label{Gollnitz-AA1}
\begin{split}
&\sum_{N_1\geq \cdots\geq N_{k-1}\geq0}\frac{q^{2(N_1^2+\cdots+N_{k-1}^2+N_i+\cdots+N_{k-1})}(-q^{1-2N_2};q^2)_{N_2}(-q^{1+2N_1};q^2)_\infty}{(q^2;q^2)_{N_1-N_2}\cdots(q^2;q^2)_{N_{k-2}-N_{k-1}} (q^{4-2j};q^{4-2j})_{N_{k-1}}}\\
&=\frac{(q^2;q^4)_\infty(q^{4k-2+2j},q^{2i-1},q^{4k-2i-1+2j};q^{4k-2+2j})_\infty}
{(q;q)_\infty}.
\end{split}
\end{equation}
\end{thm}

We will show that the left-hand side of \eqref{Gollnitz-AA1} can be interpreted combinatorially as the generating function of $C_j(k,i;n)$  in Theorems \ref{Gordon-Gollnitz-odd} and \ref{Gordon-Gollnitz-even}. More precisely, We will give the following formula for the generating function of $C_j(k,i;m,n)$,
where $C_j(k,i;m,n)$ denotes the number of partitions enumerated by $C_j(k,i;n)$ with exactly $m$ parts.

\begin{thm}\label{main-thm}
For $j=0$ or $1$, $k\geq i\geq2$ and $(2k+j)/2>2$,
\begin{equation}\label{main-eqn}
\begin{split}
&\displaystyle\sum_{m,n\geq0}C_j(k,i;m,n)x^mq^n\\
&=\sum_{N_1\geq \cdots\geq N_{k-1}\geq0}\frac{q^{2(N_1^2+\cdots+N_{k-1}^2+N_i+\cdots+N_{k-1})}(-q^{1-2N_2};q^2)_{N_2}(-xq^{1+2N_1};q^2)_\infty x^{N_1+\cdots+N_{k-1}}}{(q^2;q^2)_{N_1-N_2}\cdots(q^2;q^2)_{N_{k-2}-N_{k-1}} (q^{4-2j};q^{4-2j})_{N_{k-1}}}.
\end{split}
\end{equation}
\end{thm}

 This paper is organized as follows. In Section 2, we will give a proof of Theorem \ref{Gollnitz-3-ee} with the aid of Bailey pairs.  Section 3 is a
preliminary for the proof of Theorem 
\ref{main-thm}, in which we introduce the definition of the G\"ollnitz-Gordon marking of a partition, and give equivalent statements of Theorem \ref{main-thm}. Section 4 is devoted to introducing the insertion operation and the separation operation. These operations allow us to provide a proof of Theorem \ref{equiv-main-thm-2}. In order to show  Theorem \ref{equiv-main-thm}, we  introduce the reduction operation and the dilation operation in Section 5. In light of these operations, we finally give a proof of Theorem \ref{equiv-main-thm} in Section 6.

\section{Proof of Theorem \ref{Gollnitz-3-ee}}

In this section, we will give a proof of Theorem \ref{Gollnitz-3-ee} by using some related results on Bailey pairs. To this end, we  first give a brief review  of Bailey pairs.   For more information on Bailey pairs, see, for example, \cite{Agarwal-Andrews-Bressoud, Andrews-1986, Andrews-2000, Bressoud-Ismail-Stanton-2000, Lovejoy-2004b, Paule-1987,Warnaar-2001}.  Recall that  a pair of sequences $(\alpha_n(a,q),\beta_n(a,q))$ is called a Bailey pair relative to $(a,q)$  if  for  $n\geq 0,$
\begin{equation*}\label{bailey pair}
\beta_n(a,q)=\sum_{r=0}^n\frac{\alpha_r(a,q)}{(q;q)_{n-r}(aq;q)_{n+r}}.
\end{equation*}

The following formulation of Bailey's lemma was given by Andrews \cite{Andrews-1984, Andrews-1986}.

\begin{thm}[Bailey's lemma]\label{BL}
	If $(\alpha_n(a,q),\beta_n(a,q))$ is  a Bailey pair  relative to  $(a,q)$,
	then $(\alpha_n'(a,q),\beta_n'(a,q))$ is also a Bailey pair relative to  $(a,q)$, where
	\begin{equation*}\label{BL-eq}
	\begin{split}
	\alpha_n'(a,q)&=\frac{(\rho_1;q)_n(\rho_2;q)_n}{(aq/\rho_1;q)_n(aq/\rho_2;q)_n}
	\left(\frac{aq}
	{\rho_1\rho_2}\right)^n\alpha_n(a,q),\\
	\beta_n'(a,q)&=
	\sum_{r=0}^{n}\frac{(\rho_1;q)_r(\rho_2;q)_r
		(aq/\rho_1\rho_2;q)_{n-r}}
	{(aq/\rho_1;q)_n(aq/\rho_2;q)_n(q;q)_{n-r}}
	\left(\frac{aq}{\rho_1\rho_2}\right)^r\beta_r(a,q).
	\end{split}
	\end{equation*}
\end{thm}

When $\rho_1,\rho_2\rightarrow \infty$ and $\rho_1\rightarrow \infty,\rho_2=-\sqrt{aq}$, Bailey's lemma reduces to the following two forms, which play an important role in proof of Theorem \ref{Gollnitz-3-e}.

\begin{lem}\label{BL-1}
	If $(\alpha_n(a,q),\beta_n(a,q))$ is  a Bailey pair  relative to  $(a,q)$,
	then $(\alpha_n'(a,q),\beta_n'(a,q))$ is also a Bailey pair relative to  $(a,q)$, where
	\begin{equation*}\label{BL-1-eq}
	\begin{split}
	\alpha_n'(a,q)&=a^nq^{n^2}\alpha_n(a,q),\\
	\beta_n'(a,q)&=
	\sum_{r=0}^{n}\frac{a^rq^{r^2}}
	{(q;q)_{n-r}}
	\beta_r(a,q).
	\end{split}
	\end{equation*}
\end{lem}

\begin{lem}\label{BL-2}
	If $(\alpha_n(a,q),\beta_n(a,q))$ is  a Bailey pair  relative to  $(a,q)$,
	then $(\alpha_n'(a,q),\beta_n'(a,q))$ is also a Bailey pair relative to  $(a,q)$, where
	\begin{align*}
	\alpha_n'(a,q)&=a^{n/2}q^{n^2/2}\alpha_n(a,q),\\
	\beta_n'(a,q)&=\sum_{r=0}^n
	\frac{(-\sqrt{aq};q)_r}{(-\sqrt{aq};q)_n(q;q)_{n-r}}a^{r/2}q^{r^2/2}
	\beta_r(a,q).
	\end{align*}
\end{lem}

By successively using Lemma \ref{BL-1} and the following proposition, together with Lemma \ref{BL-2}, Bressoud, Ismail and Stanton found a proof of Theorem \ref{Gollnitz-3-e}.

\begin{prop}\label{Bl-Bp}{\rm  \cite[Proposition 4.1]{Bressoud-Ismail-Stanton-2000}}
	If $(\alpha_n(1,q),\beta_n(1,q))$ is a Bailey pair relative to $(1,q)$, where
	\[\alpha_n(1,q)=\left\{
	\begin{array}{ll}
	1, & \hbox{if $n=0$}, \\[5pt]
	(-1)^nq^{An^2}(q^{(A-1)n}+q^{-(A-1)n}), & \hbox{if $n\geq 1$,}
	\end{array}
	\right.
	\]
	then $(\alpha_n'(1,q),\beta_n'(1,q))$ is also a Bailey pair relative to $(1,q)$, where
	\begin{align*}
	\alpha_n'(1,q)&=\left\{
	\begin{array}{ll}
	1, & \hbox{if $n=0$}, \\[3pt]
	(-1)^nq^{An^2}(q^{An}+q^{-An}), & \hbox{if $n\geq 1$,}
	\end{array}
	\right. \\[10pt]
	\beta_n'(1,q)&=q^n\beta_n(1,q).
	\end{align*}
\end{prop}
Using Lemma \ref{BL-1} followed by Proposition \ref{Bl-Bp}, we
derive the following corollary.
\begin{core}\label{corBL}
	If  $(\alpha_n(1,q),\beta_n(1,q))$ is a Bailey pair relative to $(1,q)$, where
\[\alpha_n(1,q)=\left\{
\begin{array}{ll}
1, & \hbox{if $n=0$}, \\[3pt]
(-1)^nq^{An^2}(q^{An}+q^{-An}), & \hbox{if $n\geq 1$,}
\end{array}
\right.
\]
then $(\alpha_n'(1,q),\beta_n'(1,q))$ is also a Bailey pair relative to $(1,q)$, where
\begin{align*}
\alpha_n'(1,q)&=\left\{
\begin{array}{ll}
1, & \hbox{if $n=0$}, \\[5pt]
(-1)^nq^{(A+1)n^2}(q^{(A+1)n}+q^{-(A+1)n}), & \hbox{if $n\geq 1$,}
\end{array}
\right. \\[10pt]
\beta_n'(1,q)&=q^n\sum_{r=0}^{n}\frac{q^{r^2}}
{(q;q)_{n-r}}
\beta_r(1,q).
	\end{align*}
\end{core}

Now we are ready to show Theorem \ref{Gollnitz-3-ee}.

\begin{proof}[Proof of Theorem \ref{Gollnitz-3-ee}]
We begin with the  Bailey pair  $(\alpha^{(0)}_n(1,q),\beta^{(0)}_n(1,q))$ relative to $(1,q)$, where
\begin{align*}
&\alpha^{(0)}_n(1,q)=\left\{
\begin{array}{ll}
1, & \hbox{if $n=0$,}\\[5pt]
(-1)^nq^{jn^2/2}(q^{-jn/2}+q^{jn/2})
, & \hbox{if $n\geq1$.}
\end{array}
\right.
\end{align*}
 It is easy to see that when $j=1$,
\begin{align*}
&\beta^{(0)}_n(1,q)=\left\{
\begin{array}{ll}
1, & \hbox{if $n=0$,}\\[5pt]
0, & \hbox{if $n\geq1$.}
\end{array}
\right.
\end{align*}
The Bailey pair $(\alpha^{(0)}_n(1,q),\beta^{(0)}_n(1,q))$ corresponds to the unit  Bailey pair \cite[H(17)]{Slater-1952}. And when $j=0$,
$$\beta^{(0)}_n(1,q)=\frac{(-1)^n}{(q^2;q^2)_n}.$$
It corresponds to the Bailey pair   \cite[E(2)]{Slater-1952}.

Applying Corollary \ref{corBL} to $(\alpha^{(0)}_n(1,q),\beta^{(0)}_n(1,q))$, we get
\begin{align*}
&\alpha^{(1)}_n(1,q)=\left\{
\begin{array}{ll}
1, & \hbox{if $n=0$}, \\[5pt]
(-1)^nq^{(j+2)n^2/2}(q^{-(j+2)n/2}+q^{(j+2)n/2}), & \hbox{if $n\geq1$,}
\end{array}
\right.\\[5pt]
&\beta^{(1)}_n(1,q)=q^n\sum_{N_k=0}^{n}\frac{q^{N_k^2}}
{(q;q)_{n-N_k}}
\beta^{(0)}_{N_k}(1,q).
\end{align*}
It is easy to check that
$$\beta^{(1)}_n(1,q)=
\frac{q^n}{(q^{2-j};q^{2-j})_n},$$
in which the $j=0$ case follows from the $q$-Chu-Vandermonde formula \cite[Eq. (II.7)]{Gasper-Rahman-1990} with $c=-q$ and $a\rightarrow \infty$,
\begin{align}\label{qCV}
\sum_{r=0}^n\frac{(a;q)_r(q^{-n};q)_r}{(c;q)_r(q;q)_r}\left(\frac{cq^n}{a}\right)^r=\frac{(c/a;q)_n}{(c;q)_n}.
\end{align}
Applying  Corollary \ref{corBL} $k-i-1$  times to $(\alpha^{(1)}_n(1,q),\beta^{(1)}_n(1,q))$ yields the Bailey pair $(\alpha^{(k-i)}_n(1,q),\beta^{(k-i)}_n(1,q))$, where
\begin{align*}
&\alpha^{(k-i)}_n(1,q)=\left\{
\begin{array}{ll}
1, & \hbox{if $n=0$}, \\[5pt]
(-1)^nq^{\frac{2k-2i+j}{2}n^2}(q^{-\frac{2k-2i+j}{2}n}+q^{\frac{2k-2i+j}{2}n}), & \hbox{if $n\geq1$,}
\end{array}
\right.\\[10pt]
&\beta^{(k-i)}_n(1,q)=q^n\sum_{n\geq N_{i+1}\geq \cdots\geq N_{k-1}\geq 0}\frac{q^{N^2_{i+1}+\cdots+N^2_{k-1}+N_{i+1}+\cdots+N_{k-1}}}{(q;q)_{n-N_{i+1}}(q;q)_{N_{i+1}-N_{i+2}}\cdots(q^{2-j};q^{2-j})_{N_{k-1}}}.
\end{align*}

Next we substitute the Bailey pair
$(\alpha^{(k-i)}_n(1,q),\beta^{(k-i)}_n(1,q))$ into Lemma
\ref{BL-1} to get the following Bailey pair
\begin{equation*}
\begin{split}
&\alpha^{(k-i+1)}_n(1,q)=\left\{
\begin{array}{ll}
1, & \hbox{if $n=0$}, \\[5pt]
(-1)^nq^{\frac{2k-2i+j+2}{2}n^2}(q^{-\frac{2k-2i+j}{2}n}+q^{\frac{2k-2i+j}{2}n}), & \hbox{if $n\geq1$,}
\end{array}
\right.\\[10pt]
&\beta^{(k-i+1)}_n(1,q)=\sum_{n\geq N_{i}\geq \cdots\geq N_{k-1}\geq 0}\frac{q^{N^2_{i}+\cdots+N^2_{k-1}+N_{i}+\cdots+N_{k-1}}}{(q;q)_{n-N_{i}}(q;q)_{N_{i}-N_{i+1}}\cdots(q^{2-j};q^{2-j})_{N_{k-1}}}.
\end{split}
\end{equation*}

Applying Lemma
\ref{BL-1} to the Bailey pair
$(\alpha^{(k-i+1)}_n(1,q),\beta^{(k-i+1)}_n(1,q))$ $i-3$ times, we will obtain the following Bailey pair
\begin{equation*}
\begin{split}
&\alpha^{(k-2)}_n(1,q)=\left\{
\begin{array}{ll}
1, & \hbox{if $n=0$}, \\[5pt]
(-1)^nq^{\frac{2k+j-4}{2}n^2}(q^{-\frac{2k-2i+j}{2}n}+q^{\frac{2k-2i+j}{2}n}), & \hbox{if $n\geq1$,}
\end{array}
\right.\\[10pt]
&\beta^{(k-2)}_n(1,q)=\sum_{n\geq N_{3}\geq \cdots\geq N_{k-1}\geq 0}\frac{q^{N^2_{3}+\cdots+N^2_{k-1}+N_{i}+\cdots+N_{k-1}}}{(q;q)_{n-N_{3}}(q;q)_{N_{3}-N_{4}}\cdots(q^{2-j};q^{2-j})_{N_{k-1}}}.
\end{split}
\end{equation*}

Plugging  $(\alpha^{(k-2)}_n(1,q),\beta^{(k-2)}_n(1,q))$ into Lemma \ref{BL-2}, we get the Bailey pair $(\alpha^{(k-1)}_n(1,q),$\\$\beta^{(k-1)}_n(1,q))$, where
\begin{equation*}
\begin{split}
&\alpha^{(k-1)}_n(1,q)=\left\{
\begin{array}{ll}
1, & \hbox{if $n=0$}, \\[5pt]
(-1)^nq^{\frac{2k+j-3}{2}n^2}(q^{-\frac{2k-2i+j}{2}n}+q^{\frac{2k-2i+j}{2}n}), & \hbox{if $n\geq1$,}
\end{array}
\right.\\[10pt]
&\beta^{(k-1)}_n(1,q)=\sum_{n\geq N_{2}\geq \cdots\geq N_{k-1}\geq 0}\frac{q^{N_2^2 /2+N^2_{3}+\cdots+N^2_{k-1}+N_{i}+\cdots+N_{k-1}}(-q^{1/2};q)_{N_2}}
{(-q^{1/2};q)_n(q;q)_{n-N_{2}}(q;q)_{N_{2}-N_{3}}\cdots(q^{2-j};q^{2-j})_{N_{k-1}}}.
\end{split}
\end{equation*}
Then we apply Lemma \ref{BL-1} to
the Bailey pair $(\alpha^{(k-1)}_n(1,q),\beta^{(k-1)}_n(1,q))$. This gives
\begin{equation*}
\begin{split}
&\alpha^{(k)}_n(1,q)=\left\{
\begin{array}{ll}
1, & \hbox{if $n=0$}, \\[5pt]
(-1)^nq^{\frac{2k+j-1}{2}n^2}(q^{-\frac{2k-2i+j}{2}n}+q^{\frac{2k-2i+j}{2}n}), & \hbox{if $n\geq1$,}
\end{array}
\right.\\[10pt]
&\beta^{(k)}_n(1,q)=\sum_{n\geq N_{1}\geq \cdots\geq N_{k-1}\geq 0}\frac{q^{N_1^2+N_2^2 /2+N^2_{3}+\cdots+N^2_{k-1}+N_{i}+\cdots+N_{k-1}}(-q^{1/2};q)_{N_2}}
{(q;q)_{n-N_{1}}(q;q)_{N_{1}-N_{2}}\cdots(q^{2-j};q^{2-j})_{N_{k-1}}(-q^{1/2};q)_{N_1}}.
\end{split}
\end{equation*}

According to the definition of a Bailey pair, we have
\begin{equation*}
\beta^{(k)}_n(1,q)=\sum_{r=0}^n\frac{\alpha^{(k)}_r(1,q)}{(q;q)_{n-r}(q;q)_{n+r}}.
\end{equation*}
Letting $n\rightarrow \infty$ and multiplying both sides by $(-q^{1/2};q)_\infty(q;q)_\infty$, we obtain
\begin{align}\nonumber
&\sum_{ N_{1}\geq \cdots\geq N_{k-1}\geq 0}\frac{q^{N_1^2+N_2^2 /2+N^2_{3}+\cdots+N^2_{k-1}+N_{i}+\cdots+N_{k-1}}(-q^{1/2};q)_{N_2}(-q^{1/2};q)_\infty}
{(q;q)_{N_{1}-N_{2}}\cdots(q^{2-j};q^{2-j})_{N_{k-1}}(-q^{1/2};q)_{N_1}}\\ \nonumber
&\qquad=\frac{(-q^{1/2};q)_\infty}{(q;q)_\infty}\left(1+\sum_{n=1}^\infty
(-1)^nq^{\frac{2k+j-1}{2}n^2}(q^{-\frac{2k-2i+j}{2}n}+q^{\frac{2k-2i+j}{2}n})\right)\\ \label{vv1}
&\qquad=\frac{(-q^{1/2};q)_\infty(q^{i-1/2},q^{2k+j-i-1/2},q^{2k+j-1};q^{2k+j-1})_\infty}{(q;q)_\infty},
\end{align}
where the last equality follows from  Jacobi's triple product identity.

Observing that
$$\frac{(-q^{1/2};q)_\infty}{(q;q)_\infty}=\frac{(q;q^2)_\infty}{(q^{1/2};q^{1/2})_\infty},\ \   \frac{(-q^{1/2};q)_\infty}{(-q^{1/2};q)_{N_1}}=(-q^{(2N_1+1)/2};q)_\infty,$$
and
$$(-q^{1/2};q)_{N_2}=q^{N_2^2/2}(-q^{(1-2N_2)/2};q)_{N_2}, $$
by replacing $q$ by $q^2$ we derive  \eqref{Gollnitz-AA1}. This
completes the proof.
\end{proof}

\begin{rem} We point out some special cases of \eqref{Gollnitz-AA1}. For $k=i$, the Bailey pair
$(\alpha^{(k-i)}_n(1,q),\beta^{(k-i)}_n(1,q))$ is actually $(\alpha^{(0)}_n(1,q),\beta^{(0)}_n(1,q))$. Thus following the remainder process in the above proof, we find that  \eqref{Gollnitz-AA1} reduces to the following forms.
\begin{itemize}
 \item[]Case 1.  $k=i=2$ and $j=1$
 \begin{align*}
 &\sum_{N_1\geq 0}\frac{q^{2N_1^2}(-q^{1+2N_1};q^2)_\infty}{ (q^{2};q^{2})_{N_{1}}}
 =\frac{(q^2;q^4)_\infty(q^{3},q^{5},q^8;q^{8})_\infty}
 {(q;q)_\infty}.
 \end{align*}
 \item[] Case 2. $k=i\geq3$
 \begin{align*}
 &\sum_{N_1\geq \cdots\geq N_{k-1}\geq0}\frac{q^{2(N_1^2+\cdots+N_{k-1}^2)}(-q^{1-2N_2};q^2)_{N_2}(-q^{1+2N_1};q^2)_\infty}{(q^2;q^2)_{N_1-N_2}\cdots(q^2;q^2)_{N_{k-2}-N_{k-1}} (q^{4-2j};q^{4-2j})_{N_{k-1}}}\\[5pt]
 &=\frac{(q^2;q^4)_\infty(q^{2k-1},q^{2k-1+2j},q^{4k-2+2j};q^{4k-2+2j})_\infty}
 {(q;q)_\infty}.
 \end{align*}
\end{itemize}

In addition, for $k=i=2$ and $j=0$, the case excluded by Theorem \ref{Gollnitz-3-ee}, using the above process we can also derive the following identity : 
\begin{align*}
&\sum_{N_1\geq 0}\frac{q^{2N_1^2}(q;q^2)_{N_1}(-q^{1+2N_1};q^2)_\infty}{ (q^{4};q^{4})_{N_{1}}}
=\frac{(q^2;q^4)_\infty(q^{3},q^{3},q^6;q^{6})_\infty}
{(q;q)_\infty}.
\end{align*}

\end{rem}

\section{The G\"ollnitz-Gordon marking}

In this section, we introduce the notion of the G\"ollnitz-Gordon marking of a partition. Furthermore, we investigate some attributes of a partition. And at the end of this section we give  an explanation of Theorem \ref{main-thm} on the combinatorial side, which is stated in Theorems \ref{equiv-main-thm-2} and \ref{equiv-main-thm}.

\begin{defi}[G\"ollnitz-Gordon  marking]\label{Gollnitz-Gordon-marking}
The G\"ollnitz-Gordon  marking $GG(\lambda)$ of a partition $\lambda=(\lambda_1,\lambda_2,\ldots,\lambda_\ell)$ is an assignment of positive integers {\rm(}marks{\rm)} to the parts of  $\lambda$ from  smallest to  largest such that the marks are as small as possible subject to the condition that for $1\leq s\leq \ell$,
the integer assigned to $\lambda_s$ is different from the integers assigned to the parts $\lambda_g$ such that $g>s$ and $\lambda_s-\lambda_g\leq 2$ with strict inequality if $\lambda_s$ is odd.
\end{defi}
For example, the G\"ollnitz-Gordon marking of
\[\lambda=(18,15,14,14,12,11,10,8,8,6,4,3,2,1)\]
 is
\[GG(\lambda)=(18_1,15_3,14_2,14_1,12_3,11_2,10_1,8_3,8_2,6_1,4_3,3_1,2_2,1_1).\]
The G\"ollnitz-Gordon marking of a partition can be represented by an array, where the column indicates the size of a part and the row (counted from bottom to top) indicates the mark, so the G\"ollnitz-Gordon marking of $\lambda$ above would be:
\begin{equation}\label{example-new-1}
GG(\lambda)=\left[
\begin{array}{cccccccccccccccccccc}
 &&&4&&&&8&&&&12&&&15&\\
 &2&&&&&&8&&&11&&&14&&&&\\
{1}&&3&&&6&&&&10&&&&14&&&&18
\end{array}
\right].
\end{equation}
Let $N_r(\lambda)$ (or $N_r$ for short) denote the number of  parts in the $r$-th row of $GG(\lambda)$. From the definition of the G\"ollnitz-Gordon marking, it is not hard  to show that $N_1\geq N_2\geq\cdots$. By definition, we have $N_1(\lambda)=6$, $N_2(\lambda)=4$ and $N_3(\lambda)=4$. 

 For $k\geq i\geq1$, a partition $\lambda$ is counted by ${{C}}_1(k,i;n)$  if and only if  the marks  of $1$ and $2$  are not exceed to $i-1$, the marks of the odd parts are at most 1, and there are at most $k-1$ rows in $GG(\lambda)$.

For $k\geq i\geq1$ and $j=0$ or $1$, let $\mathbb{C}_j(k,i;n)$ denote the set of partitions counted by $C_j(k,i;n)$. Define
\[\mathbb{C}_j(k,i)=\bigcup_{n\geq0}\mathbb{C}_j(k,i;n).\]
We wish to give a criterion to determine whether a partition in $\mathbb{C}_1(k,i)$ also  belongs to $\mathbb{C}_0(k,i)$.
Let $\lambda=(\lambda_1,\lambda_2,\ldots, \lambda_\ell)$ be a partition in $\mathbb{C}_1(k,i)$. From the definitions of $\mathbb{C}_1(k,i)$ and $\mathbb{C}_0(k,i)$, we see that $\lambda$  belongs to $\mathbb{C}_0(k,i)$ if and only if $\lambda$ satisfies the item (4) in Theorem \ref{Gordon-Gollnitz-even}, that is, for $1\leq s\leq \ell-k+2$, if
\begin{equation}\label{diff-cond-oo}
\lambda_s\leq\lambda_{s+k-2}+2 \text{ with strict inequality if $\lambda_s$ is odd},
\end{equation}
then
\begin{equation}\label{cong-cond-o}
[\lambda_s/2]+\cdots+[\lambda_{s+k-2}/2]\equiv i-1+{O}_\lambda(\lambda_s)\pmod{2}.
\end{equation}
For the consecutive parts $\lambda_s,\lambda_{s+1}, \ldots,  \lambda_{s+k-2}$ of $\lambda$, if such $k-1$ parts satisfy the difference condition \eqref{diff-cond-oo}, then we say that $\{\lambda_{s+l}\}_{0\leq l\leq k-2}$ is a $(k-1)$-band of $\lambda$. 

Let $\{\lambda_{s+l}\}_{0\leq l\leq k-2}$ be a $(k-1)$-band of $\lambda$.  We say that the $(k-1)$-band $\{\lambda_{s+l}\}_{0\leq l\leq k-2}$ is  good if the $k-1$ parts in $\{\lambda_{s+l}\}_{0\leq l\leq k-2}$ also   satisfy the congruence condition \eqref{cong-cond-o}, otherwise, we say that  the $(k-1)$-band $\{\lambda_{s+l}\}_{0\leq l\leq k-2}$ is  bad. 

For a partition $\lambda\in\mathbb{C}_1(k,i)$, we see that $\lambda$ is also a partition in $\mathbb{C}_0(k,i)$ if and only if all the $(k-1)$-bands of $\lambda$ are good. 
Next we restrict our attention to certain special $(k-1)$-bands of $\lambda$ to determine whether $\lambda$ is a partition in $\mathbb{C}_0(k,i)$.  Assume that there are $N_{k-1}$ $(k-1)$-marked parts in $GG(\lambda)$, denoted   $\lambda^{(k-1)}_{1}> \lambda^{(k-1)}_{2}>\cdots >\lambda^{(k-1)}_{N_{k-1}}$. For $1\leq p\leq N_{k-1}$, assume that $\lambda^{(k-1)}_{p}$ is the $s_p$-th part $\lambda_{s_p}$ of $\lambda$, there must exist $k-2$ parts $\lambda_g$ in $\lambda$ such that $g>s_p$ and $\lambda_g\geq \lambda_{s_p}-2$ with strict inequality if $\lambda_{s_p}$ is odd. Then $\{\lambda_{s_p+l}\}_{0\leq l\leq k-2}$ is a $(k-1)$-band of $\lambda$. We say that  $\{\lambda_{s_p+l}\}_{0\leq l\leq k-2}$ is the $(k-1)$-band induced by $\lambda^{(k-1)}_{p}$. We also denote the $(k-1)$-band   $\{\lambda_{s_p+l}\}_{0\leq l\leq k-2}$ by $\{\lambda^{(k-1)}_{p}\}_{k-1}$:
 \[\{ \lambda^{(k-1)}_{p,1},\lambda^{(k-1)}_{p,2},\ldots,\lambda^{(k-1)}_{p,k-1}\},\]
where   $\lambda^{(k-1)}_{p,1}\geq \lambda^{(k-1)}_{p,2}\geq \cdots \geq \lambda^{(k-1)}_{p,k-1}$, that is, $\lambda^{(k-1)}_{p,1}=\lambda^{(k-1)}_{p}$ and $\lambda^{(k-1)}_{p}-\lambda^{(k-1)}_{p,k-1}\leq 2$ with strict inequality if  $\lambda^{(k-1)}_{p}$ is odd.

  The following theorem is the main result of this section, which gives a criterion  to  determine whether  a partition $\lambda$ in $\mathbb{C}_1(k,i)$ also belongs to $\mathbb{C}_0(k,i)$.

\begin{thm}\label{parity k-1 sequence} For $k\geq i\geq1$, let $\lambda$ be a partition in $\mathbb{C}_1(k,i)$ with $N_{k-1}$ $(k-1)$-marked parts in $GG(\lambda)$, denoted   $\lambda^{(k-1)}_{1}> \lambda^{(k-1)}_{2}>\cdots >\lambda^{(k-1)}_{N_{k-1}}$. Then $\lambda$  is a partition in $\mathbb{C}_0(k,i)$ if and only if $\{\lambda^{(k-1)}_{p}\}_{k-1}$ is   good   for $1\leq p\leq N_{k-1}$.
\end{thm}

To prove Theorem \ref{parity k-1 sequence}, we need the following lemma.

   \begin{lem}\label{parity k-1 sequence-over} For $k\geq i\geq 1$, let $\lambda=(\lambda_1,\lambda_2,\ldots, \lambda_\ell)$ be a partition in $\mathbb{C}_1(k,i)$, and let $\{\lambda_{s+l}\}_{0\leq l\leq k-2}$ and $\{\lambda_{g+l}\}_{0\leq l\leq k-2}$  be two $(k-1)$-bands of $\lambda$.  If $\lambda_s<\lambda_g$ and $\lambda_{s+k-2}\geq \lambda_g-4$ with strict inequality if $\lambda_g$ is odd, then $\{\lambda_{s+l}\}_{0\leq l\leq k-2}$ and $\{\lambda_{g+l}\}_{0\leq l\leq k-2}$ are both good or bad. 

\end{lem}

\pf Since $\lambda_s<\lambda_g$, then we have  $s>g$. We  assume that $s=g+t$ where $t\geq1$. Note that $\lambda$ is a partition in $\mathbb{C}_1(k,i)$, then  for $1\leq s\leq \ell-k+1$, we have
\begin{equation}\label{diff-k-1}
\lambda_s\geq\lambda_{s+k-1}+2 \text{ with strict inequality if $\lambda_s$ is even}.
\end{equation}
It follows  that there are at most $2k-2$ parts of $\lambda$ in $[\lambda_g-4,\lambda_g]$ (resp. $(\lambda_g-4,\lambda_g]$) if $\lambda_g$ is even (resp. odd). So $1\leq t\leq k-1$.
Since $(s+k-1-t)-g=k-1$, using \eqref{diff-k-1}, we have  \begin{equation}\label{diff-k-1-1}
\lambda_{g}\geq\lambda_{s+k-1-t}+2 \text{ with strict inequality if $\lambda_g$ is even}.
\end{equation}
Note that $\{\lambda_{g+l}\}_{0\leq l\leq k-2}$ is a $(k-1)$-band of $\lambda$, we have 
\begin{equation}\label{diff-k-1-2}
\lambda_{g}\leq\lambda_{g+k-2}+2 \text{ with strict inequality if $\lambda_g$ is odd}.
\end{equation}
Combining \eqref{diff-k-1-1} and \eqref{diff-k-1-2}, we arrive at $\lambda_{s+k-1-t}<\lambda_{s+k-2-t}$. With the similar argument, we derive that  $\lambda_{g+t}<\lambda_{g+t-1}$.  To summarize, the overlapping structure of $\{\lambda_{s+l}\}_{0\leq l\leq k-2}$ and $\{\lambda_{g+l}\}_{0\leq l\leq k-2}$ can be described as follows. There are two cases.

 If $1\leq t< k-1$, then  $\{\lambda_{s+l}\}_{0\leq l\leq k-2}$ and $\{\lambda_{g+l}\}_{0\leq l\leq k-2}$ have overlapping parts.
{\footnotesize\[\begin{array}{ccccccccccccccccccc}
\lambda_{s+k-2}&\leq &\cdots
&\leq& \lambda_{s+k-1-t}&<&\lambda_{s+k-2-t}&\leq&\cdots&\leq& \lambda_s\\[5pt]
&&&&&&\|&&&&\|&&&&\\[5pt]
&&&&&&\lambda_{g+k-2} &\leq& \cdots &\leq &\lambda_{g+t}&<&\lambda_{g+t-1}&\leq
&\cdots&\leq&\lambda_{g}
\end{array}
\]}

If $t=k-1$, then $\{\lambda_{s+l}\}_{0\leq l\leq k-2}$ and $\{\lambda_{g+l}\}_{0\leq l\leq k-2}$ have the following relation:
\[\lambda_{s+k-2}\leq\cdots\leq \lambda_s<\lambda_{g+k-2}\leq\cdots\leq \lambda_g.\]

Set
  \begin{equation}\label{cong-x-lem-over}
  [\lambda_{s+k-2}/2]+\cdots+[\lambda_{s}/2]\equiv a+{O}_\lambda(\lambda_{s})\pmod2,
  \end{equation}
   where $a=i-1$ or $i$. We aim to prove that for the same $a$ in \eqref{cong-x-lem-over}, we have
    \begin{equation}\label{cong-g-lem-over}
  [\lambda_{g+k-2}/2]+\cdots+[\lambda_{g}/2]\equiv a+{O}_\lambda(\lambda_{g})\pmod2.
   \end{equation}

   Note that  $\lambda_{g+l}=\lambda_{s+l-t}$ for $t\leq l\leq k-2$, so
      \begin{eqnarray*}
   && [\lambda_{g+k-2}/2]+\cdots+[\lambda_{g+t}/2]+[\lambda_{g+t-1}/2]+\cdots
   +[\lambda_{g}/2]\\[5pt]
   &&\qquad=
   [\lambda_{s+k-2-t}/2]+\cdots+[\lambda_s/2]+[\lambda_{g+t-1}/2]+\cdots
   +[\lambda_{g}/2]\\[5pt]
   &&\qquad=
   [\lambda_{s+k-2}/2]+\cdots+[\lambda_s/2]\\[5pt]
   &&\qquad\qquad+
  [\lambda_{g+t-1}/2]+\cdots
   +[\lambda_{g}/2]-
   \left([\lambda_{s+k-2}/2]+\cdots+[\lambda_{s+k-1-t}/2]\right).
   \end{eqnarray*}
 Combining with \eqref{cong-x-lem-over}, we find that in order to prove \eqref{cong-g-lem-over}, it suffices to show that
   \begin{equation} \label{cong-g-x-f-over}
   \begin{split}
  &[\lambda_{g+t-1}/2]+\cdots
   +[\lambda_{g}/2]-
   \left([\lambda_{s+k-2}/2]+\cdots+[\lambda_{s+k-1-t}/2]\right)\\[5pt]
   &\qquad \equiv O_\lambda(\lambda_g)-O_\lambda(\lambda_s)\pmod{2}.
   \end{split}
   \end{equation}
We consider the following two cases.

Case 1: $\lambda_g$ is even. We may write $\lambda_g=2b+4$. Under the condition that $\lambda_{s+k-2}\geq \lambda_g-4$ with strict inequality if $\lambda_g$ is odd, we have $\lambda_{s+k-2}\geq 2b$. Using \eqref{diff-k-1-2}, we have $\lambda_{g+k-2}\geq\lambda_g-2\geq 2b+2$. Furthermore, $\lambda_{s+k-1-t}<2b+2$ and $\lambda_{g+t-1}>2b+2$. Hence we conclude that
\[2b \leq \lambda_{s+k-2}\leq   \cdots \leq \lambda_{s+k-1-t}<2b+2,\] and
\[2b+2<\lambda_{g+t-1}\leq  \cdots \leq \lambda_{g}=2b+4.\]
 It yields that for $k-1-t\leq l\leq k-2$,
$[\lambda_{s+l}/2]=b$,  and for  $0\leq l\leq t-1$,
\begin{align*}
&[\lambda_{g+l}/2]=\left\{
\begin{array}{ll}
b+1, & \hbox{if $\lambda_{g+l}$ is odd,}\\[5pt]
b+2, & \hbox{otherwise.}
\end{array}
\right.
\end{align*}
It follows that
 \begin{equation*}
     \begin{split}
 &[\lambda_{g+t-1}/2]+\cdots
   +[\lambda_{g}/2]-
   \left([\lambda_{s+k-2}/2]+\cdots+[\lambda_{s+k-1-t}/2]\right)\nonumber\\[5pt]
   &\qquad\equiv \text{the number of $\lambda_{g+l}$ ($0\leq l\leq t-1 $) which are odd}\pmod{2}.
\end{split}
   \end{equation*}
   By the definitions of $O_\lambda(\lambda_g)$ and $O_\lambda(\lambda_s)$, we arrive at \eqref{cong-g-x-f-over}.

Case 2: $\lambda_g$ is odd. We may write $\lambda_g=2b+3$. In such case, we have $\lambda_{g+k-2}=\cdots=\lambda_{g-1}=2b+2$. Under the condition that $\lambda_{s+k-2}\geq \lambda_g-4$ with strict inequality if $\lambda_g$ is odd, we have  $\lambda_{s+k-2}\geq 2b$. Furthermore, we see that $\lambda_{s+k-3}=\cdots=\lambda_{s}=2b+2$. It follows that $t=1$, $\lambda_{s+k-2}=2b$ or $2b+1$, and $O_\lambda(\lambda_g)-O_\lambda(\lambda_s)=1$. Thus we have
 \begin{equation*}
 \begin{split}
   &[\lambda_{g+t-1}/2]+\cdots
   +[\lambda_{g}/2]-
   \left([\lambda_{s+k-2}/2]+\cdots+[\lambda_{s+k-1-t}/2]\right)\nonumber\\[5pt]
   &\qquad = [\lambda_{g}/2]-[\lambda_{s+k-2}/2]=1=O_\lambda(\lambda_g)-O_\lambda(\lambda_s)\pmod{2}.
   \end{split}
   \end{equation*}
We arrive at \eqref{cong-g-x-f-over}.  This completes the proof.    \qed

We are now in a position to give a proof of Theorem \ref{parity k-1 sequence}.

\begin{proof}[Proof of Theorem \ref{parity k-1 sequence}]
 If  $\lambda$ is a partition in $\mathbb{C}_0(k,i)$, then all the $(k-1)$-bands of $\lambda$ are good. Thus the $(k-1)$-band $\{\lambda^{(k-1)}_{p}\}_{k-1}$ induced by $\lambda^{(k-1)}_{p}$ is good for $1\leq p\leq N_{k-1}$.

On the other hand, we assume that $\{\lambda^{(k-1)}_{p}\}_{k-1}$ is good for $1\leq p\leq N_{k-1}$. To prove   $\lambda$ is a partition in $\mathbb{C}_0(k,i)$, it suffices to show  all the $(k-1)$-bands of $\lambda$ are good. Given a $(k-1)$-band $\{\lambda_{g+l}\}_{0\leq l\leq k-2}$ of $\lambda$, that is,  $\lambda_g\leq \lambda_{g+k-2}+2$ with strict inequality if $\lambda_g$ is odd. We aim to show  $\{\lambda_{g+l}\}_{0\leq l\leq k-2}$  is good.

 By the definition of the G\"ollnitz-Gordon marking, we see that the  parts in $\{\lambda_{g+l}\}_{0\leq l\leq k-2}$ have different marks in $GG(\lambda)$. Thus there is a part, say $\lambda_{g+l_g}$ ($0\leq l_g\leq k-2$), marked with $k-1$  in $GG(\lambda)$.   Assume that $\lambda_{g+l_g}$ is the $p_g$-th $(k-1)$-marked part $\lambda^{(k-1)}_{p_g}$ in $GG(\lambda)$, then we have $\lambda_g-\lambda^{(k-1)}_{p_g}\leq 2$ with strict inequality if $\lambda_g$ is odd. Let 
\[\{\lambda^{(k-1)}_{p_g,1}, \lambda^{(k-1)}_{p_g,2}, \ldots, \lambda^{(k-1)}_{p_g,k-1}\}\]
be the $(k-1)$-band $\{\lambda^{(k-1)}_{p_g}\}_{k-1}$ induced by $\lambda^{(k-1)}_{p_g}$, that is, $\lambda^{(k-1)}_{p_g,1}=\lambda^{(k-1)}_{p_g}$, and $\lambda^{(k-1)}_{p_g}-\lambda^{(k-1)}_{p_g,k-1}\leq 2$ with strict inequality if $\lambda^{(k-1)}_{p_g}$ is odd. Hence
 \[\lambda_g-\lambda^{(k-1)}_{p_g,k-1}=(\lambda_g-\lambda^{(k-1)}_{p_g})+(\lambda^{(k-1)}_{p_g}-\lambda^{(k-1)}_{p_g,k-1})\leq 4\]
   with strict inequality if $\lambda_{g}$ is odd.  By Lemma \ref{parity k-1 sequence-over}, we deduce that $\{\lambda_{g+l}\}_{0\leq l\leq k-2}$ and $\{\lambda^{(k-1)}_{p_g}\}_{k-1}$ are both good or bad  in $\lambda$. Under the assumption that $\{\lambda^{(k-1)}_{p}\}_{k-1}$ is good for $1\leq p\leq N_{k-1}$, we derive that   $\{\lambda_{g+l}\}_{0\leq l\leq k-2}$ is good. It follows that $\lambda$ is a partition in $\mathbb{C}_0(k,i)$.  This completes the proof.  
   \end{proof}

Now, we turn to give equivalent statements of Theorem \ref{main-thm}.

For $k\geq i\geq1$ and $j=0$ or $1$, let $\mathbb{E}_j(k,i)$ denote the number of partitions in $\mathbb{C}_j(k,i)$ without odd parts. For $N_1\geq N_2\geq \cdots\geq N_{k-1}\geq 0$, let $\mathbb{E}_j(N_1,\ldots,N_{k-1};i)$ denote the set of partitions $\lambda$ in $\mathbb{E}_j(k,i)$ with $N_r$ parts in the $r$-th row of $GG(\lambda)$ for $1\leq r\leq k-1$.  Kur\c{s}ung\"{o}z \cite{Kursungoz-2010} established the following identity based on the Gordon marking of a partition.
\begin{equation*}
\sum_{\lambda\in\mathbb{E}_j(N_1,\ldots,N_{k-1};i)}q^{\frac{|\lambda|}{2}}=\frac{q^{N_1^2+\cdots+N_{k-1}^2+N_i+\cdots+N_{k-1}}}{(q;q)_{N_1-N_2}\cdots(q;q)_{N_{k-2}-N_{k-1}} (q^{2-j};q^{2-j})_{N_{k-1}}}.
\end{equation*}
Then, we can see that
\begin{equation*}
\sum_{\lambda\in\mathbb{E}_j(N_1,\ldots,N_{k-1};i)}q^{|\lambda|}=\frac{q^{2(N_1^2+\cdots+N_{k-1}^2+N_i+\cdots+N_{k-1})}}{(q^2;q^2)_{N_1-N_2}\cdots(q^2;q^2)_{N_{k-2}-N_{k-1}} (q^{4-2j};q^{4-2j})_{N_{k-1}}}.
\end{equation*}
It follows that
\begin{equation*}
\begin{split}
&\sum_{\lambda\in\mathbb{E}_j(k,i)}x^{\ell(\lambda)}q^{|\lambda|}=\sum_{N_1\geq \cdots\geq N_{k-1}\geq0}x^{N_1+\cdots+N_{k-1}}\sum_{\lambda\in\mathbb{E}_j(N_1,\ldots,N_{k-1};i)}q^{|\lambda|}\\
&\qquad=\sum_{N_1\geq \cdots\geq N_{k-1}\geq0}\frac{q^{2(N_1^2+\cdots+N_{k-1}^2+N_i+\cdots+N_{k-1})}x^{N_1+\cdots+N_{k-1}}}{(q^2;q^2)_{N_1-N_2}\cdots(q^2;q^2)_{N_{k-2}-N_{k-1}} (q^{4-2j};q^{4-2j})_{N_{k-1}}}.
\end{split}
\end{equation*}

For $N\geq 0$, let $\mathbb{I}_{N}$ denote the set of partitions $\eta=(\eta_1,\eta_2,\ldots,\eta_\ell)$ with distinct odd parts greater than or equal to $2N+1$, that is, $\eta_1>\eta_2>\cdots>\eta_{\ell}\geq 2N+1$.  The generating functions for partitions in $\mathbb{I}_{N}$ is
\[\sum_{\eta\in\mathbb{I}_{N}}x^{\ell(\eta)}q^{|\eta|}=(1+xq^{2N+1})(1+xq^{2N+3})\cdots=(-xq^{2N+1};q^2)_\infty.\]

We define $\mathbb{F}_1(2,2)$ to be set of pairs $(\lambda,-,\eta)$  of partitions such that
\[\lambda\in\mathbb{E}_1(2,2)\text{ and }\eta\in\mathbb{I}_{N_1(\lambda)}.\]
Then for $k=i=2$ and $j=1$, Theorem \ref{main-thm} is equivalent to the following combinatorial statement.
\begin{thm}\label{equiv-main-thm-2}
There is a bijection $\Phi_{2,2}$ between $\mathbb{F}_1(2,2)$ and $\mathbb{C}_1(2,2)$. Moreover, for a pair $(\lambda,-,\eta)\in \mathbb{F}_1(2,2)$, we have $\pi=\Phi_{2,2}(\lambda,-,\eta)\in\mathbb{C}_1(2,2)$ such that
\[|\pi|=|\lambda|+|\eta| \text{ and } \ell(\pi)=\ell(\lambda)+\ell(\eta).\]
\end{thm}

For $N\geq 0$, let $\mathbb{O}_{N}$ denote the set of partitions $\tau=(\tau_1,\tau_2,\ldots,\tau_\ell)$ with distinct negative odd parts which lie in $[1-2N,-1]$, that is, $-1\geq\tau_{1}>\tau_2>\cdots>\tau_{\ell}\geq 1-2N$. Clearly, the generating functions for partitions in $\mathbb{O}_{N}$ is
\[\sum_{\tau\in\mathbb{O}_{N}}q^{|\tau|}=(1+q^{1-2N})(1+q^{3-2N})\cdots(1+q^{-1})=(-q^{1-2N};q^2)_{N}.\]

Then for $k\geq 3$, $k\geq i\geq2$ and $j=0$ or $1$, we define $\mathbb{F}_j(k,i)$ to be set of triplets $(\lambda,\tau,\eta)$ of partitions such that
\[\lambda\in\mathbb{E}_j(k,i), \tau\in\mathbb{O}_{N_2(\lambda)}\text{ and }\eta\in\mathbb{I}_{N_1(\lambda)}.\]
In this case, Theorem \ref{main-thm} is equivalent to the following combinatorial statement.
\begin{thm}\label{equiv-main-thm}
For $k\geq 3$, $k\geq i\geq2$ and $j=0$ or $1$, there is a bijection $\Phi_{k,i}$ between $\mathbb{F}_{j}(k,i)$ and $\mathbb{C}_j(k,i)$, namely, for a triplet $(\lambda,\tau,\eta)\in \mathbb{F}_j(k,i)$, we have $\pi=\Phi_{k,i}(\lambda,\tau,\eta)\in\mathbb{C}_j(k,i)$ such that
\[|\pi|=|\lambda|+|\tau|+|\eta|\text{ and }\ell(\pi)=\ell(\lambda)+\ell(\eta).\]
\end{thm}

\section{Proof of Theorem \ref{equiv-main-thm-2}}
The main objective of this section is to give a proof of Theorem \ref{equiv-main-thm-2}.
To this end, we construct the bijection $\theta_{(p)}$ and   the combination operation, where $\theta_{(p)}$ is essentially the same as that defined in Section 5 in \cite{he-ji-wang-zhao-2019}. Composing the combination operation with $\theta_{(p)}$ gives the insertion operation which is the main ingredient for the proof of Theorem \ref{equiv-main-thm-2}.

\subsection{The bijection  $\theta_{(p)}$}

In this subsection, we modify the results in   Section 5 in \cite{he-ji-wang-zhao-2019}, which have been  proved  for $j=1$.  We just show that the results in this subsection also hold for $j=0$. Throughout this subsection, we assume that $k\geq i\geq 1$, $j=0$ or $1$, and $N_1\geq\cdots\geq N_{k-1}\geq0$.

Let $\mathbb{C}_j(N_1,\ldots,N_{k-1};i)$ denote the set of partitions $\lambda$ in $\mathbb{C}_j(k,i)$ such that there are $N_r$  $(k-1)$-marked parts in $GG(\lambda)$ for $1\leq r\leq k-1$. Let $\lambda$ be a partition in $\mathbb{C}_j(N_1,\ldots,N_{k-1};i)$. 
We use $\lambda^{(1)}=(\lambda^{(1)}_1,\lambda^{(1)}_2,\ldots,\lambda^{(1)}_{N_1})$ to denote the $1$-marked parts in $GG(\lambda)$, where $\lambda^{(1)}_1>\lambda^{(1)}_2>\cdots>\lambda^{(1)}_{N_1}$. We now divide  the parts of  $\lambda^{(1)}$ into the following two types. For $1\leq s\leq N_1$, we say  $\lambda^{(1)}_s$ is an odd type if  $\lambda^{(1)}_s$ is an odd part, or there is an odd part $\lambda^{(1)}_s+1$ in $\lambda$, otherwise, we say   $\lambda^{(1)}_s$ is an  even type.

 For $1\leq p \leq N_1$, we prepare three subsets $\mathbb{C}_j(N_1,\ldots,N_{k-1};i|p)$, $\overline{\mathbb{C}}_j(N_1,\ldots,N_{k-1};i|p)$ and $\overrightarrow{\mathbb{C}}_j(N_1,\ldots,N_{k-1};i|p)$ of $\mathbb{C}_j(N_1,\ldots,N_{k-1};i)$, which are described by using the sub-partition $\lambda^{(1)}=(\lambda^{(1)}_1,\,\lambda^{(1)}_2,\ldots, \lambda^{(1)}_{N_1})$ of a partition $\lambda$ in $\mathbb{C}_j(N_1,\ldots,N_{k-1};i)$.

 $\lozenge$ Let $\mathbb{C}_j(N_1,\ldots,N_{k-1};i|p)$ be the set of partitions $\lambda$ in $\mathbb{C}_j(N_1,\ldots,N_{k-1};i)$, where
     $\lambda_p^{(1)}$  is odd type and $\lambda_s^{(1)}$  is even type   for $1\leq p\leq N_1$ and $ 1\leq s\leq p-1$.

$\lozenge$  Let $\mathbb{\overline{C}}_j(N_1,\ldots,N_{k-1};i|p)$ be the set of partitions $\lambda$ in $\mathbb{C}_j(N_1,\ldots,N_{k-1};i)$, where
     $\lambda_p^{(1)}$  is even type, $\lambda_{p-1}^{(1)}$  is odd type, and $\lambda_s^{(1)}$ is even type for $1< p\leq N_1$ and $1\leq s\leq p-2$.

$\lozenge$  Let $\overrightarrow{\mathbb{C}}_j(N_1,\ldots,N_{k-1};i|p)$ be the set of partitions $\lambda$ in $\mathbb{C}_j(N_1,\ldots,N_{k-1};i)$, where $\lambda_s^{(1)}$ is even type for $1\leq p\leq N_1$ and $1\leq s\leq p$.

We give a bijection $\Theta_p$ between $\mathbb{C}_j(N_1,\ldots,N_{k-1};i|p)$ and  $\overline{\mathbb{C}}_j(N_1,\ldots,N_{k-1};i|p)$ for $1<p\leq N_1$, a bijection $\Theta_1$ between $\mathbb{C}_j(N_1,\ldots,N_{k-1};i|1)$ and  $\overrightarrow{\mathbb{C}}_j(N_1,\ldots,N_{k-1};i|1)$, and a bijection $\Theta_{(p)}$ between $\mathbb{C}_j(N_1,\ldots,N_{k-1};i|p)$ and  $\overrightarrow{\mathbb{C}}_j(N_1,\ldots,N_{k-1};i|p)$ for $1\leq p\leq N_1$.

 \begin{lem}\label{N-2-1-lem}
For $1< p\leq N_1$, there is a bijection $\Theta_p$ between $\mathbb{   C}_j(N_1,\ldots,N_{k-1};i|p)$ and   $\mathbb{\overline{C}}_j(N_1,\ldots,N_{k-1};i|p)$.
Furthermore, for $\lambda \in \mathbb{{C}}_j(N_1,\ldots,N_{k-1};i|p)$, we have   $\mu=\Theta_p(\lambda)\in\mathbb{\overline{C}}_j(N_1,\ldots,N_{k-1};i|p)$ and  $|\mu|=|\lambda|+2$. 

 \end{lem}

To show Lemma \ref{N-2-1-lem}, for $1< p\leq N_1$, we divide    $\mathbb{C}_j(N_1,\ldots,N_{k-1};i|p)$ into two disjoint subsets  $\mathbb{C}^{(1)}_j(N_1,\ldots,N_{k-1};i|p)$ and $\mathbb{C}^{(2)}_j(N_1,\ldots,N_{k-1};i|p)$, and   divide
  $\mathbb{\overline{C}}_j(N_1,\ldots,N_{k-1};i|p)$ into two disjoint subsets  $\mathbb{\overline{C}}^{(1)}_j(N_1,\ldots,N_{k-1};i|p)$ and $\mathbb{\overline{C}}^{(2)}_j(N_1,\ldots,N_{k-1};i|p)$. We then build the bijection $\Theta_p$ consisting of two bijections $\Theta_{l,p}$ between  $\mathbb{C}^{(l)}_j(N_1,\ldots,N_{k-1};i|p)$  and  $\mathbb{\overline{C}}^{(l)}_j(N_1,\ldots,N_{k-1};i|p)$, where $1\leq l\leq 2$.

  For $1< p\leq N_1$, we divide $\mathbb{C}_j(N_1,\ldots,N_{k-1};i|p)$  into the  following two disjoint subsets  $\mathbb{C}^{(1)}_j(N_1,\ldots,N_{k-1};i|p)$ and $\mathbb{C}^{(2)}_j(N_1,\ldots,N_{k-1};i|p)$. For  $\lambda\in\mathbb{C}_j(N_1,\ldots,N_{k-1};i|p)$, we set $r$ to be the mark of the odd part $\lambda_p^{(1)}+1$ in $GG(\lambda)$ when $\lambda_p^{(1)}$ is an even part. It is easy to see that $r\geq 2$.

 \begin{itemize}
 \item[(1)] $\mathbb{C}^{(1)}_j(N_1,\ldots,N_{k-1};i|p)$ is the set of partitions $\lambda$ in $\mathbb{C}_j(N_1,\ldots,N_{k-1};i|p)$ such that there exists a $1$-marked even part  $\lambda_p^{(1)}+3$ in $GG(\lambda)$ when $\lambda_p^{(1)}$ is an odd part  or there exists a $r$-marked  even part $\lambda_p^{(1)}+4$ in $GG(\lambda)$  when $\lambda_p^{(1)}$ is an even part.

 \item[(2)] $\mathbb{C}^{(2)}_j(N_1,\ldots,N_{k-1};i|p)$ is the set of partitions $\lambda$  in $\mathbb{C}_j(N_1,\ldots,N_{k-1};i|p)$ such that  there does not exist a $1$-marked even part $\lambda_p^{(1)}+3$ in $GG(\lambda)$ when $\lambda_p^{(1)}$ is an odd part  or there does not exist a $r$-marked  even part  $\lambda_p^{(1)}+4$ in $GG(\lambda)$  when $\lambda_p^{(1)}$ is an even part.
 \end{itemize}

For $1< p\leq N_1$, we divide $\mathbb{\overline{C}}_j(N_1,\ldots,N_{k-1};i|p)$  into the following two disjoint subsets  $\mathbb{\overline{C}}^{(1)}_j(N_1,\ldots,N_{k-1};i|p)$ and
 $\mathbb{\overline{C}}^{(2)}_j(N_1,\ldots,N_{k-1};i|p)$. For $\mu\in\mathbb{\overline{C}}_j(N_1,\ldots,N_{k-1};i,p)$,
  if $\mu_{p-1}^{(1)}$ is an even part, then set $r'$ to be the mark of the odd part $\mu_{p-1}^{(1)}+1$ in $GG(\mu)$. It is easy to see that   $\mu_{p-1}^{(1)}\geq 5$ and $r'\geq 2$.

 \begin{itemize}
 \item[(1)] $\mathbb{\overline{C}}^{(1)}_j(N_1,\ldots,N_{k-1};i|p)$ is the set of partitions $\mu$ in $\mathbb{\overline{C}}_j(N_1,\ldots,N_{k-1};i|p)$ such that there exists a $1$-marked even part $\mu_{p-1}^{(1)}-3$ in $GG(\mu)$
       if $\mu_{p-1}^{(1)}$ is an odd part or there exists a $r'$-marked even part $\mu_{p-1}^{(1)}-2$ in $GG(\mu)$ if $\mu_{p-1}^{(1)}$ is an even part.

 \item[(2)] $\mathbb{\overline{C}}^{(2)}_j(N_1,\ldots,N_{k-1};i|p)$ is the set of partitions $\mu$ in $\mathbb{\overline{C}}_j(N_1,\ldots,N_{k-1};i|p)$ such that there does not exist a $1$-marked even part $\mu_{p-1}^{(1)}-3$ in $GG(\mu)$
       if $\mu_{p-1}^{(1)}$ is an odd part or there does not exist a $r'$-marked even part  $\mu_{p-1}^{(1)}-2$ in $GG(\mu)$ if $\mu_{p-1}^{(1)}$ is an even part.
 \end{itemize}

\begin{lem}\label{varphi1}
For $1< p\leq N_1$, there is a bijection $\Theta_{1,p}$ between  $\mathbb{C}^{(1)}_j(N_1,\ldots,N_{k-1};i|p)$ and   $\mathbb{\overline{C}}^{(1)}_j(N_1,\ldots,N_{k-1};i|p)$. Moreover, for $\lambda \in \mathbb{C}^{(1)}_j(N_1,\ldots,N_{k-1};i|p)$ and $\mu=\Theta_{1,p}(\lambda)\in\mathbb{\overline{C}}^{(1)}_j(N_1,\ldots,N_{k-1};i|
p)$, we have $|\mu|=|\lambda|+2$.
 \end{lem}

  \pf If $\lambda_p^{(1)}$ is an odd part, set $\lambda_p^{(1)}={2t+1}$, then there is a 1-marked $2t+4$ in $GG(\lambda)$, that is, $\lambda_{p-1}^{(1)}=2t+4$; if $\lambda_p^{(1)}$ is an even part, set $\lambda_p^{(1)}=2t$, then there is a $r$-marked ${2t+1}$ and a $r$-marked $2t+4$ in $GG(\lambda)$.
For $1< p\leq N_1$, define $\mu=\Theta_{1,p}(\lambda)$  as follows.

\begin{itemize}
\item If $\lambda_p^{(1)}={2t+1}$, then replace $\lambda_p^{(1)}={2t+1}$ by  $2t+2$ and replace  $\lambda_{p-1}^{(1)}=2t+4$ by ${2t+5}$.

 \item If $\lambda_p^{(1)}=2t$, then replace the $r$-marked ${2t+1}$ by  $2t+2$ and replace the $r$-marked $2t+4$  by ${2t+5}$.

\end{itemize}

Let $\lambda^{(k-1)}=(\lambda^{(k-1)}_1,\lambda^{(k-1)}_2,\ldots, \lambda^{(k-1)}_{N_{k-1}})$ and $\mu^{(k-1)}=(\mu^{(k-1)}_1,\mu^{(k-1)}_2,\ldots, \mu^{(k-1)}_{N_{k-1}})$ be the $(k-1)$-marked parts in $GG(\lambda)$ and $GG(\mu)$, respectively. We just need to show that  if   $\lambda$ is a partition in $\mathbb{C}^{(1)}_0(N_1,\ldots,N_{k-1};i|p)$, then $\{\mu^{(k-1)}_s\}_{k-1}$ is   good for $1\leq s\leq N_{k-1}$. Assume that  $\lambda$ is a partition in $\mathbb{C}^{(1)}_0(N_1,\ldots,N_{k-1};i,p)$, then by Theorem \ref{parity k-1 sequence} we have $\{\lambda^{(k-1)}_s\}_{k-1}$ is   good for $1\leq s\leq N_{k-1}$.

For each $s\in[1,N_{k-1}]$, by the construction of $\mu$, we see that $2t+1$ and $2t+3$ do not occur in $\mu$, and so $\mu^{(k-1)}_s\neq 2t+1$, $2t+3$.  To show that $\{\mu^{(k-1)}_s\}_{k-1}$ is   good, we consider the following five cases.

Case 1: $\mu^{(k-1)}_s\leq 2t$ or $\mu^{(k-1)}_s\geq 2t+7$. By the construction of $\mu$, we see that $\{\mu^{(k-1)}_s\}_{k-1}=\{\lambda^{(k-1)}_s\}_{k-1}$ and $O_{\mu}(\mu^{(k-1)}_s)=O_{\lambda}(\lambda^{(k-1)}_s)$. Note that $\{\lambda^{(k-1)}_s\}_{k-1}$ is   good, so we get that  $\{\mu^{(k-1)}_s\}_{k-1}$ is good.

Case 2: $\mu^{(k-1)}_s=2t+2$. By the construction of $\mu$, we see that $f_{2t}(\mu)=f_{2t}(\lambda)$, $f_{2t+1}(\mu)=f_{2t+1}(\lambda)-1$, $f_{2t+2}(\mu)=f_{2t+2}(\lambda)+1$ and  $O_{\mu}(2t+2)=O_{\lambda}(2t+2)-1$. We have
\[
\begin{split}
&[\mu^{(k-1)}_{s,1}/2]+[\mu^{(k-1)}_{s,2}/2]+\cdots+[\mu^{(k-1)}_{s,k-1}/2]\\
&\qquad =tf_{2t}(\mu)+tf_{2t+1}(\mu)+(t+1)f_{2t+2}(\mu)\\
&\qquad =tf_{2t}(\lambda)+t(f_{2t+1}(\lambda)-1)+(t+1)(f_{2t+2}(\lambda)+1)\\
&\qquad =tf_{2t}(\lambda)+tf_{2t+1}(\lambda)+(t+1)f_{2t+2}(\lambda)+1\\
&\qquad =[\lambda^{(k-1)}_{s,1}/2]+[\lambda^{(k-1)}_{s,2}/2]+\cdots+[\lambda^{(k-1)}_{s,k-1}/2]+1\\
&\qquad \equiv i-1+O_\lambda(\lambda^{(k-1)}_{s,1})+1=i-1+O_\lambda(2t+2)+1\\
&\qquad \equiv i-1+O_\mu(2t+2)=i-1+O_\mu(\mu^{(k-1)}_{s,1})\pmod2.
\end{split}
\]
So, $\{\mu^{(k-1)}_s\}_{k-1}$ is  good.

Case 3: $\mu^{(k-1)}_s=2t+4$. In this case, we have $\mu^{(k-1)}_s=\lambda^{(k-1)}_s=2t+4$, $f_{2t+2}(\mu)=f_{2t+2}(\lambda)+1$, $f_{2t+3}(\mu)=f_{2t+3}(\lambda)=0$, $f_{2t+4}(\mu)=f_{2t+4}(\lambda)-1$ and  $O_{\mu}(2t+4)=O_{\lambda}(2t+4)-1$. So, we obtain that 
\[
\begin{split}
&[\mu^{(k-1)}_{s,1}/2]+[\mu^{(k-1)}_{s,2}/2]+\cdots+[\mu^{(k-1)}_{s,k-1}/2]\\
&\qquad =(t+1)f_{2t+2}(\mu)+(t+2)f_{2t+4}(\mu)\\
&\qquad =(t+1)(f_{2t+2}(\lambda)+1)+(t+2)(f_{2t+4}(\lambda)-1)\\
&\qquad =(t+1)f_{2t+2}(\lambda)+(t+2)f_{2t+4}(\lambda)-1\\
&\qquad =[\lambda^{(k-1)}_{s,1}/2]+[\lambda^{(k-1)}_{s,2}/2]+\cdots+[\lambda^{(k-1)}_{s,k-1}/2]-1\\
&\qquad \equiv i-1+O_\lambda(\lambda^{(k-1)}_{s,1})-1=i-1+O_\lambda(2t+4)-1\\
&\qquad =i-1+O_\mu(2t+4)=i-1+O_\mu(\mu^{(k-1)}_{s,1})\pmod2.
\end{split}
\]
It follows that $\{\mu^{(k-1)}_s\}_{k-1}$ is  good.

Case 4: $\mu^{(k-1)}_s=2t+5$. By the definition of G\"ollnitz-Gordon marking, we see that $\mu^{(k-1)}_{s,2}=\cdots=\mu^{(k-1)}_{s,k-1}=2t+4$. Moreover, by the construction of $\mu$, we have $\lambda^{(k-1)}_{s,1}=\lambda^{(k-1)}_{s,2}=\cdots=\lambda^{(k-1)}_{s,k-1}=2t+4$ and $O_\mu(2t+5)=O_\lambda(2t+4)$. So
\[
\begin{split}
&[\mu^{(k-1)}_{s,1}/2]+[\mu^{(k-1)}_{s,2}/2]+\cdots+[\mu^{(k-1)}_{s,k-1}/2]\\
&\qquad =(t+2)f_{2t+4}(\mu)+(t+2)f_{2t+5}(\mu)=(t+2)(k-2)+(t+2)\\
&\qquad =(t+2)(k-1)=(t+2)f_{2t+4}(\lambda)\\
&\qquad =[\lambda^{(k-1)}_{s,1}/2]+[\lambda^{(k-1)}_{s,2}/2]+\cdots+[\lambda^{(k-1)}_{s,k-1}/2]\\
&\qquad \equiv i-1+O_\lambda(\lambda^{(k-1)}_{s,1})=i-1+O_\lambda(2t+4)\\
&\qquad =i-1+O_\mu(2t+5)=i-1+O_\mu(\mu^{(k-1)}_{s,1})\pmod2,
\end{split}
\]
which implies that  $\{\mu^{(k-1)}_s\}_{k-1}$ is   good.

Case 5: $\mu^{(k-1)}_s=2t+6$. By the construction of $\mu$, we see that $\mu^{(k-1)}_s=\lambda^{(k-1)}_s=2t+6$, $f_{2t+4}(\mu)=f_{2t+4}(\lambda)-1$, $f_{2t+5}(\mu)=f_{2t+5}(\lambda)+1$, $f_{2t+6}(\mu)=f_{2t+6}(\lambda)$ and  $O_{\mu}(2t+6)=O_{\lambda}(2t+6)$. We have
\[
\begin{split}
&[\mu^{(k-1)}_{s,1}/2]+[\mu^{(k-1)}_{s,2}/2]+\cdots+[\mu^{(k-1)}_{s,k-1}/2]\\
&\qquad =(t+2)f_{2t+4}(\mu)+(t+2)f_{2t+5}(\mu)+(t+3)f_{2t+6}(\mu)\\
&\qquad =(t+2)(f_{2t+4}(\lambda)-1)+(t+2)(f_{2t+5}(\lambda)+1)+(t+3)f_{2t+6}(\lambda)\\
&\qquad =(t+2)f_{2t+4}(\lambda)+(t+2)f_{2t+5}(\lambda)+(t+3)f_{2t+6}(\lambda)\\
&\qquad =[\lambda^{(k-1)}_{s,1}/2]+[\lambda^{(k-1)}_{s,2}/2]+\cdots+[\lambda^{(k-1)}_{s,k-1}/2]\\
&\qquad \equiv i-1+O_\lambda(\lambda^{(k-1)}_{s,1})=i-1+O_\lambda(2t+6)\\
&\qquad = i-1+O_\mu(2t+6)=i-1+O_\mu(\mu^{(k-1)}_{s,1})\pmod2.
\end{split}
\]
Therefore, $\{\mu^{(k-1)}_s\}_{k-1}$ is good. Hence, we conclude that  $\{\mu^{(k-1)}_s\}_{k-1}$ is  good.  

  To prove that $\Theta_{1,p}$ is a bijection, we construct the inverse map $\Lambda_{1,p}$ of $\Theta_{1,p}$. If $\mu_{p-1}^{(1)}$ is an odd part, set $\mu_{p-1}^{(1)}={2t+5}$, then there is a $1$-marked $2t+2$ in $GG(\mu)$. If $\mu_{p-1}^{(1)}$ is an even part, set $\mu_{p-1}^{(1)}=2t+4$,  then there are $r'$-marked $2t+2$ and $r'$-marked ${2t+5}$ in $GG(\mu)$. For $1< p\leq N_1$, define $\lambda=\Lambda_{1,p}(\mu)$ as follows.
 \begin{itemize}
 \item If  $\mu_{p-1}^{(1)}={2t+5}$, then  replace $\mu_{p}^{(1)}=2t+2$ by ${2t+1}$ and replace
      $\mu_{p-1}^{(1)}={2t+5}$ by $2t+4$.

\item If $\mu_{p-1}^{(1)}=2t+4$, then replace the $r'$-marked $2t+2$  by ${2t+1}$ and replace the  $r'$-marked   ${2t+5}$ by $2t+4$.

 \end{itemize}
It can be verified that $\Lambda_{1,p}$ is the inverse map of $\Theta_{1,p}$. This completes the proof.  \qed

\begin{lem}\label{varphi2}
For $1< p\leq N_1$, there is a bijection $\Theta_{2,p}$ between   $\mathbb{C}^{(2)}_j(N_1,\ldots,N_{k-1};i|p)$ and  $\mathbb{\overline{C}}^{(2)}_j(N_1,\ldots,N_{k-1};i|p)$.
Moreover, for $\lambda \in \mathbb{C}^{(2)}_j(N_1,\ldots,N_{k-1};i|p)$ and $\mu=\Theta_{2,p}(\lambda)\in\mathbb{\overline{C}}^{(2)}_j(N_1,\ldots,N_{k-1};i|p)$, we have $|\mu|=|\lambda|+2$.
 \end{lem}

 \pf If $\lambda_p^{(1)}$ is an odd part, set $r=1$ and $\lambda_p^{(1)}={2t+1}$, then there does not exist a $1$-marked $2t+4$ in $GG(\lambda)$, that is $\lambda_{p-1}^{(1)}\geq 2t+6$. If $\lambda_p^{(1)}$ is an even part, set $\lambda_p^{(1)}=2t$, then there is a $r$-marked ${2t+1}$  and  there does not exist a $r$-marked $2t+4$ in $GG(\lambda)$, where $r\geq 2$. Let $\lambda_{p-1}^{(1)}=2b+2$ and let $l$ be the largest mark of the parts $2b+2$ in $GG(\lambda)$. If $\lambda_p^{(1)}={2t+1}$, then $b\geq t+2$. If $\lambda_p^{(1)}=2t$, then $b\geq t+1$. For $1< p\leq N_1$, define $\mu=\Theta_{2,p}(\lambda)$ as follows.
\begin{itemize}
\item If $r=1$, then replace $\lambda_p^{(1)}={2t+1}$ by  $2t+2$ and replace the $l$-marked $2b+2$ by ${2b+3}$.

\item If $r\geq 2$, then replace the $r$-marked ${2t+1}$ by  $2t+2$ and replace the $l$-marked $2b+2$ by ${2b+3}$.

\end{itemize}

Let $\lambda^{(k-1)}=(\lambda^{(k-1)}_1,\lambda^{(k-1)}_2,\ldots, \lambda^{(k-1)}_{N_{k-1}})$ and $\mu^{(k-1)}=(\mu^{(k-1)}_1,\mu^{(k-1)}_2,\ldots, \mu^{(k-1)}_{N_{k-1}})$ be the $(k-1)$-marked parts in $GG(\lambda)$ and $GG(\mu)$, respectively.  We just need to show that  if   $\lambda$ is a partition in $\mathbb{C}^{(2)}_0(N_1,\ldots,N_{k-1};i|p)$, then $\{\mu^{(k-1)}_s\}_{k-1}$ is   good for $1\leq s\leq N_{k-1}$. Assume that $\lambda$ is a partition in $\mathbb{C}^{(2)}_0(N_1,\ldots,N_{k-1};i,p)$, then by Theorem \ref{parity k-1 sequence} it follows that $\{\lambda^{(k-1)}_s\}_{k-1}$ is good for $1\leq s\leq N_{k-1}$.

For each $s\in[1,N_{k-1}]$,
by the construction of $\mu$, we see that $2t+1$ and $2t+3$ do not occur in $\mu$, and so $\mu^{(k-1)}_s\neq 2t+1$, $2t+3$. Then, we proceed to show that $\mu^{(k-1)}_s\neq 2t+4$. If $\lambda_p^{(1)}={2t+1}$, then there is no $1$-marked $2t+2$, $2t+3$ and $2t+4$ in $GG(\lambda)$. If $\lambda_p^{(1)}=2t$, then there is no $r$-marked $2t+2$, $2t+3$ and $2t+4$ in $GG(\lambda)$. By definition of G\"ollnitz-Gordon marking, we see that there is no $(k-1)$-marked $2t+4$ in $GG(\lambda)$. By the construction of $\mu$, we see that there is no $(k-1)$-marked $2t+4$ in $GG(\mu)$, and so $\mu^{(k-1)}_s\neq 2t+4$.

Now, we turn to show that if $\mu^{(k-1)}_s\in[2t+2,2b+2]$, then $\mu^{(k-1)}_s=2t+2$. If $b=t+1$, then recall that $\mu^{(k-1)}_s\neq 2t+3$, $2t+4$, we have $\mu^{(k-1)}_s=2t+2$. If $b\geq t+2$, then by the construction of $\mu$, we have $\mu_p^{(1)}\leq {2t+2}$, $\mu_{p-1}^{(1)}\geq 2b+2$ and $\mu^{(k-1)}_s\neq 2b+2$, which implies that $\mu^{(k-1)}_s\not\in [2t+5,2b+2]$. Again by $\mu^{(k-1)}_s\neq 2t+3$ and $2t+4$, we have   $\mu^{(k-1)}_s=2t+2$. Hence we obtain that $\mu^{(k-1)}_s=2t+2$ if $\mu^{(k-1)}_s\in[2t+2,2b+2]$. Then, with the similar argument in the proof of Lemma \ref{varphi1}, we can obtain that $\{\mu^{(k-1)}_s\}_{k-1}$ is good.

 To prove that $\Theta_{2,p}$ is a bijection, we construct the inverse map $\Lambda_{2,p}$ of
$\Theta_{2,p}$. If $\mu_{p-1}^{(1)}$ is an odd part, then set $l'=1$, $\mu_{p-1}^{(1)}={2b+3}$ and $\mu_{p}^{(1)}=2t$, where $b\geq t+1$. If $\mu_{p-1}^{(1)}$ is an even part, set $\mu_{p-1}^{(1)}=2b+2$ and $\mu_{p}^{(1)}=2t$ where $b\geq t+1$, then there is a $l'$-marked ${2b+3}$ and there is no $l'$-marked $2b$ in $GG(\mu)$, where $l'\geq 2$. If there exist parts $2t+2$ in $GG(\mu)$, then set $r'$ be the smallest mark of the parts $2t+2$ in $GG(\mu)$, where $r'\geq2$, otherwise we set $r'=1$.
For $1< p\leq N_1$, define $\lambda=\Lambda_{2,p}(\mu)$ as follows.

 \begin{itemize}
 \item If $r'=1$, then replace $\mu_{p}^{(1)}=2t$ by ${2t-1}$ and replace the $l'$-marked ${2b+3}$ by $2b+2$.

\item If $r'\geq 2$, then replace the $r'$-marked $2t+2$ by ${2t+1}$ and replace the $l'$-marked ${2b+3}$ by $2b+2$.

 \end{itemize}
It can be verified  that $\Lambda_{2,p}$ is the inverse map of $\Theta_{2,p}$. This completes the proof.  \qed

\noindent{\bf Proof of Lemma \ref{N-2-1-lem}.} For $1< p\leq N_1$ and $\lambda\in\mathbb{{C}}_j(N_1,\ldots,N_{k-1};i|p)$, define
\[\mu=\Theta_p(\lambda)=\left\{\begin{array}{cc}\Theta_{1,p}(\lambda),&\text{ if }\lambda\in\mathbb{C}^{(1)}_j(N_1,\ldots,N_{k-1};i|p);\\[5pt]
\Theta_{2,p}(\lambda),&\text{ if }\lambda\in\mathbb{C}^{(2)}_j(N_1,\ldots,N_{k-1};i|p).
\end{array}\right.\]
Combining Lemmas \ref{varphi1} and   \ref{varphi2}, we complete the proof of Lemma \ref{N-2-1-lem}. \qed

 \begin{lem}\label{N-2-1-lem-b}
 There is a bijection $\Theta_{1}$ between  $\mathbb{C}_j(N_1,\ldots,N_{k-1};i|1)$ and  $\overrightarrow{\mathbb{C}}_j(N_1,\ldots,N_{k-1};i\break$$|1)$. Moreover, for $\lambda \in \mathbb{{C}}_j(N_1,\ldots,N_{k-1};i|1)$ and $\mu=\Theta_{1}(\lambda)\in\overrightarrow{\mathbb{C}}_j(N_1,\ldots,N_{k-1};i|1)$, we have $|\mu|=|\lambda|+1$.
 \end{lem}

 \pf By definition, we see that $\lambda_{1}^{(1)}$ is odd type. Let $r$ be the mark of the odd part $\lambda_{1}^{(1)}$ (resp. $\lambda_{1}^{(1)}$+1) if $\lambda_{1}^{(1)}$ is odd  (resp. even) in $GG(\lambda)$.  More precisely,  if $\lambda_{1}^{(1)}$ is an odd part,  set $\lambda_{1}^{(1)}={2t+1}$, then $r=1$. If $\lambda_{1}^{(1)}$ is an even part, set $\lambda_{1}^{(1)}=2t$, then there is $r$-marked ${2t+1}$ in $GG(\lambda)$. Define $\mu=\Theta_{1}(\lambda)$ by replacing the $r$-marked ${2t+1}$  by    $2t+2$.

Let $\lambda^{(k-1)}=(\lambda^{(k-1)}_1,\lambda^{(k-1)}_2,\ldots, \lambda^{(k-1)}_{N_{k-1}})$ and $\mu^{(k-1)}=(\mu^{(k-1)}_1,\mu^{(k-1)}_2,\ldots, \mu^{(k-1)}_{N_{k-1}})$ be the $(k-1)$-marked parts in $GG(\lambda)$ and $GG(\mu)$, respectively. We just need to show that  if   $\lambda$ is a partition in $\mathbb{C}_0(N_1,\ldots,N_{k-1};i|1)$, then $\{\mu^{(k-1)}_s\}_{k-1}$ is   good for $1\leq s\leq N_{k-1}$. Assume that $\lambda$ is a partition in $\mathbb{C}_0(N_1,\ldots,N_{k-1};i|1)$, then $\{\lambda^{(k-1)}_s\}_{k-1}$ is  good  for $1\leq s\leq N_{k-1}$.

For each $s\in[1,N_{k-1}]$, by the construction of $\mu$, we have $\mu^{(k-1)}_s\leq 2t+2$ and $\mu^{(k-1)}_s\neq 2t+1$.  To show that $\{\mu^{(k-1)}_s\}_{k-1}$ is   good, we consider the following two cases.

Case 1:  $\mu^{(k-1)}_s\leq 2t$. By the construction of $\mu$, we have $\{\mu^{(k-1)}_s\}_{k-1}=\{\lambda^{(k-1)}_s\}_{k-1}$ and $O_{\mu}(\mu^{(k-1)}_s)=O_{\lambda}(\lambda^{(k-1)}_s)$. Note that $\{\lambda^{(k-1)}_s\}_{k-1}$ is good, so   $\{\mu^{(k-1)}_s\}_{k-1}$ is also  good.

Case 2:   $\mu^{(k-1)}_s=2t+2$. With the similar argument of Case 2 in the proof of Lemma \ref{varphi1}, we can obtain that $\{\mu^{(k-1)}_s\}_{k-1}$ is good.

In either case, we obtain that $\{\mu^{(k-1)}_s\}_{k-1}$ is  good. To prove that $\Theta_{1}$ is a bijection, we construct the inverse map $\Lambda_{1}$ of
$\Theta_{1}$. By definition,  $\mu_{1}^{(1)}$ is  even type, set $\mu_{1}^{(1)}=2t$. Let $r'$ be the smallest mark of the largest even parts in $GG(\mu)$, that is, if there exist even parts $2t+2$ in $\mu$, then $r'\geq 2$; otherwise,  $r'=1$.  Define $\lambda=\Lambda_{1}(\mu)$ as follows.

\begin{itemize}
 \item  If $r'=1$, then replace $\mu_{1}^{(1)}=2t$ by ${2t-1}$.

\item If $r'\geq 2$, then replace the $r'$-marked $2t+2$ by ${2t+1}$.

 \end{itemize}
It can be verified  that $\Lambda_{1}$ is the inverse map of $\Theta_{1}$. This completes the proof.  \qed

Iterating   the bijections in Lemma \ref{N-2-1-lem} $p-1$ times as well as the bijection $\Theta_{1}$ in Lemma \ref{N-2-1-lem-b}, we  get a bijection $\Theta_{(p)}$ between   $\mathbb{C}_j(N_1,\ldots,N_{k-1};i|p)$ and   $\overrightarrow{\mathbb{C}}_j(N_1,\ldots,N_{k-1};i|p)$.

 \begin{lem}\label{N-2-2-lem}
For $1\leq p\leq N_1$, there is a bijection $\Theta_{(p)}$ between $\mathbb{C}_j(N_1,\ldots,N_{k-1};i|p)$ and   $\overrightarrow{\mathbb{C}}_j(N_1,\ldots,N_{k-1};i|p)$. Moreover, for  $\lambda \in \mathbb{{C}}_j(N_1,\ldots,N_{k-1};i|p)$ and $\mu=\Theta_{(p)}(\lambda)\in\overrightarrow{\mathbb{C}}_j(N_1,\ldots,N_{k-1};i|p)$, we have $|\mu|=|\lambda|+2p-1$.
 \end{lem}

  \pf Define $\Theta_{(p)}=\Theta_{1}\Theta_{2}\cdots \Theta_p$, then by Lemmas \ref{N-2-1-lem} and \ref{N-2-1-lem-b}, we complete the proof. \qed

  \subsection{The combination $C_{p,t}$ and the division $D_{p,t}$}

 In this subsection, we assume that $j=0$ or $1$, $(2k+j)/2>2$, $k\geq i\geq 2$, $N_1\geq\cdots\geq N_{k-1}\geq0$, $t\geq 0$ and $0\leq p\leq N_1$. Let $\lambda$ be a partition in $\mathbb{C}_j(N_1,\ldots,N_{k-1};i)$. For convenience, we set  $\lambda^{(1)}_{0}=+\infty$ and $\lambda^{(1)}_{N_{1}+1}=-\infty$.  We will prepare two sets and then establish a bijection between these two subsets.
 
 Let $\mathbb{C}^{<}_j(N_1,\ldots,N_{k-1};i|p,t)$ denote the set of   partitions $\lambda$ in $\mathbb{C}_j(N_1,\ldots,N_{k-1};i)$ satisfying the following conditions: \begin{itemize}
\item[\rm {(1)}] $\lambda^{(1)}_{p+1}+2\leq 2t+1<\lambda^{(1)}_p$;
\item[\rm {(2)}] if $p\geq 1$, then  $\lambda\in\overrightarrow{\mathbb{C}}_j(N_1,\ldots,N_{k-1};i|p)$.
\end{itemize}

   It is worth mentioning  that if  $\mathbb{C}^{<}_j(N_1,\ldots,N_{k-1};i|p,t)$ is non-empty, then $p+t\geq N_1$. If $p=N_1$, then it is obviously right. If $p<N_1$, then by the definition of G\"ollnitz-Gordon marking, we see that
   \[2t+1\geq \lambda^{(1)}_{p+1}+2\geq \cdots\geq \lambda^{(1)}_{N_1}+2(N_1-p)\geq 2(N_1-p)+1,\]
   which leads to $p+t\geq N_1$. Hence, $p+t\geq N_1$.

 Let $\mathbb{C}^{\lhd}_j(N_1+1,\ldots,N_{k-1};i|p,t)$ denote the set of   partitions $\mu$ in $\mathbb{C}_j(N_1+1,\ldots,N_{k-1};i)$ satisfying the following conditions:
 \begin{itemize}
\item[\rm {(1)}] $\mu^{(1)}_{p+1}= 2t+1$;

\item[\rm {(2)}] if  $p\geq 1$, then $\mu\in\mathbb{C}_j(N_1+1,\ldots,N_{k-1};i|p)$;

\item[\rm {(3)}]  if there exist parts $2t+2$ in $\mu$, then $p\geq 1$ and $\mu^{(1)}_{p}= 2t+3$.
 \end{itemize}
 
We proceed to show that there is a bijection   between $\mathbb{C}^{<}_j(N_1,\ldots,N_{k-1};i|p,t)$ and $\mathbb{C}^{\lhd}_j(N_1+1,\ldots,N_{k-1};i|p,t)$, from which we find that if  $\mathbb{C}^{\lhd}_j(N_1+1,\ldots,N_{k-1};i|p,t)$ is non-empty, then  we have $p+t\geq N_1$.

   \begin{lem}\label{deltagammathmbb} For $1\leq p\leq N_1$,  there is a bijection, the combination $C_{p,t}$, between   $\mathbb{C}^{<}_j(N_1,\ldots,N_{k-1};i|p,t)$ and  $\mathbb{C}^{\lhd}_j(N_1+1,\ldots,N_{k-1};i|p,t)$. Moreover, for  $\lambda \in \mathbb{C}^{<}_j(N_1,\ldots,\break N_{k-1};i|p,t)$ and $\mu=C_{p,t}(\lambda)\in\mathbb{C}^{\lhd}_j(N_1+1,\ldots,N_{k-1};i|p,t)$, we have $|\mu|=|\lambda|+2t+2$.
 \end{lem}

\pf For $k\geq i\geq 2$, let $\lambda$ be a partition in $\mathbb{C}^{<}_1(N_1,\ldots,N_{k-1};i|p,t)$. Note that $p\geq 1$, then we have  $\lambda\in\overrightarrow{\mathbb{C}}_j(N_1,\ldots,N_{k-1};i|p)$, which implies that  $\lambda^{(1)}_p$ is even type, and so  $\lambda^{(1)}_p$ is an even part. We may write $\lambda^{(1)}_p=2b+2$.

If $b=t$, namely $\lambda^{(1)}_p=2t+2$, then set  $r=1$. If $b\geq t+1$, namely $\lambda^{(1)}_p\geq 2t+4$, then set $r$ be the largest mark of parts $2b+2$ in $GG(\lambda)$. Define the combination $\mu=C_{p,t}(\lambda)$ by adding  $2t+1$ as a new part to $\lambda$ and replacing the $r$-marked $2b+2$ in $GG(\lambda)$ by $2b+3$ to obtain $\mu$.

Obviously, $|\mu|=|\lambda|+2t+2$. To show that $\mu\in\mathbb{C}^{\lhd}_1(N_1+1,\ldots,N_{k-1};i|p,t)$, we first prove that  $\mu\in\mathbb{C}_1(N_1+1,\ldots,N_{k-1};i)$. It is enough to verify that $\mu$ satisfies the following three conditions:
\begin{itemize}
    \item[(A)] There is only one $2t+1$ in $\mu$ and $2t+1$ is marked with $1$ in $GG(\mu)$; 
    
    \item[(B)] The new generated part $2b+3$   is also marked with $r$ in $GG(\mu)$.
    
     \item[(C)] The marks of the unchanged parts in $GG(\mu)$ are the same as those in $GG(\lambda)$.
    
\end{itemize}

{\noindent Condition (A).} Under the condition that $\lambda^{(1)}_{p+1}+2\leq 2t+1<\lambda^{(1)}_p$, we  deduce that  $\lambda^{(1)}_{p+1}\leq 2t-1$ and $\lambda^{(1)}_p\geq 2t+2$. So, there is no part $2t+1$ in $\lambda$. Otherwise, we must have $\lambda^{(1)}_p=2t+1$, which contradicts the fact that $\lambda^{(1)}_p\geq 2t+2$. Hence, there is only one $2t+1$ in $\mu$, which is marked with $1$ in $GG(\mu)$.

{\noindent Condition (B).} It suffices to show that there are $1$-marked, \ldots, $(r-1)$-marked $2b+2$ in $GG(\lambda)$ when $r\geq 2$. If $r\geq 2$, then $b\geq t+1$, namely, $\lambda^{(1)}_p\geq 2t+4$.  Note that $\lambda^{(1)}_{p+1}\leq 2t-1$, so  $2t+1$,  \ldots, $2b$ and $2b+1$ do not occur in $\lambda$. Under the condition that  there is a $r$-marked $2b+2$ in $GG(\lambda)$,  by the definition of G\"ollnitz-Gordon marking, we see that there are $1$-marked, \ldots, $(r-1)$-marked $2b+2$ in $GG(\lambda)$. 

{\noindent Condition (C).}  From the proofs of Conditions (A) and (B), we find that the marks of the unchanged parts less than $2b+3$ in $GG(\mu)$ are the same as those in $GG(\lambda)$. Moreover, by the definition of G\"ollnitz-Gordon marking, we derive that  the marks of the unchanged parts greater than $2b+3$ in $GG(\mu)$ are the same as those in $GG(\lambda)$. Hence, we conclude that $\mu$ is a partition in $\mathbb{C}_1(N_1+1,\ldots,N_{k-1};i)$.

Now, we turn to show that $\mu\in\mathbb{C}^{\lhd}_1(N_1+1,\ldots,N_{k-1};i|p,t)$.  From the proof above, we have $\mu^{(1)}_s=\lambda^{(1)}_s$ for $1\leq s\leq p-1$, $\mu^{(1)}_{p+1}=2t+1$, $\mu^{(1)}_{s}=\lambda^{(1)}_{s-1}$ for $p+2\leq s\leq N_1+1$. 
If $r=1$, then we have $\mu^{(1)}_p=\lambda^{(1)}_p+1=2b+3$. If $r\geq 2$, then we have $\mu^{(1)}_p=\lambda^{(1)}_p=2b+2$ and there is a $r$-marked $\lambda^{(1)}_p+1=2b+3$ in $GG(\mu)$. So, $\mu^{(1)}_p$ is odd type. Moreover, by the construction of $\mu$, we find that  $\mu^{(1)}_s$ is even type for  $1\leq s\leq p-1$, and so we deduce that  $\mu\in\mathbb{C}_1(N_1+1,\ldots,N_{k-1};i|p)$.

Assume that  there exist parts $2t+2$ in $\mu$, then the marks of parts $2t+2$ are greater than $1$ in $GG(\mu)$ since $\mu^{(1)}_{p+1}=2t+1$.
From the proof above, we find that the marks of parts $2t+2$ in $GG(\mu)$ are the same as those in $GG(\lambda)$. By the definition of G\"ollnitz-Gordon marking, we derive that there is a $1$-marked $2t$ or $2t+1$ or $2t+2$ in $GG(\lambda)$. Note that $\lambda^{(1)}_{p+1}\leq 2t-1$ and $\lambda^{(1)}_{p}\geq 2t+2$, so we get $\lambda^{(1)}_{p}=2t+2$. Therefore, we deduce  that $r=1$ and $\mu^{(1)}_{p}=2t+3$.
Hence, we conclude that $\mu$ is a partition in $\mathbb{C}^{\lhd}_1(N_1+1,\ldots,N_{k-1};i|p,t)$.

We proceed to prove that for $k\geq 3$ and $k\geq i\geq 2$, if $\lambda\in\mathbb{C}^{<}_0(N_1,\ldots,N_{k-1};i|p,t)$, then $\mu\in\mathbb{C}^{\lhd}_0(N_1+1,\ldots,N_{k-1};i|p,t)$. Let $\lambda^{(k-1)}=(\lambda^{(k-1)}_1,\lambda^{(k-1)}_2,\ldots, \lambda^{(k-1)}_{N_{k-1}})$ and $\mu^{(k-1)}=(\mu^{(k-1)}_1,\mu^{(k-1)}_2,\ldots, \mu^{(k-1)}_{N_{k-1}})$ be the $(k-1)$-marked parts in $GG(\lambda)$ and $GG(\mu)$, respectively. Assume that   $\lambda$ is a partition in $\mathbb{C}^{<}_0(N_1,\ldots,N_{k-1};i|p,t)$, then  $\{\lambda^{(k-1)}_s\}_{k-1}$ is  good for $1\leq s\leq N_{k-1}$. We aim to show that  $\{\mu^{(k-1)}_s\}_{k-1}$ is  good for $1\leq s\leq N_{k-1}$.

For each $s\in[1,N_{k-1}]$, we have $\lambda^{(1)}_{p+1}\leq 2t-1$ and $\lambda^{(1)}_{p}=2b+2\geq 2t+2$, and so $\lambda^{(k-1)}_s\not \in[2t+1,2b+1]$. By the construction of $\mu$, we derive that $\mu^{(k-1)}_s\not \in[2t+1,2b+1]$. To show that  $\{\mu^{(k-1)}_s\}_{k-1}$ is  good,  we consider the following four cases.

Case 1: $\mu^{(k-1)}_s\leq 2t$ or $\mu^{(k-1)}_s\geq 2b+5$. By the construction of $\mu$, we have $\{\mu^{(k-1)}_s\}_{k-1}=\{\lambda^{(k-1)}_s\}_{k-1}$. Moreover, if $\mu^{(k-1)}_s\leq 2t$, then $O_{\mu}(\mu^{(k-1)}_s)=O_{\lambda}(\lambda^{(k-1)}_s)$; if  $\mu^{(k-1)}_s\geq 2b+5$, then $O_{\mu}(\mu^{(k-1)}_s)= O_{\lambda}(\lambda^{(k-1)}_s)+2$. This implies that $O_{\mu}(\mu^{(k-1)}_s)\equiv O_{\lambda}(\lambda^{(k-1)}_s)\pmod2$. Note that $\{\lambda^{(k-1)}_s\}_{k-1}$ is good, so $\{\mu^{(k-1)}_s\}_{k-1}$ is  good.

Case 2: $\mu^{(k-1)}_s=2b+2$. By the construction of $\mu$, we find that  $b=t$ and $r=1$. Moreover, we have $\mu^{(k-1)}_s=\lambda^{(k-1)}_s=2b+2=2t+2$, $f_{2t}(\mu)=f_{2t}(\lambda)$, $f_{2t+1}(\mu)=f_{2t+1}(\lambda)+1$, $f_{2t+2}(\mu)=f_{2t+2}(\lambda)-1$ and  $O_{\mu}(2t+2)=O_{\lambda}(2t+2)+1$. So we get
\[
\begin{split}
&[\mu^{(k-1)}_{s,1}/2]+[\mu^{(k-1)}_{s,2}/2]+\cdots+[\mu^{(k-1)}_{s,k-1}/2]\\
&\qquad =tf_{2t}(\mu)+tf_{2t+1}(\mu)+(t+1)f_{2t+2}(\mu)\\
&\qquad =tf_{2t}(\lambda)+t(f_{2t+1}(\lambda)+1)+(t+1)(f_{2t+2}(\lambda)-1)\\
&\qquad =tf_{2t}(\lambda)+tf_{2t+1}(\lambda)+(t+1)f_{2t+2}(\lambda)-1\\
&\qquad =[\lambda^{(k-1)}_{s,1}/2]+[\lambda^{(k-1)}_{s,2}/2]+\cdots+[\lambda^{(k-1)}_{s,k-1}/2]-1\\
&\qquad \equiv i-1+O_\lambda(\lambda^{(k-1)}_{s,1})-1=i-1+O_\lambda(2t+2)-1\\
&\qquad \equiv i-1+O_\mu(2t+2)=i-1+O_\mu(\mu^{(k-1)}_{s,1})\pmod2,
\end{split}
\]
which implies that $\{\mu^{(k-1)}_s\}_{k-1}$ is good.

Case 3: $\mu^{(k-1)}_s=2b+3$. By the construction of $\mu$, we find that $O_{\mu}(\mu^{(k-1)}_s)= O_{\lambda}(\lambda^{(k-1)}_s)+2\equiv O_{\lambda}(\lambda^{(k-1)}_s)\pmod2$. With the similar argument of Case 4 in the proof of Lemma \ref{varphi1}, we derive that $\{\mu^{(k-1)}_s\}_{k-1}$ is  good.

Case 4: $\mu^{(k-1)}_s=2b+4$. By the construction of $\mu$, we observe that $O_{\mu}(\mu^{(k-1)}_s)= O_{\lambda}(\lambda^{(k-1)}_s)+2\equiv O_{\lambda}(\lambda^{(k-1)}_s)\pmod2$.  With the similar argument of Case 5 in the proof of Lemma \ref{varphi1}, we   obtain that $\{\mu^{(k-1)}_s\}_{k-1}$ is   good.

Therefore, we conclude that $\{\mu^{(k-1)}_s\}_{k-1}$ is good. Hence, by Theorem \ref{parity k-1 sequence}, we see that
$\mu$ is a partition in $\mathbb{C}^{\lhd}_0(N_1+1,\ldots,N_{k-1};i|p,t)$.

 To prove that the combination $C_{p,t}$ is a bijection, we construct the inverse map  of
$C_{p,t}$, the division $D_{p,t}$. Assume that $\mu$ is a partition in $\mathbb{C}^{\lhd}_j(N_1+1,\ldots,N_{k-1};i|p,t)$, then we see that   $\mu^{(1)}_{p+1}=2t+1$ and $\mu_{p}^{(1)}$ is odd type. If $\mu_{p}^{(1)}$ is an odd part, then we may write $\mu_{p}^{(1)}=2b+3$ and set $r=1$, where $b\geq t$. If $\mu_{p}^{(1)}$ is an even part, then we may write $\mu_{p}^{(1)}=2b+2$ and set $r$ be the mark of $2b+3$ in $GG(\mu)$, where $b\geq t+1$ and $r\geq 2$. Define $\lambda=D_{p,t}(\mu)$ by removing the part $\mu^{(1)}_{p+1}=2t+1$ from $\mu$ and replacing  the $r$-marked ${2b+3}$ in $GG(\mu)$ by   ${2b+2}$ to obtain $\lambda$. It can be verified  that the division $D_{p,t}$ is the inverse map of the combination $C_{p,t}$. This completes the proof.  \qed

   \begin{lem}\label{deltagammathmbb-p0} There is a bijection, the combination $C_{0,t}$, between   $\mathbb{C}^{<}_j(N_1,\ldots,N_{k-1};i|0,t)$ and  $\mathbb{C}^{\lhd}_j(N_1+1,\ldots,N_{k-1};i|0,t)$. Moreover, for  $\lambda \in \mathbb{C}^{<}_j(N_1,\ldots, N_{k-1};i|0,t)$ and $\mu=C_{0,t}(\lambda)\in\mathbb{C}^{\lhd}_j(N_1+1,\ldots,N_{k-1};i|0,t)$, we have $|\mu|=|\lambda|+2t+1$.
 \end{lem}

\pf For $k\geq i\geq 2$, let $\lambda$ be a partition in $\mathbb{C}^{<}_1(N_1,\ldots,N_{k-1};i|0,t)$. Define $\mu=C_{0,t}(\lambda)$ by adding $2t+1$ as a new part of $\lambda$ to obtain $\mu$. Obviously, we have $|\mu|=|\lambda|+2t+1$.   Note that $\lambda^{(1)}_1=\lambda^{(1)}_{p+1}\leq 2t-1$, we find that the parts of $\lambda$ are not exceed to $2t$. So, $2t+1$ is marked with $1$ in $GG(\mu)$ and the marks of the remaining parts in $GG(\mu)$ are the same as those in $GG(\lambda)$. Therefore,  $\mu$ is a partition  in $\mathbb{C}_1(N_1+1,\ldots,N_{k-1};i)$ and $\mu^{(1)}_1=2t+1$, which implies that $\mu\in\mathbb{C}^{\lhd}_1(N_1+1,\ldots,N_{k-1};i|0,t)$.

We turn to show that for $k\geq 3$ and $k\geq i\geq 2$, if $\lambda\in\mathbb{C}^{<}_0(N_1,\ldots,N_{k-1};i|0,t)$, then $\mu\in\mathbb{C}^{\lhd}_0(N_1+1,\ldots,N_{k-1};i|0,t)$. Let $\lambda^{(k-1)}=(\lambda^{(k-1)}_1,\lambda^{(k-1)}_2,\ldots, \lambda^{(k-1)}_{N_{k-1}})$ and $\mu^{(k-1)}=(\mu^{(k-1)}_1,\mu^{(k-1)}_2,\ldots, \mu^{(k-1)}_{N_{k-1}})$ be the $(k-1)$-marked parts in $GG(\lambda)$ and $GG(\mu)$, respectively. Assume that   $\lambda$ is a partition in $\mathbb{C}^{<}_0(N_1,\ldots,N_{k-1};i|0,t)$, then  $\{\lambda^{(k-1)}_s\}_{k-1}$ is  good for $1\leq s\leq N_{k-1}$. We aim to show that  $\{\mu^{(k-1)}_s\}_{k-1}$ is  good for $1\leq s\leq N_{k-1}$.

For each $s\in[1,N_{k-1}]$, by the construction of $\mu$, we have $\mu^{(k-1)}_s=\lambda^{(k-1)}_s$, $\{\mu^{(k-1)}_s\}_{k-1}=\{\lambda^{(k-1)}_s\}_{k-1}$ and $O_{\mu}(\mu^{(k-1)}_s)=O_{\lambda}(\lambda^{(k-1)}_s)$. Note that $\{\lambda^{(k-1)}_s\}_{k-1}$ is good, so we get that  $\{\mu^{(k-1)}_s\}_{k-1}$ is  good.
Hence, by Theorem \ref{parity k-1 sequence}, we see that
$\mu$ is a partition in $\mathbb{C}^{\lhd}_0(N_1+1,\ldots,N_{k-1};i|0,t)$. 

 To prove that the combination $C_{0,t}$ is a bijection, we construct the inverse map  of
$C_{0,t}$, the division $D_{0,t}$. Assume that $\mu$ is a partition in $\mathbb{C}^{\lhd}_j(N_1+1,\ldots,N_{k-1};i|0,t)$, then we see that   $\mu^{(1)}_{1}=2t+1$. Define $\lambda=D_{0,t}(\mu)$ by removing the part $\mu^{(1)}_{1}=2t+1$ from $\mu$ to obtain $\lambda$. It can be verified  that the division $D_{0,t}$ is the inverse map of the combination $C_{0,t}$. This completes the proof.  \qed

\subsection{The insertion $I_{m}$ and the separation $S_{m}$}

In this subsection, we assume that $j=0$ or $1$, $(2k+j)/2>2$, $k\geq i\geq 2$ and $N_1\geq\cdots\geq N_{k-1}\geq0$. Let $\mu$ be a partition  in $\mathbb{C}_j(N_1+1,\ldots,N_{k-1};i)$ such that there exist odd parts in $\mu$. We say that $2t+1$ is a insertion part of $\mu$ if $2t+1$ is the largest odd part $2t+1$ of $\mu$, $\mu_{p+1}^{(1)}=2t+1$, where $0\leq p\leq N_1$, and $\mu$ satisfies either of the following two conditions:
\begin{itemize}
\item[\rm{(1)}] There do not exist parts $2t+2$ in $\mu$.

\item[\rm{(2)}] There exist parts $2t+2$ in $\mu$, and  there exists $s\in[1,p]$ such that  $\mu^{(1)}_{p+1-s}=2t+4s$ and there do not exist parts $2t+4s+2$ in $\mu$.
\end{itemize}
For $t\geq 0$ and $0\leq p\leq N_1$, let $\mathbb{C}^{=}_j(N_1+1,\ldots,N_{k-1};i|p,t)$ denote the set of   partitions $\mu$ in $\mathbb{C}_j(N_1+1,\ldots,N_{k-1};i)$ such that $\mu^{(1)}_{p+1}= 2t+1$ is a insertion part of $\mu$.

  \begin{thm}\label{deltagammathmbb-insertion} For $t\geq 0$ and $0\leq p\leq N_1$, there exists a bijection  $I_{p,t}$  between  $\mathbb{C}^{<}_j(N_1,\ldots,\break N_{k-1};i|p,t)$ and    $\mathbb{C}^{=}_j(N_1+1,\ldots,N_{k-1};i|p,t)$. Moreover, for  $\lambda \in \mathbb{C}^{<}_j(N_1,\ldots,N_{k-1};i|p,t)$ and $\mu=I_{p,t}(\lambda)\in\mathbb{C}^{=}_j(N_1+1,\ldots,N_{k-1};i|p,t)$, we have $|\mu|=|\lambda|+2p+2t+1$.
 \end{thm}

 \pf For convenience, we define $\Theta_{(0)}$ as the identity map. Then we define $\mu=I_{p,t}(\lambda)=\Theta_{(p)}(C_{p,t}(\lambda))$. By Lemmas \ref{N-2-2-lem}, \ref{deltagammathmbb} and \ref{deltagammathmbb-p0}, we find that $I_{p,t}$ is well-defined, $\mu\in \mathbb{C}_j(N_1+1,\ldots,N_{k-1};i)$, $\mu^{(1)}_{p+1}=2t+1$ is the largest odd part of $\mu$ and $|\mu|=|\lambda|+2p+2t+1$. We proceed to show that $\mu\in\mathbb{C}^{=}_j(N_1+1,\ldots,N_{k-1};i|p,t)$.
 
If $p=0$, then by Lemma  \ref{deltagammathmbb-p0}, we have  $\mu\in\mathbb{C}^{=}_1(N_1+1,\ldots,N_{k-1};i|0,t)$, which implies that $\mu\in\mathbb{C}^{=}_j(N_1+1,\ldots,N_{k-1};i|0,t)$.

 If $p\geq 1$, then it suffices to show that $\mu$ satisfies either condition (1) or condition (2) in the definition of an insertion part. 
 If there do not exist parts $2t+2$ in $\mu$, then it is clear that $\mu$  satisfies the condition (1)  in the definition of an insertion part. 
 
  Assume that there exist parts $2t+2$ in $\mu$, recall that $\mu^{(1)}_{p+1}=2t+1$, then we see that the marks of  parts $2t+2$ are greater than $1$ in $GG(\mu)$. By the construction of $\mu$, we find that  $\lambda^{(1)}_p=2t+2$ and there exist parts $2t+2$ with marks greater than $1$ in $GG(\lambda)$.
  Set $s$ be the largest integer such that $\lambda^{(1)}_{p+1-s}=2t+4s-2$ and there exist parts $2t+4s-2$ with marks greater than $1$ in $GG(\lambda)$. Using $\lambda^{(1)}_p=2t+2$, we deduce that $1\leq s\leq p$. It follows from the construction of $\mu$ that $\lambda^{(1)}_{p+1-s}=2t+4s-2$  is replaced by $2t+4s$ in $\mu$, and so  $\mu^{(1)}_{p+1-s}=2t+4s$. If there do not exist parts $2t+4s+2$ in $\lambda$, then by the construction of $\mu$, we see that there also do not exist parts $2t+4s+2$ in $\mu$. If there exist parts $2t+4s+2$ in $\lambda$, then by the choice of $s$, we find that $\lambda^{(1)}_{p-s}=2t+4s+2$ and there do not  parts  exist $2t+4s+2$ with marks greater than $1$ in $GG(\lambda)$. Again by the construction of $\mu$, we see that  $\lambda^{(1)}_{p-s}=2t+4s+2$ is replaced by $2t+4s+4$ in $\mu$, which implies that there do not exist parts $2t+4s+2$ in $\mu$. So, $\mu$  satisfies the condition (2)  in the definition of insertion part.  Hence, $\mu$ is a partition in $\mathbb{C}^{=}_j(N_1+1,\ldots,N_{k-1};i|p,t)$.

 To prove that $I_{p,t}$ is a bijection, we construct the inverse map of the insertion
$I_{p,t}$, the separation $S_{p,t}$. Define $\Lambda_{(0)}$ as the identity map. Then we define $\lambda=S_{p,t}(\mu)=D_{p,t}(\Lambda_{(p)}(\mu))$. By Lemmas \ref{N-2-2-lem}, \ref{deltagammathmbb} and \ref{deltagammathmbb-p0}, it can be verified  that the separation $S_{p,t}$ is the inverse map of the insertion $I_{p,t}$. This completes the proof.  \qed

For $m\geq N_1$, define
\[\mathbb{C}^{<}_j(N_1,\ldots,N_{k-1};i|m)=\bigcup_{p+t=m}\mathbb{C}^{<}_j(N_1,\ldots,N_{k-1};i|p,t),\]
and
\[\mathbb{C}^{=}_j(N_1+1,\ldots,N_{k-1};i|m)=\bigcup_{p+t=m}\mathbb{C}^{=}_j(N_1+1,\ldots,N_{k-1};i|p,t).\]
We establish a bijection between $\mathbb{C}^{<}_j(N_1,\ldots,N_{k-1};i|m)$ and $\mathbb{C}^{=}_j(N_1+1,\ldots,N_{k-1};i|m)$. 

 \begin{thm}\label{deltagammathmbb-insertion-all} For   $m\geq N_1$, there is a bijection, the insertion $I_{m}$, between   $\mathbb{C}^{<}_j(N_1,\ldots,\break N_{k-1};i|m)$ and $\mathbb{C}^{=}_j(N_1+1,\ldots,N_{k-1};i|m)$. Moreover, for  $\lambda \in \mathbb{C}^{<}_j(N_1,\ldots,N_{k-1};i|m)$ and $\mu=I_{m}(\lambda)\in\mathbb{C}^{=}_j(N_1+1,\ldots,N_{k-1};i|m)$, we have $|\mu|=|\lambda|+2m+1$.
 \end{thm}

 \pf Let $\lambda$ be a partition in $\mathbb{C}^{<}_j(N_1,\ldots,N_{k-1};i|m)$ and let $p$ be the smallest integer such that $2(m-p)+1\geq \lambda^{(1)}_{p+1}+2$, where $0\leq p\leq N_1$. Set $t=m-p$, then we have  $\lambda\in\mathbb{C}^{<}_j(N_1,\ldots,N_{k-1};i|p,t)$. Define $\mu=I_m(\lambda)=I_{p,t}(\lambda)$. By Theorem \ref{deltagammathmbb-insertion}, we see that $\mu\in\mathbb{C}^{=}_j(N_1+1,\ldots,N_{k-1};i|p,t)\subset\mathbb{C}^{=}_j(N_1+1,\ldots,N_{k-1};i|m)$ and $|\mu|=|\lambda|+2t+2p+1=|\lambda|+2m+1$.

 To prove that $I_{m}$ is a bijection, we construct the inverse map of the insertion
$I_{m}$, the separation $S_{m}$.  Let $\mu$ be a partition in $\mathbb{C}^{=}_j(N_1+1,\ldots,N_{k-1};i|m)$.   Assume that $\mu^{(1)}_{p+1}=2t+1$ is the insertion part of $\mu$, where $m=p+t$. Define $\lambda=S_{m}(\mu)=S_{p,t}(\mu)$. It is clear  that the separation $S_{m}$ is the inverse map of $I_{m}$.  This completes the proof.  \qed

The following theorem give a criterion to determine whether a partition in $\mathbb{C}^{=}_j(N_1+1,\ldots,N_{k-1};i|m)$ also belongs to $\mathbb{C}^{<}_j(N_1+1,\ldots,N_{k-1};i|m')$, which involves the successive applications of the insertion.

 \begin{thm}\label{ssins}
 For  $m\geq N_1$, let $\mu$ be a partition in $\mathbb{C}^{=}_j(N_1+1,\ldots,N_{k-1};i|m)$. Then $\mu$ is a partition in $\mathbb{C}^{<}_j(N_1+1,\ldots,N_{k-1};i|m')$ if and only if $m'>m$.
\end{thm}

\pf  Assume that $\mu^{(1)}_{p+1}=2t+1$ is the insertion part of $\mu$, where $m=p+t$.

If $m'>m$, then we have $2m'+1-2p\geq 2(m+1)+1-2p=(2t+1)+2=\mu^{(1)}_{p+1}+2$. Assume that $p'$ is the smallest integer such that  $2m'+1-2p'\geq \mu^{(1)}_{p'+1}+2$, we see that  $p'\leq p$. Set $t'=m'-p'$, then we have $\mu\in \mathbb{C}^{<}_j(N_1+1,\ldots,N_{k-1};i|p',t')$, which implies that $\mu\in\mathbb{C}^{<}_j(N_1+1,\ldots,N_{k-1};i|m')$.

On the other hand, assume that $\mu$ is a partition in $\mathbb{C}^{<}_j(N_1+1,\ldots,N_{k-1};i|m')$. Let $p'$ be the smallest integer such that $2(m'-p')+1\geq \mu^{(1)}_{p'+1}+2$, where $0\leq p'\leq N_1+1$. Set $t'=m'-p'$, then we have  $\mu\in\mathbb{C}^{<}_j(N_1+1,\ldots,N_{k-1};i|p',t')$, and so $t'> t$ and $p'\leq p$. Hence
 \[2t'+1\geq \mu^{(1)}_{p'+1}+2\geq\cdots\geq \mu^{(1)}_{p+1}+2(p-p'+1)= 2t+1+2(p-p'+1),\]
which leads to $p'+t'\geq p+t+1$. This implies $m'=p'+t'>p+t=m$. 

Thus, we complete the proof.  \qed

 \subsection{Proof of Theorem \ref{equiv-main-thm-2}}

 We are now in a position to give a proof of Theorem \ref{equiv-main-thm-2}.

 {\noindent \bf Proof of Theorem \ref{equiv-main-thm-2}.} Let $(\lambda,-,\eta)$ be a pair in $\mathbb{F}_1(2,2)$. Then, we have $\lambda\in\mathbb{E}_1(2,2)$. We define $\pi=\Phi_{2,2}(\lambda,-,\eta)$ as follows. There are two cases.
 
 Case 1: $\eta=\emptyset$. Then we set $\pi=\lambda$. It is clear that $\pi\in\mathbb{E}_1(2,2)\subset C_1(2,2)$. Moreover, $|\pi|=|\lambda|+|\eta|$ and $\ell(\pi)=\ell(\lambda)+\ell(\eta)$.
 
 Case 2: $\eta\neq\emptyset$. Note that  $\lambda$ is a partition in  $\mathbb{E}_1(2,2)$, so  the parts of $\lambda$ are even and the marks of parts in $GG(\lambda)$ are not exceed to $1$. Assume that there are $N_1$ $1$-marked parts in $GG(\lambda)$, denoted
 $\lambda=\lambda^{(1)}=(\lambda^{(1)}_1,\lambda^{(1)}_2,\ldots,\lambda^{(1)}_{N_1})\in\mathbb{E}_1({N_1};2)$. Then we have $\eta\in\mathbb{I}_{N_1}$, denoted  $\eta=(2m_1+1,2m_2+1,\ldots,2m_\ell+1)$, where $m_1>m_2>\cdots>m_\ell\geq N_1$.

Starting with $\lambda$, apply the insertion repeatedly to get $\pi$. Denote the intermediate partitions by $\lambda^0,\lambda^1,\ldots,\lambda^\ell$ with $\lambda^0=\lambda$ and $\lambda^\ell=\pi$. Since $m_\ell\geq N_1$,  there exists $p$ such that $p$ is the smallest integer $p$ satisfying $2m_\ell+1-2p\geq \lambda^{(1)}_{p+1}+2$. Note that there do not exist odd parts in $\lambda$, so $\lambda^{(1)}_s$ is even type for $1\leq s\leq N_1$. Set $t=m_\ell-p$, then we have $\lambda\in\mathbb{C}^<_1({N_1};2|p,t)\subset\mathbb{C}^<_1({N_1};2|m_\ell)$.

Set $b=0$ and repeat following process until $b=\ell$.

\begin{itemize}
    \item[(A)] Note that $\lambda^b\in\mathbb{C}^<_1({N_1+b};2|m_{\ell-b})$, we apply the insertion $I_{m_{\ell-b}}$ to $\lambda^b$ to get $\lambda^{b+1}$, that is,
   \[\lambda^{b+1}=I_{m_{\ell-b}}(\lambda^{b}).\]
   In view of Theorem \ref{deltagammathmbb-insertion-all}, we deduce that 
   \[
   \lambda^{b+1}\in\mathbb{C}^=_1({N_1+b+1};2|m_{\ell-b})\]
   and 
   \[|\lambda^{b+1}|=|\lambda^b|+2m_{\ell-b}+1.\]
    
    \item[(B)] Replace $b$ by $b+1$. If $b=\ell$, then we are done. If $b<\ell$, then by Theorem \ref{ssins}, we have
     \[
   \lambda^{b}\in\mathbb{C}^<_1({N_1+b};2|m_{\ell-b}),\]
 since $m_{\ell-b}>m_{\ell-b+1}$. Go back to (A).    
\end{itemize}  

Eventually, the above process yields $\pi=\lambda^\ell\in \mathbb{C}^=_1({N_1+\ell};2|m_{1})$ such that $|\pi|=|\lambda|+|\eta|$. Moreover, we have $\pi\in\mathbb{C}_1(2;2)$ and $\ell(\pi)=\ell(\lambda)+\ell(\eta)$.

 To show that $\Phi_{2,2}$ is a bijection, we give the inverse map $\Psi_{2,2}$ of $\Phi_{2,2}$ by successively applying the separation defined in the proof of Theorem \ref{deltagammathmbb-insertion-all}. Let $\pi$ be a partition in $\mathbb{C}_1(2;2)$. We shall construct a pair $(\lambda,-,\eta)$, that is, $(\lambda,-,\eta)=\Psi_{2,2}(\pi)$, such that $(\lambda,-,\eta)\in \mathbb{F}_1(2;2)$, $|\pi|=|\lambda|+|\eta|$ and $\ell(\pi)=\ell(\lambda)+\ell(\eta)$.

 Assume that there are $\ell\geq 0$ odd parts and $N_1$ even parts  in $\pi$. We eliminate all odd parts of $\pi$. There are two cases.
 
 Case 1: $\ell=0$. Then set $\lambda=\pi$ and $\eta=\emptyset$. Clearly, $(\lambda,-,\eta)\in \mathbb{F}_1(2;2)$, $|\pi|=|\lambda|+|\eta|$ and $\ell(\pi)=\ell(\lambda)+\ell(\eta)$.
 
 Case 2: $\ell\geq 1$. Assume that $2t_0+1>2t_1+1>\cdots>2t_{\ell-1}+1$ are the odd parts of $\pi$. We eliminate the $\ell$ odd parts of $\pi$  by successively applying the separation defined in the proof of Theorem \ref{deltagammathmbb-insertion-all}. Denote the intermediate pairs by $(\pi^0,-,\eta^0)$, $(\pi^1,-,\eta^1)$, \ldots, $(\pi^\ell,-,\eta^\ell)$ with  $(\pi^0,-,\eta^0)=(\pi,-,\emptyset)$.
 
 Set $b=0$ and carry out the following procedure.

 \begin{itemize}
     \item[(A)] Note that $\pi^b$ is a partition in $\mathbb{C}_1(N_1+\ell-b;2)$ and  $2t_b+1$ is the largest odd part of $\pi^b$, we see that there are no parts $2t_b+2$ in $\pi^{b}$,  otherwise the largest mark in $GG(\pi^{b})$ are greater than $1$, which leads to a contradiction.  So, $2t_b+1$ is a insertion part of $\pi^{b}$. Assume that $2t_b+1$ is the $(p_b+1)$-th part of $\pi^b$.  Let $m_b=p_b+t_b$. Then, 
      \[
   \pi^{b}\in\mathbb{C}^=_1({N_1+\ell-b};2|m_{b}).\]
     Apply the separation $S_{m_b}$ to $\pi^b$ to get $\pi^{b+1}$, that is,
      \[\pi^{b+1}=S_{m_{b}}(\pi^{b}).\]
      Employing Theorem \ref{deltagammathmbb-insertion-all}, we find that 
   \[
   \pi^{b+1}\in\mathbb{C}^<_1({N_1+\ell-b-1};2|m_{b})\]
   and 
   \[|\pi^{b+1}|=|\pi^b|-2m_{b}-1.\]
     Then insert $2m_{b}+1$ into $\eta^b$ to obtain $\eta^{b+1}$.
     
     \item[(B)] Replace $b$ by $b+1$. If $b=\ell$, then we are done. Otherwise, go back to (A).
     
 \end{itemize}

Observe that for $0\leq b\leq \ell$, there are $\ell-b$ odd parts in $\pi^b$. Theorem \ref{ssins} reveals that for $0\leq b<\ell-1$,
\[m_b>m_{b+1}\geq N_1+\ell-b-2.\]
Therefore, there are no odd parts in $\pi^\ell$ and $\eta^\ell=(2m_0+1,2m_1+1,\ldots,2m_{\ell-1}+1)$ is a partition in $\mathbb{I}_{N_1}$.
Set $(\lambda,-,\eta)=(\pi^\ell,-,\eta^\ell)$.  It is clear that $\lambda=\pi^\ell\in\mathbb{C}^<_1({N_1};2|m_{\ell-1})\subset\mathbb{C}_1(N_1;2)$, $\eta=\eta^\ell\in \mathbb{I}_{N_1}$, $|\pi|=|\lambda|+|\eta|$ and $\ell(\pi)=\ell(\lambda)+\ell(\eta)$. Recall that there do not exist odd parts in $\lambda$,  so  $\lambda\in\mathbb{E}_1(2;2)$. Hence, $(\lambda,-,\eta)$ is a pair in  $\mathbb{F}_1(2,2)$.

It can be verified that $\Psi_{2,2}$ is the inverse map of $\Phi_{2,2}$. Thus, we complete the proof. \qed

 \section{The reduction $R_p$ and the dilation $H_p$}

Assume that  $k\geq 3$, $k\geq i\geq 1$, $j=0$ or $1$, $N_1\geq\cdots\geq N_{k-1}\geq0$ and $1\leq p\leq N_2$. In this section, we introduce two subsets of $\mathbb{C}_j(N_1,\ldots,N_{k-1};i)$, $\mathbb{\hat{C}}_j(N_1,\ldots,N_{k-1};i|p)$ and  $\mathbb{\check{C}}_j(N_1,\ldots,N_{k-1};i|p)$, show the follow theorem which is proved by the reduction operation and the dilation operation.  

  \begin{thm}\label{reduction-1} There is a bijection, the reduction $R_p$, between    $\mathbb{\hat{C}}_j(N_1,\ldots,N_{k-1};i|p)$ and the set $\mathbb{\check{C}}_j(N_1,\ldots,N_{k-1};i|p)$. Moreover, for  $\lambda \in \mathbb{\hat{C}}_j(N_1,\ldots,N_{k-1};i|p)$ and $\mu=R_p(\lambda)\in\mathbb{\check{C}}_j(N_1,\ldots,N_{k-1};i|p)$, we have
$|\mu|=|\lambda|-2p+1$.
 \end{thm}

\subsection{The set $\mathbb{\hat{C}}_j(N_1,\ldots,N_{k-1};i|p)$}

Let $\lambda$ be a partition in $\mathbb{C}_j(N_1,\ldots,N_{k-1};i)$ and let $\lambda^{(r)}=(\lambda^{(r)}_1,\lambda^{(r)}_2,\ldots,\lambda^{(r)}_{N_r})$ denote $r$-marked parts in $GG(\lambda)$, where $\lambda^{(r)}_1>\lambda^{(r)}_2>\cdots>\lambda^{(r)}_{N_r}$ and $1\leq r\leq k-1$. We will focus on the $2$-marked parts in $GG(\lambda)$.

{\bf \noindent The types of $\lambda^{(2)}_1,\lambda^{(2)}_2,\ldots,\lambda^{(2)}_p$.} Assume that $\lambda^{(2)}_p$ is even and  there do not exist odd parts greater than $\lambda^{(2)}_p$ in $\lambda$, then we define the types of $\lambda^{(2)}_1,\lambda^{(2)}_2,\ldots,\lambda^{(2)}_p$  from largest to smallest  as follows. We first define the type of $\lambda^{(2)}_1$. There are two cases.
\begin{itemize}
\item[(1)] If there exist parts $\lambda^{(2)}_1+2$ with marks greater than $2$ in $GG(\lambda)$, then by the definition of G\"ollnitz-Gordon marking, there is a $1$-marked part $\lambda^{(1)}_{s_1}=\lambda^{(2)}_1$ or $\lambda^{(2)}_1+2$ in $GG(\lambda)$. We set $\hat{R}_1(\lambda)=s_1$. There are two subcases.

\begin{itemize}
   \item[(1.1)] If $\lambda^{(1)}_{s_1}=\lambda^{(2)}_1$, then we say that  $\lambda^{(2)}_1$ is of type $\hat{A}_1$.

        \item[(1.2)] If $\lambda^{(1)}_{s_1}=\lambda^{(2)}_1+2$, then we say that  $\lambda^{(2)}_1$ is of type $\hat{B}$.
\end{itemize}

\item[(2)] If there do not exist parts $\lambda^{(2)}_1+2$ with marks greater than $2$ in $GG(\lambda)$, then by the definition of G\"ollnitz-Gordon marking, there is a $1$-marked part $\lambda^{(1)}_{s_1}=\lambda^{(2)}_1-2$ or $\lambda^{(2)}_1-1$ or $\lambda^{(2)}_1$ in $GG(\lambda)$. We set $\hat{R}_1(\lambda)=s_1$. There are three subcases.

\begin{itemize}
\item[(2.1)] If $\lambda^{(1)}_{s_1}=\lambda^{(2)}_1-2$, then we say that  $\lambda^{(2)}_1$ is of type $\hat{C}$.

    \item[(2.2)] If $\lambda^{(1)}_{s_1}=\lambda^{(2)}_1-1$, then we say that  $\lambda^{(2)}_1$ is of type $\hat{O}$.

    \item[(2.3)] If $\lambda^{(1)}_{s_1}=\lambda^{(2)}_1$, then we say that  $\lambda^{(2)}_1$ is of type $\hat{A}_2$.
\end{itemize}

\end{itemize}

 For $1<b\leq p$, assume that  $\hat{R}_{b-1}(\lambda)$ and the type of $\lambda^{(2)}_{b-1}$  have been defined, then we define the type of $\lambda^{(2)}_b$ as follows. There are two cases.

 \begin{itemize}
\item[(1)] If there exist parts $\lambda^{(2)}_b+2$ with marks greater than $2$ in $GG(\lambda)$, then by the definition of G\"ollnitz-Gordon marking, there is a $1$-marked part $\lambda^{(1)}_{s_b}=\lambda^{(2)}_b$ or $\lambda^{(2)}_b+2$ in $GG(\lambda)$. There are four subcases.

     \begin{itemize}

       \item[(1.1)] If $\lambda^{(1)}_{s_b}=\lambda^{(2)}_b$, then we say that  $\lambda^{(2)}_b$ is of type $\hat{A}_1$ and set $\hat{R}_b(\lambda)=s_b$.

    \item[(1.2)] If $\lambda^{(1)}_{s_b}=\lambda^{(2)}_b+2$ and $\hat{R}_{b-1}(\lambda)\neq s_b$, then we say that $\lambda^{(2)}_b$ is of type $\hat{B}$ and set $\hat{R}_b(\lambda)=s_b$.

          \item[(1.3)] If $\lambda^{(1)}_{s_b}=\lambda^{(2)}_b+2$, $\hat{R}_{b-1}(\lambda)= s_b$ and $\lambda^{(1)}_{s_b+1}=\lambda^{(2)}_b-2$, then we say that $\lambda^{(2)}_b$ is of type $\hat{C}$ and set $\hat{R}_{b}(\lambda)=s_b+1$.

               \item[(1.4)] If $\lambda^{(1)}_{s_b}=\lambda^{(2)}_b+2$, $\hat{R}_{b-1}(\lambda)= s_b$ and $\lambda^{(1)}_{s_b+1}=\lambda^{(2)}_b-1$, then we say that $\lambda^{(2)}_b$ is of type $\hat{O}$ and set $\hat{R}_{b}(\lambda)=s_b+1$.

 \end{itemize}

\item[(2)] If there do not exist parts $\lambda^{(2)}_b+2$ with marks greater than $2$ in $GG(\lambda)$, then by the definition of G\"ollnitz-Gordon marking, there is a $1$-marked part $\lambda^{(1)}_{s_b}=\lambda^{(2)}_b-2$ or $\lambda^{(2)}_b-1$ or $\lambda^{(2)}_b$ in $GG(\lambda)$. We set $\hat{R}_l(\lambda)=s_b$. There are three subcases.

\begin{itemize}
    \item[(2.1)] If $\lambda^{(1)}_{s_b}=\lambda^{(2)}_b-2$, then we say that  $\lambda^{(2)}_b$ is of type $\hat{C}$.

    \item[(2.2)] If $\lambda^{(1)}_{s_b}=\lambda^{(2)}_b-1$, then we say that $\lambda^{(2)}_b$ is of type $\hat{O}$.

    \item[(2.3)] If $\lambda^{(1)}_{s_b}=\lambda^{(2)}_b$, then we say that  $\lambda^{(2)}_b$ is of type $\hat{A}_2$.

\end{itemize}
\end{itemize}
For each $1\leq b\leq p$, we say that $\lambda^{(2)}_b$ is of type $\hat{A}$ if $\lambda^{(2)}_b$ it is of type $\hat{A}_1$ or type $\hat{A}_2$.

For example, let $\lambda$ be a partition in $\mathbb{C}_1(10,9,6;3)$, whose G\"ollnitz-Gordon marking is
{\footnotesize \begin{equation}\label{example-reduction-type-1}
GG(\lambda)=\left[
\begin{array}{cccccccccccccccccccc}
 &&4&&&&12&&16&&&22&&&&30&&&36\\
 &2&&6&&10&&14&&&20&&24&&28&&&34&&38\\
1&&4&&8&&12&&16&&20&&24&&28&&&34&&38
\end{array}
\right].
\end{equation}}
 For $p=9$, it is clear that $\lambda^{(2)}_1=38$, $\lambda^{(2)}_2=34$, $\lambda^{(2)}_3=28$, $\lambda^{(2)}_4=24$ and $\lambda^{(2)}_5=20$  are of types  $\hat{A}_2$,  $\hat{A}_1$,  $\hat{A}_1$,  $\hat{A}_2$ and  $\hat{A}_1$, respectively, and so these $2$-marked parts are of type $\hat{A}$. Moreover, we have $\hat{R}_l(\lambda)=1$, $\hat{R}_2(\lambda)=2$,
$\hat{R}_3(\lambda)=3$,
$\hat{R}_4(\lambda)=4$ and 
$\hat{R}_5(\lambda)=5$. 

 There are $1$-marked and $3$-marked parts $16$ in $GG(\lambda)$. Moreover, note that $\lambda^{(1)}_6=16$ and $\hat{R}_5(\lambda)=5\neq6$, we see that $\hat{R}_6(\lambda)=6$ and $\lambda^{(2)}_6=14$ is of type $\hat{B}$. Similarly, we get that $\hat{R}_7(\lambda)=7$ and $\lambda^{(2)}_7=10$ is of type $\hat{B}$.

Note that there do not exist parts $8$ with marks greater than $2$ and $\lambda^{(1)}_9=4$, then we see that  $\hat{R}_8(\lambda)=9$ and $\lambda^{(2)}_8=6$ is of type $\hat{C}$. There are $1$-marked and $3$-marked parts $4$ in $GG(\lambda)$. Moreover, note that $\lambda^{(1)}_{9}=4$, $\lambda^{(1)}_{10}=1$ and $\hat{R}_8(\lambda)=9$, we see that $\hat{R}_9(\lambda)=10$ and $\lambda^{(2)}_9=2$ is of type $\hat{O}$.

Now, we define $\mathbb{\hat{C}}_j(N_1,\ldots,N_{k-1};i|p)$ to be the set of partitions $\lambda$ in  $\mathbb{C}_j(N_1,\ldots,N_{k-1};i)$ such that there do not exist odd parts greater than or equal to $\lambda^{(2)}_p$, and $\lambda^{(2)}_p$ is of type $\hat{A}$ or type $\hat{B}$ or type $\hat{C}$. For example, the partition $\lambda$  in $\eqref{example-reduction-type-1}$ belongs to $\mathbb{\hat{C}}_1(10,9,6;3|8)$.

Let $\lambda$ be a partition in $\mathbb{\hat{C}}_j(N_1,\ldots,N_{k-1};i|p)$, we define the clusters based on the types of $\lambda^{(2)}_1,\lambda^{(2)}_2,\ldots,\lambda^{(2)}_p$. For convenience, we assume that $\lambda^{(2)}_{0}=+\infty$. We start from the part $\lambda^{(2)}_{p}$, assume that $p_1$ is the smallest integer such that $p_1\leq p$, $\lambda^{(2)}_{p_1}-\lambda^{(2)}_{p}=4(p-p_1)$ and satisfying one of the following four conditions:

\begin{itemize}

\item[(1)] $\lambda^{(2)}_{p},\ldots,\lambda^{(2)}_{p_1}$ are of type $\hat{A}_1$ and $\lambda^{(2)}_{p_1-1}>\lambda^{(2)}_{p_1}+4$, we say that the set $\{p_1,\ldots,p\}$ is of cluster $\hat{A}_1$.

\item[(2)] $\lambda^{(2)}_{p},\ldots,\lambda^{(2)}_{p_1}$ are of type $\hat{A}$ and $\lambda^{(2)}_{p_1}$ is of type $\hat{A}_2$, we say that the set $\{p_1,\ldots,p\}$ is of cluster $\hat{A}_2$.

\item[(3)] $\lambda^{(2)}_{p},\ldots,\lambda^{(2)}_{p_1}$ are of type $\hat{B}$, we say that the set $\{p_1,\ldots,p\}$ is of cluster $\hat{B}$.

    \item[(4)] $\lambda^{(2)}_{p},\ldots,\lambda^{(2)}_{p_1}$ are of type $\hat{C}$, we say that the set $\{p_1,\ldots,p\}$ is of cluster $\hat{C}$.

\end{itemize}
Define $\hat{\alpha}_p(\lambda)=\cdots=\hat{\alpha}_{p_1}(\lambda)=\{p_1,\ldots,p\}$. If $p_1>1$, then set $b=1$ and execute the following procedure:
\begin{itemize}
\item[(E)]  Assume that $p_{b+1}$ is the smallest integer such that $p_{b+1}\leq p_b-1$, $\lambda^{(2)}_{p_{b+1}}-\lambda^{(2)}_{p_b-1}=4(p_b-1-p_{b+1})$ and satisfying one of the following five conditions:
    \begin{itemize}
\item[(1)] $\lambda^{(2)}_{p_b-1},\ldots,\lambda^{(2)}_{p_{b+1}}$ are of type $\hat{A}_1$ and $\lambda^{(2)}_{p_{b+1}-1}>\lambda^{(2)}_{p_{b+1}}+4$, we say that the set $\{p_{b+1},\ldots,p_b-1\}$ is of cluster $\hat{A}_1$.

\item[(2)] $\lambda^{(2)}_{p_b-1},\ldots,\lambda^{(2)}_{p_{b+1}}$ are of type $\hat{A}$, $\lambda^{(2)}_{p_{b+1}}$ is of type $\hat{A}_2$, and there is no $1$-marked part $\lambda^{(2)}_{p_b-1}-4$ in $GG(\lambda)$, then we say that the set $\{p_{b+1},\ldots,p_b-1\}$ is of cluster $\hat{A}_2$.

 \item[(3)] $\lambda^{(2)}_{p_b-1},\ldots,\lambda^{(2)}_{p_{b+1}}$ are of type $\hat{A}$, $\lambda^{(2)}_{p_{b+1}}$ is of type $\hat{A}_2$,  and there is a $1$-marked part $\lambda^{(2)}_{p_b-1}-4$  in $GG(\lambda)$, then we say that the set $\{p_{b+1},\ldots,p_b-1\}$ is of cluster $\hat{A}_3$.

\item[(4)] $\lambda^{(2)}_{p_b-1},\ldots,\lambda^{(2)}_{p_{b+1}}$ are of type $\hat{B}$, we say that the set $\{p_{b+1},\ldots,p_b-1\}$  is of cluster $\hat{B}$.

       \item[(5)] $\lambda^{(2)}_{p_b-1},\ldots,\lambda^{(2)}_{p_{b+1}}$ are of type $\hat{C}$, we say that the set $\{p_{b+1},\ldots,p_b-1\}$ is of cluster $\hat{C}$.

\end{itemize}
Define $\hat{\alpha}_{p_b-1}(\lambda)=\cdots=\hat{\alpha}_{p_{b+1}}(\lambda)=\{p_{b+1},\ldots,p_b-1\}$.

\item[(F)] Replace $b$ by $b+1$. If $p_b=1$, then we are done. If $p_b>1$, then go back to (E).    
    
\end{itemize}

For example, let $\lambda$ be the partition in \eqref{example-reduction-type-1}. For $p=8$, we find that the intervals $\{8\}$, $\{6,7\}$, $\{4,5\}$, $\{3\}$ and $\{1,2\}$ are of clusters $\hat{C}$, $\hat{B}$, $\hat{A}_3$, $\hat{A}_1$ and $\hat{A}_2$, respectively. For $p=5$, we derive that the intervals $\{4,5\}$, $\{3\}$ and $\{1,2\}$ are of clusters $\hat{A}_2$, $\hat{A}_1$ and $\hat{A}_2$, respectively.

\subsection{The set $\mathbb{\check{C}}_j(N_1,\ldots,N_{k-1};i|p)$}

Let $\mu$ be a partition in $\mathbb{C}_j(N_1,\ldots,N_{k-1};i)$ such that there exist odd parts in $\mu$. Let $\mu^{(r)}=(\mu^{(r)}_1,\mu^{(r)}_2,\ldots,\mu^{(r)}_{N_r})$ denote $r$-marked parts in $GG(\mu)$, where $\mu^{(r)}_1>\mu^{(r)}_2>\cdots>\mu^{(r)}_{N_r}$ and $1\leq r\leq k-1$. We consider the largest odd part $2t+1$ in $\mu$. We say that $2t+1$ is a reduction part of $\mu$ if $2t+1$  is not a insertion part of $\mu$, namely, $\mu$ satisfies either of the following two conditions:

\begin{itemize}
\item[(1)] The mark of $2t+1$ is greater than $1$ in $GG(\mu)$.

 \item[(2)] $\mu^{(1)}_s=2t+1$, $2t+2$ occurs in $\mu$, and  $2t+4b+2$ occurs in $\mu$ if $\mu^{(1)}_{s-b}=2t+4b$, where $1\leq b<s\leq N_1$.

\end{itemize}

By definition, we see that if $2t+1$ is a reduction part in $\mu$, then there exists a $2$-marked part $\mu^{(2)}_{p}=2t$ or $2t+1$ or $2t+2$ in $GG(\mu)$. We say that the reduction index of $2t+1$ is $p$.

We define $\mathbb{\check{C}}_j(N_1,\ldots,N_{k-1};i|p)$ to be the set of partitions $\mu$ in  $\mathbb{C}_j(N_1,\ldots,N_{k-1};i)$ such that there exist odd parts in $\mu$  and the largest odd part is a reduction part with the reduction index $p$.

For example, let $\mu$ be a partition in $\mathbb{C}_1(10,9,6;3)$, whose G\"ollnitz-Gordon marking is
{\footnotesize \begin{equation}\label{example-reduction-type-2}
GG(\mu)=\left[
\begin{array}{cccccccccccccccccccc}
 &&4&&&10&&14&&&&22&&&28&&&&36\\
 &2&&5&&10&&14&&18&&22&&&28&&&34&&38\\
1&&4&&8&&12&&16&&20&&24&&28&&32&&36&
\end{array}
\right].
\end{equation}}
We see that $\lambda^{(2)}_8=5$ is the largest odd part in $\mu$. So we have $\mu\in \mathbb{\check{C}}_1(10,9,6;3|8)$.

For another example, let $\mu$ be a partition in $\mathbb{C}_1(10,9,6;3)$,  whose G\"ollnitz-Gordon marking is
{\footnotesize \begin{equation}\label{example-reduction-type-3}
GG(\mu)=\left[
\begin{array}{cccccccccccccccccccc}
 &&4&&&&12&&16&&&22&&&28&&&&36\\
 &2&&6&&10&&14&&&20&&24&&28&&&34&&38\\
1&&4&&8&&12&&16&19&&22&&&28&&32&&36&
\end{array}
\right].
\end{equation}}
We see that $\mu^{(1)}_5=19$ is the largest odd part in $\mu$. Observe that $\mu^{(1)}_4=22$, $\mu^{(1)}_3=28>26$, and $20$ and $24$ occur in $\mu$, so $19$ is a reduction of $\mu$. Moreover, note that $\mu^{(2)}_5=20$, we have $\mu\in \mathbb{\check{C}}_1(10,9,6;3|5)$.

It is clear that for $1< p\leq N_2$,
\[\mathbb{\check{C}}_j(N_1,\ldots,N_{k-1};i|p)\subseteq \mathbb{\hat{C}}_j(N_1,\ldots,N_{k-1};i|p-1).\]

 Let $\mu$ be a partition in $\mathbb{\check{C}}_j(N_1,\ldots,N_{k-1};i|p)$. We will define the types of $\mu^{(2)}_1,\mu^{(2)}_2,\ldots,$ $\mu^{(2)}_p$ and the clusters related to  $\mu^{(2)}_1,\mu^{(2)}_2,\ldots,\mu^{(2)}_p$. We first define the type of the part $\mu^{(2)}_p$. There are four cases.
\begin{itemize}
\item[(1)] $\mu^{(2)}_p$ is even and there is a part $\mu^{(2)}_p+1$ with mark greater than $2$ in $GG(\mu)$. Then, $\mu^{(2)}_p+1$ is the reduction part of $\mu$. We  say that $\mu^{(2)}_p$ is of type $\check{A}_1$.

    \item[(2)] $\mu^{(2)}_p$ is even and there is no  $\mu^{(2)}_p+1$ in $\mu$. Then, there is a $1$-marked part $\mu^{(2)}_p-1$  in  $GG(\mu)$, which is the reduction part of $\mu$. We say that $\mu^{(2)}_p$ is of type $\check{A}_2$.

       \item[(3)] $\mu^{(2)}_p$ is even and there is a $1$-marked part $\mu^{(2)}_p+1$ in $GG(\mu)$. Then, $\mu^{(2)}_p+1$ is the reduction part of $\mu$. We say that $\mu^{(2)}_p$ is of type $\check{B}$.

     \item[(4)] $\mu^{(2)}_p$ is odd. Then, $\mu^{(2)}_p$ is the reduction part of $\mu$. We say that $\mu^{(2)}_p$ is of type $\check{C}$.

\end{itemize}

Assume that $p_1$ is the smallest integer such that $p_1\leq p$, $\mu^{(2)}_{s}=\mu^{(2)}_p+4(p-s)$ (resp. $\mu^{(2)}_p+4(p-s)-1$) if  $\mu^{(2)}_{p}$ is an even part (resp. an odd part) for $p_1\leq s\leq p$, and  satisfying one of the following conditions:

\begin{itemize}
\item[(1)] $\mu^{(2)}_p$ is of type $\check{A}_1$ and there is a $3$-marked part $\mu^{(2)}_s$ or $\mu^{(2)}_s+1$ in $GG(\mu)$ for $p_1\leq s\leq p$.  We say that the set $\{p_1,\ldots,p\}$ is of cluster $\check{A}_1$ and $\mu^{(2)}_{p-1},\ldots,\mu^{(2)}_{p_1}$ are of type $\check{A}_1$.

\item[(2)] $\mu^{(2)}_p$ is of type $\check{A}_2$ and there is a $1$-marked part $\mu^{(2)}_{p_1}-2$ or $\mu^{(2)}_{p_1}-1$ in $GG(\mu)$.  We say that the set $\{p_1,\ldots,p\}$ is of cluster $\check{A}_2$ and $\mu^{(2)}_{p-1},\ldots,\mu^{(2)}_{p_1}$ are of type $\check{A}_2$.

\item[(3)] $\mu^{(2)}_p$ is of type $\check{B}$ and there exist parts $\mu^{(2)}_l+2$ in $\mu$ for $p_1\leq l\leq p$.  We say that the set $\{p_1,\ldots,p\}$ is of cluster $\check{B}$ and $\mu^{(2)}_{p-1},\ldots,\mu^{(2)}_{p_1}$ are of type $\check{B}$.

\item[(4)] $\mu^{(2)}_p$ is of type $\check{C}$.  We say that the set $\{p_1,\ldots,p\}$ is of cluster $\check{C}$ and $\mu^{(2)}_{p-1},\ldots,\mu^{(2)}_{p_1}$ are of type $\check{C}$.

\end{itemize}
Define $\check{\beta}_p(\mu)=\cdots=\check{\beta}_{p_1}(\mu)=\{p_1,\ldots,p\}$. If $p_1>1$, then set $b=1$ and carry out the following procedure:

\begin{itemize}
    
    \item[(E)]  Assume that $p_{b+1}$ is the smallest integer such that $p_{b+1}\leq p_b-1$, $\mu^{(2)}_{p_{b+1}}-\mu^{(2)}_{p_b-1}=4(p_b-1-p_{b+1})$, and satisfying one of the following conditions:
    \begin{itemize}

\item[(1)] There are $1$-marked and $3$-marked parts $\mu^{(2)}_s$ in $GG(\mu)$ for $p_{b+1}\leq s\leq p_b-1$. We say that the set $\{p_{b+1},\ldots,p_b-1\}$ is of cluster $\check{A}_1$ and $\mu^{(2)}_{p_b-1},\ldots,\mu^{(2)}_{p_{b+1}}$ are of type $\check{A}_1$.

\item[(2)] There is a $1$-marked part $\mu^{(2)}_s-2$ in $GG(\mu)$ for $p_{b+1}\leq s\leq p_b-1$, and  $\mu^{(2)}_{p_{b+1}}+2$ does not occur in $\mu$. We say that the set $\{p_{b+1},\ldots,p_b-1\}$ is of cluster $\check{A}_2$ and $\mu^{(2)}_{p_b-1},\ldots,\mu^{(2)}_{p_{b+1}}$ are of type $\check{A}_2$.

\item[(3)] There is a $1$-marked part $\mu^{(2)}_s-2$ in $GG(\mu)$ for $p_{b+1}\leq s\leq p_b-1$,   $\mu^{(2)}_{p_{b+1}}+2$ occurs in $\mu$, and $\mu^{(2)}_{p_b-1}$ appears exactly once in $\mu$. We say that the set $\{p_{b+1},\ldots,p_b-1\}$ is of cluster $\check{A}_3$ and $\mu^{(2)}_{p_b-1},\ldots,\mu^{(2)}_{p_{b+1}}$ are of type $\check{A}_3$.

\item[(4)] There is a $1$-marked part $\mu^{(2)}_s-2$ in $GG(\mu)$ and $\mu^{(2)}_{s}$ appears at least twice in $\mu$ for $p_{b+1}\leq s\leq p_b-1$, and   $\mu^{(2)}_{p_{b+1}}+2$ occurs in $\mu$. We say that the set $\{p_{b+1},\ldots,p_b-1\}$ is of cluster $\check{B}$ and $\mu^{(2)}_{p_b-1},\ldots,\mu^{(2)}_{p_{b+1}}$ are of type $\check{B}$.

\item[(5)] There is a $1$-marked part $\mu^{(2)}_s$ in $GG(\mu)$ for $p_{b+1}\leq s\leq p_b-1$, and $\mu^{(2)}_{p_b-1}$ appears exactly twice in $\mu$. We say that the set $\{p_{b+1},\ldots,p_b-1\}$ is of cluster $\check{C}$ and $\mu^{(2)}_{p_b-1},\ldots,\mu^{(2)}_{p_{b+1}}$ are of type $\check{C}$.

\end{itemize}
Define $\check{\beta}_{p_b-1}(\mu)=\cdots=\check{\beta}_{p_{b+1}}(\mu)=\{p_{b+1},\ldots,p_b-1\}$.

\item[(F)] Replace $b$ by $b+1$. If $p_b=1$, then we are done. If $p_b>1$, then go back to (E).
    
\end{itemize}

For example, let $\mu$ be the partition in \eqref{example-reduction-type-2}. We find that the intervals $\{8\}$, $\{6,7\}$, $\{4,5\}$, $\{3\}$ and $\{1,2\}$ are of clusters $\check{C}$, $\check{B}$, $\check{A}_3$, $\check{A}_1$ and $\check{A}_2$, respectively.

For another example, let $\mu$ be the partition in \eqref{example-reduction-type-3}, we derive that the intervals $\{4,5\}$, $\{3\}$ and $\{1,2\}$ are of clusters $\check{A}_2$, $\check{A}_1$ and $\check{A}_2$, respectively.

\subsection{Special partitions}

In order to give a description of the reduction $R_p$ between  $\mathbb{\hat{C}}_j(N_1,\ldots,N_{k-1};i|p)$ and $\mathbb{\check{C}}_j(N_1,\ldots,N_{k-1};i|p)$, we need to introduce the special partitions. A special partition $\lambda$ is an ordinary partition in which 
the largest odd part in $\lambda$ may be overlined. We generalize the definition of the G\"ollnitz-Gordon marking of an ordinary partition to a special partition as follows.

\begin{defi}\label{defi-mark-special}
The G\"ollnitz-Gordon  marking  of a special partition $\lambda$, denoted by $\overline{GG}(\lambda)$, is an assignment of positive integers {\rm(}marks{\rm)} to the parts of  $\lambda=(\lambda_1,\lambda_2,\ldots,\lambda_\ell)$ from  smallest to  largest such that the marks are as small as possible subject to the conditions that for $1\leq s\leq \ell$\rm{,}

\begin{itemize}
\item[{\rm (1)}] the integer assigned to $\lambda_s$ is different from the integers assigned to the parts $\lambda_g$ such that $\lambda_s-\lambda_g\leq 2$ with strict inequality if $\lambda_s$ is an odd part for $g>s$\rm{;}

 \item[{\rm (2)}]  $\lambda_s$ can not be assigned with $1$ if $\lambda_s$ is a special odd part.

 \end{itemize}
\end{defi}

\begin{defi}
We define $\mathbb{\overleftarrow{C}}_j(k;i)$ to be the set of special partitions $\lambda$ such that 
\begin{itemize}

 \item[\rm{(1)}] $f_1(\lambda)+f_2(\lambda)\leq i-1${\rm{;}}
 
  \item[\rm{(2)}] $f_{2t+1}(\lambda)+f_{\overline{2t+1}}(\lambda)\leq 1${\rm{;}}
  
  \item[\rm{(3)}] $f_{2t}(\lambda)+f_{2t+1}(\lambda)+f_{\overline{2t+1}}(\lambda)+f_{2t+2}(\lambda)\leq k-1${\rm{;}}
  
  \item[{\rm (4)}] if $f_{2t}(\lambda)+f_{2t+1}(\lambda)+f_{2t+2}(\lambda)=k-1$, then $tf_{2t}(\lambda)+tf_{2t+1}(\lambda)+(t+1)f_{2t+2}(\lambda)\equiv {O}_\lambda(2t+2)+i-1\pmod{2-j},$ where ${O}_\lambda(N)$ denotes the number of odd parts not exceeding $N$ in $\lambda${\rm{;}}

\item[\rm{(5)}] if there is a  $\overline{2s+1}$ in  $\lambda$, then  there are $1$-marked $2s-2$ and $2s+2$ in $\overline{GG}(\lambda)${\rm{.}}

\end{itemize}

\end{defi}

Let $N_r(\lambda)$ (or $N_r$ for short) denote the number of $r$-marked parts in $\overline{GG}(\lambda)$. If $\lambda$ is a partition in $\mathbb{\overleftarrow{C}}_j(k;i)$, then  it is clear that  $N_1\geq N_2\geq\cdots\geq N_{k-1}\geq 0$. Let $\overleftarrow{C}_j(N_1,\ldots,N_{k-1};i)$ denote the set of special partitions $\lambda$ counted by $\overleftarrow{C}_j(k,i)$ such that there are $N_r$ $r$-marked parts in $\overline{GG}(\lambda)$ for $1\leq r\leq k-1$.

We also generalize the definition of $(k-1)$-bands of an ordinary partition in $\mathbb{C}_1(k,i)$ to a special partition in $\mathbb{\overleftarrow{C}}_1(k;i)$.  For the consecutive $k-1$ parts $\lambda_s, \lambda_{s+1}, \ldots, \lambda_{s+k-2}$ of a special partition $\lambda\in\mathbb{\overleftarrow{C}}_1(k;i)$, if such $k-1$ parts satisfy the difference condition
\begin{equation}\label{condition-special-1}
\lambda_s\leq\lambda_{s+k-2}+2 \text{ with strict inequality if $\lambda_s$ is odd},
\end{equation}
then we say that $\{\lambda_{s+l}\}_{0\leq l\leq k-2}$ is a $(k-1)$-band of $\lambda$.  If the $k-1$ parts in the $(k-1)$-band $\{\lambda_{s+l}\}_{0\leq l\leq k-2}$ also satisfy the congruence
\begin{equation}\label{difference-special-1}
[\lambda_s/2]+\cdots+[\lambda_{s+k-2}/2]\equiv i-1+{O}_\lambda(\lambda_s)\pmod{2},
\end{equation}
then we say that the $(k-1)$-band $\{\lambda_{s+l}\}_{0\leq l\leq k-2}$ is   good. Otherwise, we say that the $(k-1)$-band $\{\lambda_{s+l}\}_{0\leq l\leq k-2}$ is bad.  Then  $\lambda$ is a special partition in $\mathbb{\overleftarrow{C}}_0(k;i)$ if and only if all the $(k-1)$-bands of $\lambda$ are good. With the similar argument in the proof of Lemma \ref{parity k-1 sequence-over}, we can obtain the following lemma.

   \begin{lem}\label{parity k-1 sequence-over-special} 
   Let $\{\lambda_{s+l}\}_{0\leq l\leq k-2}$ and $\{\lambda_{g+l}\}_{0\leq l\leq k-2}$  be two $(k-1)$-bands of a special partition $\lambda\in\mathbb{\overleftarrow{C}}_1(k;i)$. If $\lambda_g>\lambda_s$ and $\lambda_{s+k-2}\geq \lambda_g-4$ with strict inequality if $\lambda_g$ is odd, then $\{\lambda_{s+l}\}_{0\leq l\leq k-2}$ and $\{\lambda_{g+l}\}_{0\leq l\leq k-2}$ are both good or bad.
   \end{lem}

Let $\lambda$ be a special partition in $\mathbb{\overleftarrow{C}}_1(k;i)$, assume that there are $N_{k-1}$ $(k-1)$-marked parts in $\overline{GG}(\lambda)$, denoted  $\lambda^{(k-1)}_{1}> \lambda^{(k-1)}_{2}>\cdots >\lambda^{(k-1)}_{N_{k-1}}$. For each $p\in[1,N_{k-1}]$, assume that $\lambda^{(k-1)}_{p}$ is the $s_p$-th part $\lambda_{s_p}$ of $\lambda$.  If $\lambda^{(k-1)}_{p}$ is not a special odd part,  then   there  exist $k-2$ parts $\lambda_g$  in $\lambda$ such that $g>s$ and $\lambda_g\geq \lambda_s-2$ with strict inequality if $\lambda_s$ is an odd part. Then, $\{\lambda_{s+l}\}_{0\leq l\leq k-2}$ is a $(k-1)$-band of $\lambda$. Such a $(k-1)$-band is called the $(k-1)$-band induced by $\lambda^{(k-1)}_{p}$, denoted $\{\lambda^{(k-1)}_{p}\}_{k-1}=\{\lambda^{(k-1)}_{p,1},\lambda^{(k-1)}_{p,2}, \ldots,\lambda^{(k-1)}_{p,k-1}\},$ where $\lambda^{(k-1)}_{p,1}\geq\lambda^{(k-1)}_{p,2}\geq \cdots\geq\lambda^{(k-1)}_{p,k-1}$.

If $\lambda^{(k-1)}_{p}$ is a special odd part, set $\lambda^{(k-1)}_{p}=\overline{2t+1}$, then   there exist $2$-marked, \ldots, $(k-2)$-marked parts $2t$  and a $1$-marked part $2t+2$ in $\overline{GG}(\lambda)$. We define the $(k-1)$-band induced by $\lambda^{(k-1)}_{p}$ as $\{\lambda^{(k-1)}_{p}\}_{k-1}=\{\lambda^{(k-1)}_{p,1},\lambda^{(k-1)}_{p,2}, \ldots,\lambda^{(k-1)}_{p,k-1}\},$ where  $\lambda^{(k-1)}_{p,1}=2t+2$, $\lambda^{(k-1)}_{p,2}=\overline{2t+1}$ and $\lambda^{(k-1)}_{p,3}=\cdots= \lambda^{(k-1)}_{p,k-1}=2t$.

Clearly, the $k-1$ parts in the $(k-1)$-band $\{\lambda^{(k-1)}_{p}\}_{k-1}$ satisfy the difference condition \eqref{condition-special-1}. With the similar argument in the proof of Theorem \ref{parity k-1 sequence}, we  obtain the following theorem, which gives a criterion  to  determine whether a special partition $\lambda$ in $\mathbb{\overleftarrow{C}}_1(k;i)$ also belongs to $\mathbb{\overleftarrow{C}}_0(k;i)$.

\begin{thm}\label{parity k-1 sequence-special} Let $\lambda$ be a special partition in $\mathbb{\overleftarrow{C}}_1(k;i)$ with $N_{k-1}$ $(k-1)$-marked parts in $\overline{GG}(\lambda)$, denoted   $\lambda^{(k-1)}_{1}> \lambda^{(k-1)}_{2}>\cdots >\lambda^{(k-1)}_{N_{k-1}}$. Then $\lambda$ is a special  partition in $\mathbb{\overleftarrow{C}}_0(k;i)$ if and only if $\{\lambda^{(k-1)}_{p}\}_{k-1}$ is  good for $1\leq p\leq N_{k-1}$.
\end{thm}

\subsection{Proof of Theorem \ref{reduction-1}}

We are now in a position to give a proof a Theorem \ref{reduction-1}.

{\noindent \bf Proof of Theorem \ref{reduction-1}.} Let $\lambda$ be a partition in $\mathbb{\hat{C}}_1(N_1,\ldots,N_{k-1};i|p)$.  Let $\lambda^{(r)}=(\lambda^{(r)}_1,\lambda^{(r)}_2,\ldots,\lambda^{(r)}_{N_r})$ denote the $r$-marked parts in  ${GG}(\lambda)$, where $1\leq r\leq k-1$. We will define $\mu=R_p(\lambda)$ related to $\lambda^{(2)}_1, \lambda^{(2)}_2,\ldots,\lambda^{(2)}_p$, and denote the  intermediate special partitions by $\lambda^{1},\lambda^{2},\ldots,\lambda^{p}$.

\begin{itemize}

\item[] Step 1: We do the following operation related to $\lambda^{(2)}_1$, called the basic operation of $\lambda^{(2)}_1$. There are  five cases.

          \begin{itemize}
              \item[] Case 1: $\hat{\alpha}_1(\lambda)$ is of cluster $\hat{A}_1$. We may write         $\lambda^{(2)}_1=2t_1$. In this case,  $\lambda^{(2)}_1$ is of type  $\hat{A}_1$. So, there is a $1$-marked part $2t_1$ in ${GG}(\lambda)$ and there exist parts $2t_1+2$ in $\lambda$. Let $r_1$ be the smallest mark of parts  $2t_1+2$ in ${GG}(\lambda)$. We have $r_1>2$. Then we replace the $r_1$-marked $2t_1+2$ in ${GG}(\lambda)$ by $2t_1+1$ to get $\lambda^1$.

                  \item[] Case 2: $\hat{\alpha}_1(\lambda)$ is of cluster $\hat{A}_2$. We may write $\lambda^{(2)}_1=2t_1+2$. In this case,  $\lambda^{(2)}_1$ is of type  $\hat{A}_2$. So, there is a $1$-marked part $2t_1+2$ in ${GG}(\lambda)$ and there do not exist parts $2t_1+4$ in $\lambda$. Then we set $r_1=1$ and replace the $1$-marked $2t_1+2$ in ${GG}(\lambda)$ by $2t_1+1$ to obtain $\lambda^1$.

                        \item[] Case 3: $\hat{\alpha}_1(\lambda)$ is of cluster $\hat{A}_3$. We find that $\hat{\alpha}_1(\lambda)\neq \hat{\alpha}_p(\lambda)$ and   $\lambda^{(2)}_1$ is of type  $\hat{A}_2$. Set $\lambda^{(2)}_1=2t_1+2$. We see that there is a $1$-marked part $2t_1+2$ in ${GG}(\lambda)$ and there do not exist parts $2t_1+4$ in $\lambda$.   Then we set $r_1=2$ and replace the $2$-marked $2t_1+2$ in ${GG}(\lambda)$ by $\overline{2t_1+1}$ to obtain $\lambda^1$.

                      \item[] Case 4: $\hat{\alpha}_1(\lambda)$ is of cluster $\hat{B}$. Then, $\lambda^{(2)}_1$ is of type  $\hat{B}$. Set $\lambda^{(2)}_1=2t_1$.  We see that  there exist parts $2t_1+2$ with mark $1$ and  marks  greater than $2$ in ${GG}(\lambda)$. There are  two subcases.

                             \begin{itemize}
                    \item[] Case 4.1: $\hat{\alpha}_1(\lambda)=\hat{\alpha}_p(\lambda)$. Then we set $r_1=1$ and replace the $1$-marked $2t_1+2$ in ${GG}(\lambda)$ by $2t_1+1$ to get $\lambda^1$.

                     \item[] Case 4.2: $\hat{\alpha}_1(\lambda)\neq\hat{\alpha}_p(\lambda)$. Let $r_1$ be the smallest mark except for $1$ of parts  $2t_1+2$ in ${GG}(\lambda)$. We have $r_1>2$. Then,  the $\lambda^1$ is obtained by replacing the $r_1$-marked $2t_1+2$ in ${GG}(\lambda)$ by $\overline{2t_1+1}$.
                             \end{itemize}

                     \item[] Case 5: $\hat{\alpha}_1(\lambda)$ is of cluster $\hat{C}$. We may write $\lambda^{(2)}_1=2t_1+2$. In this case, $\lambda^{(2)}_1$ is of type  $\hat{C}$. So, there is a $1$-marked part $2t_1$   and there do not exist parts $2t_1+4$ with marks greater than $2$ in ${GG}(\lambda)$. Then we set $r_1=2$ and replace the $2$-marked $2t_1+2$ in ${GG}(\lambda)$ by $2t_1+1$ to get $\lambda^1$.

          \end{itemize}
 Obviously, $|\lambda^1|=|\lambda|-1$. We see that $\lambda^{1}$ is obtained by  replacing the $r_1$-marked $2t_1+2$ in ${GG}(\lambda)$ by $2t_1+1$ (resp. $\overline{2t_1+1}$). We proceed to prove that the new generated part $2t_1+1$ (resp. $\overline{2t_1+1}$)  is  marked with $r_1$ in $\overline{GG}(\lambda^{1})$. It suffice to show that there are $2$-marked, \ldots, $(r_1-1)$-marked parts $2t_1$ and there is a $1$-marked $2t_1$ (resp. ${2t_1+2}$) in ${GG}(\lambda)$.

 If $\lambda^{1}$ is obtained by  replacing the $r_1$-marked $2t_1+2$ in ${GG}(\lambda)$ by ${2t_1+1}$, then  $r_1$ is the smallest mark of parts $2t_1+2$  in ${GG}(\lambda)$, and so there are $1$-marked, \ldots, $(r_1-1)$-marked parts $2t_1$ in ${GG}(\lambda)$.

 If $\lambda^{1}$ is obtained by  replacing the $r_1$-marked $2t_1+2$ in ${GG}(\lambda)$ by $\overline{2t_1+1}$, then there is a $1$-marked $2t_1+2$ and  $r_1$ is the smallest mark except for $1$ of parts  $2t_1+2$ in ${GG}(\lambda)$, and so there are $2$-marked, \ldots, $(r_1-1)$-marked $2t_1$ in ${GG}(\lambda)$.

We conclude that  there are $2$-marked, \ldots, $(r_1-1)$-marked $2t_1$ and there is a $1$-marked $2t_1$ (resp. ${2t_1+2}$) in ${GG}(\lambda)$. So, the new generated part $2t_1+1$ (resp. $\overline{2t_1+1}$)   is marked with $r_1$ in $\overline{GG}(\lambda^{1})$. Moreover, the marks of the unchanged parts in $\overline{GG}(\lambda^{1})$ are the same as those in ${GG}(\lambda)$. Hence,   $\lambda^{1}$ is a special partition in $\mathbb{\overleftarrow{C}}_1(N_1,\ldots,N_{k-1};i)$.

   We turn to show that if $\lambda\in\mathbb{{C}}_0(N_1,\ldots,N_{k-1};i)$, then $\lambda^{1}\in\mathbb{\overleftarrow{C}}_0(N_1,\ldots,N_{k-1};i)$. Assume that $\lambda\in\mathbb{{C}}_0(N_1,\ldots,N_{k-1};i)$, then  $\{\lambda^{(k-1)}_s\}_{k-1}$ is  good  for $1\leq s\leq N_{k-1}$.
    Let   $\lambda^{1,(k-1)}=(\lambda^{1,(k-1)}_1,\lambda^{1,(k-1)}_2,\ldots, \lambda^{1,(k-1)}_{N_{k-1}})$ be the $(k-1)$-marked parts in  $\overline{GG}(\lambda^{1})$. From the proof above, we have $\lambda^{1,(2)}_{r_1}=2t_1+1$ (resp. $\overline{2t_1+1}$). 
   For each $s\in[1,N_{k-1}]$, we find that   $\lambda^{1,(k-1)}_s\leq 2t_1+2$. To show that $\{\lambda^{1,(k-1)}_s\}_{k-1}$ is  good, we consider the following four cases.
\begin{itemize}
\item[(1)] $\lambda^{1,(k-1)}_s\leq 2t_1$. By the construction of $\lambda^1$, we have $\{\lambda^{1,(k-1)}_s\}_{k-1}=\{\lambda^{(k-1)}_s\}_{k-1}$ and $O_{\lambda^1}(\lambda^{1,(k-1)}_s)=O_{\lambda}(\lambda^{(k-1)}_s)$. Note that $\{\lambda^{(k-1)}_s\}_{k-1}$ is  good, so $\{\lambda^{1,(k-1)}_s\}_{k-1}$ is  good.

    \item[(2)] $\lambda^{1,(k-1)}_s= 2t_1+1$. In this case, we have $r_1=k-1$ and $\lambda^{(k-1)}_s=2t_1+2$.  Moreover, we derive that  $\lambda^{1,(k-1)}_{s,2}=\cdots=\lambda^{1,(k-1)}_{s,k-1}=\lambda^{(k-1)}_{s,2}=\cdots=\lambda^{(k-1)}_{s,k-1}=2t_1$
     and $O_{\lambda^{1}}(2t_1+1)=O_\lambda(2t+2)+1$. So, we get
     \[
\begin{split}
&[\lambda^{1,(k-1)}_{s,1}/2]+[\lambda^{1,(k-1)}_{s,2}/2]+\cdots+[\lambda^{1,(k-1)}_{s,k-1}/2]\\
&\qquad=t_1f_{2t_1}(\lambda^{1})+t_1f_{2t_1+1}(\lambda^{1})=t_1(k-2)+t_1\\
&\qquad=t_1(k-2)+(t_1+1)-1=t_1f_{2t_1}(\lambda)+(t_1+1)f_{2t_1+2}(\lambda)-1\\
&\qquad=[\lambda^{(k-1)}_{s,1}/2]+[\lambda^{(k-1)}_{s,2}/2]+\cdots+[\lambda^{(k-1)}_{s,k-1}/2]-1\\
&\qquad \equiv i-1+O_\lambda(\lambda^{(k-1)}_{s,1})-1=i-1+O_\lambda(2t_1+2)-1\\
&\qquad \equiv i-1+O_{\lambda^1}(2t_1+1)=i-1+O_{\lambda^1}(\lambda^{1,(k-1)}_{s,1})\pmod2,
\end{split}
\]
which implies that  $\{\lambda^{1,(k-1)}_s\}_{k-1}$ is good.

     \item[(3)] $\lambda^{1,(k-1)}_s= \overline{2t_1+1}$. In this case, we have $r_1=k-1$ and $\lambda^{(k-1)}_s=2t_1+2$.  Moreover, we find that $\lambda^{1,(k-1)}_{s,1}=\lambda^{1,(k-1)}_{s,2}=2t_1+2$, $\lambda^{1,(k-1)}_{s,3}=\cdots=\lambda^{1,(k-1)}_{s,k-1}=\lambda^{(k-1)}_{s,3}=\cdots=\lambda^{(k-1)}_{s,k-1}=2t_1$ and $O_{\lambda^{1}}(2t_1+2)=O_\lambda(2t+2)+1$. We obtain that
     \[
\begin{split}
&[\lambda^{1,(k-1)}_{s,1}/2]+[\lambda^{1,(k-1)}_{s,2}/2]+\cdots+[\lambda^{1,(k-1)}_{s,k-1}/2]\\
&\qquad =t_1f_{2t_1}(\lambda^{1})+t_1f_{\overline{2t_1+1}}(\lambda^{1})+(t_1+1)f_{2t_1+2}(\lambda^{1})=t_1(k-3)+t_1+(t_1+1)\\
&\qquad=t_1(k-3)+2(t_1+1)-1=t_1f_{2t_1}(\lambda)+(t_1+1)f_{2t_1+2}(\lambda)-1\\
&\qquad =[\lambda^{(k-1)}_{s,1}/2]+[\lambda^{(k-1)}_{s,2}/2]+\cdots+[\lambda^{(k-1)}_{s,k-1}/2]-1\\
&\qquad \equiv i-1+O_\lambda(\lambda^{(k-1)}_{s,1})-1=i-1+O_\lambda(2t_1+2)-1\\
&\qquad \equiv i-1+O_{\lambda^1}(2t_1+2)=i-1+O_{\lambda^1}(\lambda^{1,(k-1)}_{s,1})\pmod2.
\end{split}
\]
So, $\{\lambda^{1,(k-1)}_s\}_{k-1}$ is  good.

     \item[(4)] $\lambda^{1,(k-1)}_s= 2t_1+2$. By the construction of $\lambda^{1}$, we find that $\lambda^{(k-1)}_s=2t_1+2$, $f_{2t_1}(\lambda^{1})=f_{2t_1}(\lambda)$, $f_{2t_1+1}(\lambda^{1})+f_{\overline{2t_1+1}}(\lambda^{1})=f_{2t_1+1}(\lambda)+f_{\overline{2t_1+1}}(\lambda)+1$, $f_{2t_1+2}(\lambda^{1})=f_{2t_1+2}(\lambda)-1$ and $O_{\lambda^{1}}(2t_1+2)=O_\lambda(2t+2)+1$, and so
        \[
\begin{split}
&[\lambda^{1,(k-1)}_{s,1}/2]+[\lambda^{1,(k-1)}_{s,2}/2]+\cdots+[\lambda^{1,(k-1)}_{s,k-1}/2]\\
&\qquad =t_1f_{2t_1}(\lambda^{1})+t_1(f_{2t_1+1}(\lambda^{1})+f_{\overline{2t_1+1}}(\lambda^{1}))+(t_1+1)f_{2t_1+2}(\lambda^{1})\\
&\qquad =t_1f_{2t_1}(\lambda)+t_1(f_{2t_1+1}(\lambda)+f_{\overline{2t_1+1}}(\lambda)+1)+(t_1+1)(f_{2t_1+2}(\lambda)-1)\\
&\qquad =t_1f_{2t_1}(\lambda)+t_1(f_{2t_1+1}(\lambda)+f_{\overline{2t_1+1}}(\lambda))+(t_1+1)f_{2t_1+2}(\lambda)-1\\
&\qquad =[\lambda^{(k-1)}_{s,1}/2]+[\lambda^{(k-1)}_{s,2}/2]+\cdots+[\lambda^{(k-1)}_{s,k-1}/2]-1\\
&\qquad \equiv i-1+O_\lambda(\lambda^{(k-1)}_{s,1})-1=i-1+O_\lambda(2t_1+2)-1\\
&\qquad \equiv i-1+O_{\lambda^{1}}(2t_1+2)=i-1+O_{\lambda^{1}}(\lambda^{1,(k-1)}_{s,1})\pmod2.
\end{split}
\]
Therefore, $\{\lambda^{1,(k-1)}_s\}_{k-1}$ is good.

\end{itemize}
Hence, the $(k-1)$-band $\{\lambda^{1,(k-1)}_s\}_{k-1}$ is good for $s\in[1,N_{k-1}]$. It follows from Theorem \ref{parity k-1 sequence-special} that
$\lambda^1$ is a special partition in $\mathbb{\overleftarrow{C}}_0(N_1,\ldots,N_{k-1};i)$.

\item[] Step 2: If $p=1$, then we shall do nothing. If $p\geq2$, then set $b=1$ and repeat the following process.

\begin{itemize}
    \item[(E)]  Replace the $r_{b}$-marked $2t_{b}+1$ (or $\overline{2t_{b}+1}$) in $\overline{GG}(\lambda^{b})$ by $2t_{b}$ and apply the basic operation defined in  Step 1 to $\lambda^{(2)}_{b+1}$    to get $\lambda^{b+1}$.
    
     \item[(F)] replace $b$ by $b+1$. If $b=p$, then we are done. If $b<p$, then go back to (F).
\end{itemize}

    Clearly, for $1\leq b<p$, we have  $|\lambda^{b+1}|=|\lambda^{b}|-2$. Moreover, $\lambda^{b+1}$ is obtained by  replacing the $r_{b}$-marked $2t_{b}+1$ (resp. $\overline{2t_{b}+1}$) in $\overline{GG}(\lambda^{b})$ by $2t_{b}$ and replacing  the $r_{b+1}$-marked $2t_{b+1}+2$   in $\overline{GG}(\lambda^{b})$ by
    $2t_{b+1}+1$ or $\overline{2t_{b+1}+1}$.
      With the similar argument in Step 1, we see that the new generated part $2t_{b+1}+1$ (resp. $\overline{2t_{b+1}+1}$) is marked with $r_{b+1}$ in $\overline{GG}(\lambda^{b+1})$.

    Now, we proceed to show that the new generated part $2t_{b}$ is marked with $r_{b}$ in $\overline{GG}(\lambda^{b+1})$. If there does not exist a $r_{b}$-marked $2t_{b}-2$ in  $\overline{GG}(\lambda^{b})$, then it is obviously right. If there is a $r_{b}$-marked $2t_{b}-2$ in  $\overline{GG}(\lambda^{b})$, then we have $\hat{\alpha}_{b+1}(\lambda)=\hat{\alpha}_{b}(\lambda)$ and $r_{b+1}=r_{b}$. By the construction of $\lambda^{b+1}$, we see that we will replace the $r_{b}$-marked $2t_{b}-2$ in $\overline{GG}(\lambda^{b})$ by $2t_{b}-3$ (resp. $\overline{2t_{b}-3}$). So, the new generated part $2t_{b}$  is  marked with $r_{b}$ in $\overline{GG}(\lambda^{b+1})$. 

Furthermore, the marks of the unchanged parts in $\overline{GG}(\lambda^{b+1})$ are the same as those in $\overline{GG}(\lambda^{b})$. Hence, $\lambda^{b+1}$ is a special partition in $\mathbb{\overleftarrow{C}}_1(N_1,\ldots,N_{k-1};i)$.

   We turn to show that if $\lambda^{b}\in\mathbb{\overleftarrow{C}}_0(N_1,\ldots,N_{k-1};i)$, then $\lambda^{b+1}\in\mathbb{\overleftarrow{C}}_0(N_1,\ldots,N_{k-1};i)$. Assume that $\lambda^{b}\in\mathbb{\overleftarrow{C}}_0(N_1,\ldots,N_{k-1};i)$, let $\lambda^{b,(k-1)}=(\lambda^{b,(k-1)}_1,\ldots, \lambda^{b,(k-1)}_{N_{k-1}})$ and $\lambda^{b+1,(k-1)}=(\lambda^{b+1,(k-1)}_1,\ldots, \lambda^{b+1,(k-1)}_{N_{k-1}})$ be the $(k-1)$-marked parts in $\overline{GG}(\lambda^{b})$ and $\overline{GG}(\lambda^{b+1})$, respectively. Then, $\{\lambda^{b,(k-1)}_s\}_{k-1}$ is good for $1\leq s\leq N_{k-1}$.

   For each $s\in[1,N_{k-1}]$, by the construction of $\lambda^{b+1}$, we find that $\lambda^{b+1,(k-1)}_s\not\in (2t_{b+1}+2,2t_{b})$ and there are no odd parts greater than $2t_{b}$ in $\lambda^{b+1}$.  To show that $\{\lambda^{b+1,(k-1)}_s\}_{k-1}$ is  good, we consider the following six cases.

  \begin{itemize}
  
  \item[(1)] $\lambda^{b+1,(k-1)}_s\leq 2t_{b+1}+2$.  With the similar argument in Step 1, we derive that $\{\lambda^{b+1,(k-1)}_s\}_{k-1}$ is  good. 

  \item[(2)] $\lambda^{b+1,(k-1)}_s= 2t_{b}$, $r_{b}=k-1$ and there is a $(k-1)$-marked $2t_{b}+1$ in $\overline{GG}(\lambda^{b})$. In this case, we have  $\lambda^{b,(k-1)}_s=2t_{b}+1$,  $\lambda^{b,(k-1)}_{s,2}=\cdots=\lambda^{b,(k-1)}_{s,k-1}=\lambda^{b+1,(k-1)}_{s,2}=\cdots=\lambda^{b+1,(k-1)}_{s,k-1}=2t_{b}$ and  $O_{\lambda^{b+1}}(2t_{b})=O_{\lambda^{b}}(2t_{b}+1)$, which implies that
         \[
\begin{split}
&[\lambda^{b+1,(k-1)}_{s,1}/2]+[\lambda^{b+1,(k-1)}_{s,2}/2]+\cdots+[\lambda^{b+1,(k-1)}_{s,k-1}/2]\\
&\qquad=t_{b}f_{2t_{b}}(\lambda^{b+1})=t_{b}(k-1)\\
&\qquad =t_{b}(k-2)+t_{b}=t_{b}f_{2t_{b}}(\lambda^{b})+t_{b}f_{2t_{b}+1}(\lambda^{b})\\
&\qquad =[\lambda^{b,(k-1)}_{s,1}/2]+[\lambda^{b,(k-1)}_{s,2}/2]+\cdots+[\lambda^{b,(k-1)}_{s,k-1}/2]\\
&\qquad \equiv i-1+O_{\lambda^{b}}(\lambda^{b,(k-1)}_{s,1})=i-1+O_{\lambda^{b}}(2t_{b}+1)\\
&\qquad = i-1+O_{\lambda^{b+1}}(2t_{b})=i-1+O_{\lambda^{b+1}}(\lambda^{b+1,(k-1)}_{s,1})\pmod2.
\end{split}
\]
So, $\{\lambda^{b+1,(k-1)}_s\}_{k-1}$ is   good.

      \item[(3)] $\lambda^{b+1,(k-1)}_s= 2t_{b}$, $r_{b}=k-1$ and there is a $(k-1)$-marked $\overline{2t_{b}+1}$ in $\overline{GG}(\lambda^{b})$. In this case, we have $\lambda^{b,(k-1)}_{s,2}=\overline{2t_{b}+1}$, $\lambda^{b,(k-1)}_{s,1}=2t_{b}+2$, $\lambda^{b,(k-1)}_{s,3}=\cdots=\lambda^{b,(k-1)}_{s,k-1}=\lambda^{b+1,(k-1)}_{s,2}=\cdots=\lambda^{b+1,(k-1)}_{s,k-2}=2t_{b}$,        $\lambda^{b+1,(k-1)}_{s,k-1}=2t_{b}-2$ and $O_{\lambda^{b+1}}(2t_{b})=O_{\lambda^{b}}(2t_{b}+2)$. So, we get 
            \[
\begin{split}
&[\lambda^{b+1,(k-1)}_{s,1}/2]+[\lambda^{b+1,(k-1)}_{s,2}/2]+\cdots+[\lambda^{b+1,(k-1)}_{s,k-1}/2]\\
&\qquad =(t_{b}-1)f_{2t_{b}-2}(\lambda^{b+1})+t_{b}f_{2t_{b}}(\lambda^{b+1})=(t_{b}-1)+t_{b}(k-2)\\
&\qquad \equiv t_{b}(k-3)+t_{b}+(t_{b}+1)\\
&\qquad =t_{b}f_{2t_{b}}(\lambda^{b})+t_{b}f_{\overline{2t_{b}+1}}(\lambda^{b})+(t_{b}+1)f_{2t_{b}+2}(\lambda^{b})\\
&\qquad =[\lambda^{b,(k-1)}_{s,1}/2]+[\lambda^{b,(k-1)}_{s,2}/2]+\cdots+[\lambda^{b,(k-1)}_{s,k-1}/2]\\
&\qquad \equiv i-1+O_{\lambda^{b}}(\lambda^{b,(k-1)}_{s,1})=i-1+O_{\lambda^{b}}(2t_{b}+2)\\
&\qquad = i-1+O_{\lambda^{b+1}}(2t_{b})=i-1+O_{\lambda^{b+1}}(\lambda^{b+1,(k-1)}_{s,1})\pmod2,
\end{split}
\]
which implies that $\{\lambda^{b+1,(k-1)}_s\}_{k-1}$ is good.

          \item[(4)] $\lambda^{b+1,(k-1)}_s= 2t_{b}$ and  $r_{b}<k-1$. By the construction of $\lambda^{b+1}$, we find that there is a $r_{b}$-marked $2t_{b}-2$ in  $\overline{GG}(\lambda^{b})$, which is replaced by $2t_{b}-3$ or $\overline{2t_{b}-3}$ in $\lambda^{b+1}$.
            Moreover, we have  $\lambda^{b,(k-1)}_s=2t_{b}$, $f_{2t_{b}-2}(\lambda^{b+1})=f_{2t_{b}-2}(\lambda^{b})-1$, $f_{2t_{b}}(\lambda^{b+1})=f_{2t_{b}}(\lambda^{b})+1$ and $O_{\lambda^{b+1}}(2t_{b})=O_{\lambda^{b}}(2t_{b})+1$. So, we get
         \[
\begin{split}
&[\lambda^{b+1,(k-1)}_{s,1}/2]+[\lambda^{b+1,(k-1)}_{s,2}/2]+\cdots+[\lambda^{b+1,(k-1)}_{s,k-1}/2]\\
&\qquad =(t_{b}-1)f_{2t_{b}-2}(\lambda^{b+1})+t_{b}f_{2t_{b}}(\lambda^{b+1})\\
&\qquad =(t_{b}-1)(f_{2t_{b}-2}(\lambda^{b})-1)+t_{b}(f_{2t_{b}}(\lambda^{b})+1)\\
&\qquad =(t_{b}-1)f_{2t_{b}-2}(\lambda^{b})+t_{b}f_{2t_{b}}(\lambda^{b})+1\\
&\qquad =[\lambda^{b,(k-1)}_{s,1}/2]+[\lambda^{b,(k-1)}_{s,2}/2]+\cdots+[\lambda^{b,(k-1)}_{s,k-1}/2]+1\\
&\qquad \equiv i-1+O_{\lambda^{b}}(\lambda^{b,(k-1)}_{s,1})+1=i-1+O_{\lambda^{b}}(2t_{b})+1\\
&\qquad = i-1+O_{\lambda^{b+1}}(2t_{b})=i-1+O_{\lambda^{b+1}}(\lambda^{b+1,(k-1)}_{s,1})\pmod2.
\end{split}
\]
Therefore, $\{\lambda^{b+1,(k-1)}_s\}_{k-1}$ is good.

   \item[(5)] $\lambda^{b+1,(k-1)}_s= 2t_{b}+2$.  In this case, we have $\lambda^{b,(k-1)}_s=2t_{b}+2$. Moreover, we find that $f_{2t_{b}}(\lambda^{b+1})=f_{2t_{b}}(\lambda^{b})+1$, $f_{2t_{b}+1}(\lambda^{b+1})+f_{\overline{2t_{b}+1}}(\lambda^{b+1})=f_{2t_{b}+1}(\lambda^{b})+f_{\overline{2t_{b}+1}}(\lambda^{b})-1$,  $f_{2t_{b}+2}(\lambda^{b+1})=f_{2t_{b}+2}(\lambda^{b})$ and $O_{\lambda^{b+1}}(2t_{b}+2)=O_{\lambda^{b}}(2t_{b}+2)$. So
        \[
\begin{split}
&[\lambda^{b+1,(k-1)}_{s,1}/2]+[\lambda^{b+1,(k-1)}_{s,2}/2]+\cdots+[\lambda^{b+1,(k-1)}_{s,k-1}/2]\\
&\qquad =t_{b}f_{2t_{b}}(\lambda^{b+1})+t_{b}(f_{2t_{b}+1}(\lambda^{b+1})+f_{\overline{2t_{b}+1}}(\lambda^{b+1}))+(t_{b}+1)f_{2t_{b}+2}(\lambda^{b+1})\\
&\qquad=t_{b}(f_{2t_{b}}(\lambda^{b})+1)+t_{b}(f_{2t_{b}+1}(\lambda^{b})+f_{\overline{2t_{b}+1}}(\lambda^{b})-1)+(t_{b}+1)f_{2t_{b}+2}(\lambda^{b})\\
&\qquad=t_{b}f_{2t_{b}}(\lambda^{b})+t_{b}(f_{2t_{b}+1}(\lambda^{b})+f_{\overline{2t_{b}+1}}(\lambda^{b}))+(t_{b}+1)f_{2t_{b}+2}(\lambda^{b})\\
&\qquad =[\lambda^{b,(k-1)}_{s,1}/2]+[\lambda^{b,(k-1)}_{s,2}/2]+\cdots+[\lambda^{b,(k-1)}_{s,k-1}/2]\\
&\qquad \equiv i-1+O_{\lambda^{b}}(\lambda^{b,(k-1)}_{s,1})+1=i-1+O_{\lambda^{b}}(2t_{b}+2)\\
&\qquad = i-1+O_{\lambda^{b+1}}(2t_{b}+2)=i-1+O_{\lambda^{b+1}}(\lambda^{b+1,(k-1)}_{s,1})\pmod2,
\end{split}
\]
which yields that $\{\lambda^{b+1,(k-1)}_s\}_{k-1}$ is good.

    \item[(6)] $\lambda^{b+1,(k-1)}_s\geq 2t_{b}+4$. By the construction of $\lambda^{b+1}$, we derive that $\{\lambda^{b+1,(k-1)}_s\}_{k-1}=\{\lambda^{b,(k-1)}_s\}_{k-1}$ and $O_{\lambda^{b+1}}(\lambda^{b+1,(k-1)}_s)=O_{\lambda^{b}}(\lambda^{b,(k-1)}_s)$. Note that  $\{\lambda^{b,(k-1)}_s\}_{k-1}$ is good, So $\{\lambda^{b+1,(k-1)}_s\}_{k-1}$ is good.

   \end{itemize}
    Hence, the $(k-1)$-band  $\{\lambda^{b+1,(k-1)}_s\}_{k-1}$ is  good  for $s\in[1,N_{k-1}]$. By Theorem \ref{parity k-1 sequence-special}, we see that
$\lambda^{b+1}$ is a special partition in $\mathbb{\overleftarrow{C}}_0(N_1,\ldots,N_{k-1};i)$.

\end{itemize}

Finally, we set $\mu=\lambda^{p}$. Clearly,  $|\mu|=|\lambda|-2p+1$, $\mu$ is a special partition in $\mathbb{{C}}_1(N_1,\ldots,N_{k-1};i)$ and the largest odd part in $\mu$ is $2t_p+1$, which is marked with $r_p$ in ${GG}(\mu)$. Furthermore, if  $\lambda$ is a partition in $\mathbb{{C}}_0(N_1,\ldots,N_{k-1};i)$, then $\mu$ belongs to $\mathbb{{C}}_0(N_1,\ldots,N_{k-1};i)$. Now, we proceed to show that $2t_p+1$ is a reduction part of $\mu$ and the reduction index of $2t_p+1$ is $p$. Assume that  $\hat{\alpha}_p(\lambda)=\{p_1\ldots,p\}$, we consider the following four cases.
 \begin{itemize}
\item[(1)] $\hat{\alpha}_p(\lambda)$ is of cluster $\hat{A_1}$. Then, we have $r_p>2$ and $\mu^{(2)}_p=2t_p$.

 \item[(2)] $\hat{\alpha}_p(\lambda)$ is of cluster $\hat{A_2}$.  Then, we find that $r_p=1$, $\mu^{(2)}_p=2t_p+2$ and $\check{\beta}_p(\mu)=\{p_1,\ldots,p\}$ is of cluster $\check{A}_2$.

      \item[(3)] $\hat{\alpha}_p(\lambda)$ is of cluster $\hat{B}$. Then, we derive that   $r_p=1$, $\mu^{(2)}_p=2t_p$ and $\check{\beta}_p(\mu)=\{p_1,\ldots,p\}$ is of cluster $\check{B}$.

              \item[(4)] $\hat{\alpha}_p(\lambda)$ is of cluster $\hat{C}$.  Then, we have  $r_p=2$ and $\mu^{(2)}_p=2t_p+1$.

 \end{itemize}

So, we conclude that $2t_p+1$ is a reduction part in $\mu$ with reduction index $p$. Hence $\mu$ is a partition in $\mathbb{\check{C}}_1(N_1,\ldots,N_{k-1};i|p)$. Moreover, if  $\lambda$ is a partition in $\mathbb{\hat{C}}_0(N_1,\ldots,N_{k-1};i|p)$, then $\mu$ belongs to $\mathbb{\check{C}}_0(N_1,\ldots,N_{k-1};i|p)$.

To prove that the reduction $R_p$ is a bijection, we construct the inverse map of the reduction $R_p$, the dilation $H_p$. Let $\mu$ be a partition in $\mathbb{\check{C}}_j(N_1,\ldots,N_{k-1};i|p)$. Let $\mu^{(r)}=(\mu^{(r)}_1,\mu^{(r)}_2,\ldots,\mu^{(r)}_{N_r})$ denote $r$-marked parts in ${GG}(\mu)$, where $1\leq r\leq k-1$. Assume that the $r_p$-marked odd part $2t_p+1$ in ${GG}(\mu)$ is the reduction part of $\mu$.  We will define $\lambda=H_p(\mu)$ related to $\mu^{(2)}_p, \mu^{(2)}_{p-1},\ldots,\mu^{(2)}_1$ successively, and denote the intermediate special partitions by $\mu^{p},\mu^{p-1},\ldots,\mu^{1},\lambda$ with $\mu^{p}=\mu$.

\begin{itemize}
\item[] Step 1: If $p=1$, then go to Step 2 directly. If $p>1$, then set $b=p$ and repeat the following process.

   \begin{itemize}
       \item[(E)]  There is a $r_b$-marked $2t_b+1$ (resp. $\overline{2t_b+1}$) in $\overline{GG}(\mu^{b})$. There are the following two cases.
       
      \begin{itemize}
          \item[]  Case 1: $\check{\beta}_{b}(\mu)=\check{\beta}_{b-1}(\mu)$. In this case, we find that $2t_b+4$ occurs in $\mu$, which also appears $\mu^{b}$. Set $r_{b-1}$ to the largest integer such that $r_{b-1}\leq r_b$ and there is a  $r_{b-1}$-marked $2t_b+4$ in $\overline{GG}(\mu^{b})$.
           Set  $t_{b-1}=t_b+2$, we see that there is a $r_{b-1}$-marked $2t_{b-1}$ in $\overline{GG}(\mu^{b})$. Then $\mu^{b-1}$ is obtained by replacing the $r_b$-marked $2t_b+1$ (resp. $\overline{2t_b+1}$) by $2t_b+2$ and replacing   the $r_{b-1}$-marked $2t_{b-1}$ in $\overline{GG}(\mu^{b})$ by  $2t_{b-1}+1$ (resp. $\overline{2t_{b-1}+1}$).
           
            \item[] Case 2: $\check{\beta}_{b}(\mu)\neq \check{\beta}_{b-1}(\mu)$. We first replace the $r_b$-marked $2t_b+1$ (resp. $\overline{2t_b+1}$) in $\overline{GG}(\mu^{b})$ by $2t_b+2$. For $\mu^{(2)}_{b-1}$, there are five subcases.
            
              \begin{itemize}

\item[] Case 2.1: $\mu^{(2)}_{b-1}$ is of type $\check{A}_1$. We may write  $\mu^{(2)}_{b-1}=2t_{b-1}$. Let $r_{b-1}$ be the largest mark of parts  $2t_{b-1}$ in $\overline{GG}(\mu^{b})$. We have $r_{b-1}>2$. Then replace the $r_{b-1}$-marked $2t_{b-1}$ in  $\overline{GG}(\mu^{b})$ by $2t_{b-1}+1$ to get $\mu^{b-1}$.

\item[] Case 2.2: $\mu^{(2)}_{b-1}$ is of type $\check{A}_2$. We may write  $\mu^{(2)}_{b-1}=2t_{b-1}+2$. Then set $r_{b-1}=1$ and replace the $1$-marked $2t_{b-1}$ in $\overline{GG}(\mu^{b})$ by $2t_{b-1}+1$ to obtain $\mu^{b-1}$.

\item[] Case 2.3: $\mu^{(2)}_{b-1}$ is of type $\check{A}_3$. We may write  $\mu^{(2)}_{b-1}=2t_{b-1}$. Then set $r_{b-1}=2$ and replace the $2$-marked $2t_{b-1}$ in $\overline{GG}(\mu^{b})$ by $\overline{2t_{b-1}+1}$ to get $\mu^{b-1}$.

\item[] Case 2.4: $\mu^{(2)}_{b-1}$ is of type $\check{B}$. We may write $\mu^{(2)}_{b-1}=2t_{b-1}$. Let $r_{b-1}$ be the largest mark of parts  $2t_{b-1}$ in $\overline{GG}(\mu^{b})$. We have $r_{b-1}>2$. Then replace the $r_{b-1}$-marked $2t_{b-1}$ in  $\overline{GG}(\mu^{b})$ by $\overline{2t_{b-1}+1}$ to obtain $\mu^{b-1}$.

\item Case 2.5: $\mu^{(2)}_{b-1}$ is of type $\check{C}$. We may write $\mu^{(2)}_{b-1}=2t_{b-1}$. Then set $r_{b-1}=2$ and replace the $2$-marked $2t_{b-1}$ in $\overline{GG}(\mu^{b})$ by ${2t_{b-1}+1}$ to get $\mu^{b-1}$.

\end{itemize}
           
      \end{itemize}
      
     \item[(F)]  Replace $b$ by $b-1$. If $b=1$, then go the Step 2. If $b>1$, then go back to (E).
   \end{itemize}
   
\item[] Step 2: There is a $r_1$-marked $2t_1+1$ (resp. $\overline{2t_1+1}$) in $\overline{GG}(\mu^{1})$. $\lambda$ is obtained by replacing the  $r_1$-marked $2t_1+1$ (resp. $\overline{2t_1+1}$) in $\overline{GG}(\mu^{1})$ by $2t_1+2$.

\end{itemize}

 It can be verified  that the dilation $H_{p}$ is the inverse map of $R_{p}$.  This completes the proof.  \qed
 
 We conclude this section with an example illustrating the reduction $R_p$. Let $\lambda$ be the partition in \eqref{example-reduction-type-1}. We apply the  reduction $R_8$ to $\lambda$ to get the partition $\mu$ in \eqref{example-reduction-type-2}, that is, $\mu=R_8(\lambda)$. Here we just give the intermediate special partitions $\lambda^1$, $\lambda^2$, \ldots, $\lambda^7$, where the parts in boldface are changed parts. 
 {\footnotesize \[  
\overline{GG}(\lambda^1)=\left[
\begin{array}{cccccccccccccccccccccc}
 &&4&&&&12&&16&&&22&&&&30&&&36&\\
 &2&&6&&10&&14&&&20&&24&&28&&&34&&&38\\
1&&4&&8&&12&&16&&20&&24&&28&&&34&&{\bf 37}&
\end{array}
\right].\]
 }
 \[\downarrow\]
  {\footnotesize \[  
\overline{GG}(\lambda^2)=\left[
\begin{array}{cccccccccccccccccccccc}
 &&4&&&&12&&16&&&22&&&&30&&&36\\
 &2&&6&&10&&14&&&20&&24&&28&&&34&&38\\
1&&4&&8&&12&&16&&20&&24&&28&&{\bf 33}&&{\bf 36}&
\end{array}
\right].\]
 }
  \[\downarrow\]
   {\footnotesize \[  
\overline{GG}(\lambda^3)=\left[
\begin{array}{cccccccccccccccccccccc}
 &&4&&&&12&&16&&&22&&&&{\bf 29}&&&36\\
 &2&&6&&10&&14&&&20&&24&&28&&&34&&38\\
1&&4&&8&&12&&16&&20&&24&&28&&{\bf 32}&&36&
\end{array}
\right].\]
 }
   \[\downarrow\]
   {\footnotesize \[  
\overline{GG}(\lambda^4)=\left[
\begin{array}{cccccccccccccccccccccc}
 &&4&&&&12&&16&&&22&&&&{\bf 28}&&&&36\\
 &2&&6&&10&&14&&&20&&\overline{\bf 23}&&&28&&&34&&38\\
1&&4&&8&&12&&16&&20&&&24&&28&&32&&36&
\end{array}
\right].\]
 } 
   \[\downarrow\]
   {\footnotesize \[  
\overline{GG}(\lambda^5)=\left[
\begin{array}{cccccccccccccccccccccc}
 &&4&&&&12&&16&&&22&&&28&&&&36\\
 &2&&6&&10&&14&&\overline{\bf 19}&&{\bf 22}&&&28&&&34&&38\\
1&&4&&8&&12&&16&&20&&24&&28&&32&&36&
\end{array}
\right].\]
 } 
    \[\downarrow\]
   {\footnotesize \[  
\overline{GG}(\lambda^6)=\left[
\begin{array}{cccccccccccccccccccccc}
 &&4&&&&12&&\overline{\bf 15}&&&&22&&&28&&&&36\\
 &2&&6&&10&&14&&&{\bf 18}&&22&&&28&&&34&&38\\
1&&4&&8&&12&&&16&&20&&24&&28&&32&&36&
\end{array}
\right].\]
 } 
   \[\downarrow\]
   {\footnotesize \[  
\overline{GG}(\lambda^7)=\left[
\begin{array}{cccccccccccccccccccccc}
 &&4&&&&\overline{\bf 11}&&{\bf 14}&&&&&22&&&28&&&&36\\
 &2&&6&&10&&&14&&&18&&22&&&28&&&34&&38\\
1&&4&&8&&&12&&&16&&20&&24&&28&&32&&36&
\end{array}
\right].\]
 } 
 Finally, we get $\mu=\lambda^8$ in \eqref{example-reduction-type-2}. On the contrary, the same process could be run in reverse, which illustrates the dilation $\lambda=H_8(\mu)$.

 \section{Proof of Theorem \ref{equiv-main-thm}}

Assume that  $k\geq 3$, $k\geq i\geq 2$ and $j=0$ or $1$. The goal of this section is to give a proof of Theorem \ref{equiv-main-thm}. To this end, we first establish the following theorem.

  \begin{thm}\label{success-2}
 For $N_1\geq\cdots\geq N_{k-1}\geq0$, $1\leq p\leq N_2$ and $0\leq p_1\leq N_1$, assume that $\lambda\in\mathbb{\hat{C}}_j(N_1,\ldots,N_{k-1};i|p)$, $\hat{R}_{p}(\lambda)=s$ and  $\lambda\in\mathbb{C}^{<}_j(N_1,\ldots,N_{k-1};i|p_1,t_1)$.  Let $\lambda^{(r)}=(\lambda^{(r)}_1,\lambda^{(r)}_2,\ldots,\lambda^{(r)}_{N_r})$ denote $r$-marked parts in ${GG}(\lambda)$, where $1\leq r\leq k-1$. Set $\mu^{1}=I_{p_1,t_1}(\lambda)$ and $\mu^{2}=R_{p}(\lambda)$. Then we have

  \begin{itemize}

  \item[\rm{(1)}] If $2t_1+1< \lambda^{(1)}_{s}-2$, then  $\mu^{1}\in\mathbb{\hat{C}}_j(N_1+1,\ldots,N_{k-1};i|p)$.

    \item[\rm{(2)}] If $2t_1+1=\lambda^{(1)}_{s}-1$ and $\hat{\alpha}_p(\lambda)$ is of cluster $\hat{A}_1$ or $\hat{C}$, then  $\mu^{1}\in\mathbb{\hat{C}}_j(N_1+1,\ldots,N_{k-1};i|p)$.

        \item[\rm{(3)}] If $2t_1+1=\lambda^{(1)}_{s}-1$ and $\hat{\alpha}_p(\lambda)$ is of cluster $\hat{A}_2$ or $\hat{B}$, then  $\mu^{2}\in \mathbb{C}^{<}_j(N_1,\ldots,N_{k-1};i|p_1-1,t_1+1)$.

        \item[\rm{(4)}] If $2t_1+1> \lambda^{(1)}_{s}+2$, then  $\mu^{2}\in \mathbb{C}^{<}_j(N_1,\ldots,N_{k-1};i|p_1,t_1)$ or $\mu^{2}\in\mathbb{C}^{<}_j(N_1,\ldots,N_{k-1};i|p_1\break -1,t_1+1)$.

  \end{itemize}

  \end{thm}

  \pf Let $\mu^{1,(1)}=(\mu^{1,(1)}_1,\ldots,\mu^{1,(1)}_{N_1},\mu^{1,(1)}_{N_1+1})$ denote $1$-marked parts in ${GG}(\mu^{1})$ and let $\mu^{1,(r)}=(\mu^{1,(r)}_1,\ldots,\mu^{1,(r)}_{N_r})$ denote $r$-marked parts in ${GG}(\mu^{1})$, where $2\leq r\leq k-1$.  Let $\mu^{2,(r)}=(\mu^{2,(r)}_1,\ldots,\mu^{2,(r)}_{N_r})$ denote $r$-marked parts in ${GG}(\mu^{2})$ where $1\leq r\leq k-1$.

  (1) If $2t_1+1< \lambda^{(1)}_{s}-2$, then by the construction of $\mu^1$, we find that $p_1\geq s$, $2t_1+1$ is the largest odd part in $\mu^{1}$, and $\mu^{1,(2)}_{p}=\lambda^{(2)}_{p}$ or $\lambda^{(2)}_{p}+2$, and so $2t_1+1< \mu^{1,(2)}_{p}-2$. It implies that there does not exist a $1$-marked $\mu^{1,(2)}_{p}-1$ in $GG(\mu^{1})$ and there do not exist odd parts greater than or equal to $\mu^{1,(2)}_{p}$ in $\mu^{1}$. Moreover, $\mu^{1,(2)}_{p}$ is of type $\hat{A}$ or $\hat{B}$ or $\hat{C}$ in $\mu^{1}$. Hence, $\mu^{1}\in\mathbb{\hat{C}}_j(N_1+1,\ldots,N_{k-1};i|p)$.

  (2) If $2t_1+1=\lambda^{(1)}_{s}-1$ and $\hat{\alpha}_p(\lambda)$ is of cluster $\hat{A}_1$ or $\hat{C}$, then by the construction of $\mu^1$, we derive that $p_1=s$, $2t_1+1$ is the largest odd part in $\mu^{1}$ and $\mu^{1,(2)}_{p}=\lambda^{(2)}_{p}$.  Note that $2t_1+1=\lambda^{(1)}_{s}-1\leq \lambda^{(2)}_{p}-1=\mu^{1,(2)}_{p}-1$, we see that there do not exist odd parts greater than or equal to $\mu^{1,(2)}_{p}$ in $\mu^{1}$. Moreover, $\mu^{1,(2)}_{p}$ is of type $\hat{B}$ (resp. $\hat{A}$)  if  $\hat{\alpha}_p(\lambda)$ is of cluster $\hat{A}_1$ (resp. $\hat{C}$). Hence, $\mu^{1}\in\mathbb{\hat{C}}_j(N_1+1,\ldots,N_{k-1};i|p)$.

(3) If $2t_1+1=\lambda^{(1)}_{s}-1$ and $\hat{\alpha}_p(\lambda)$ is of cluster $\hat{A}_2$ or $\hat{B}$, then we have $p_1=s$ and $\lambda^{(1)}_{p_1}=\lambda^{(1)}_{s}=2t_1+2$. By the construction of $\mu^2$, we see that the $1$-marked $2t_1+2$ in $GG(\lambda)$ is replaced by  $\mu^{2,(1)}_{p_1}=\lambda^{(1)}_{p_1}-1=2t_1+1$, which is the largest odd part in $\mu^2$. Moreover, we derive that $\mu^{2,(1)}_{p_1}+2=2t_1+3$ and $\mu^{2,(1)}_{p_1-1}\geq2t_1+4>2t_1+3$. This implies that $\mu^{2}\in\mathbb{C}^{<}_j(N_1,\ldots,N_{k-1};i|p_1-1,t_1+1)$.

  (4) If $2t_1+1> \lambda^{(1)}_{s}+2$, then we have $p_1< s$.
 By the construction of $\mu^2$, we  replace the $1$-marked $\lambda^{(1)}_{s}$ (resp. a $r$-marked $\lambda^{(1)}_{s}+2$ ($r\geq 2$)) in $GG(\lambda)$ by   $\lambda^{(1)}_{s}-1$ (resp. $\lambda^{(1)}_{s}+1$) in $\mu^2$. We see that the largest odd part of $\mu^2$ is  $\lambda^{(1)}_{s}-1$ (resp. $\lambda^{(1)}_{s}+1$), and so the largest odd part in $\mu^2$ is less than $2t_1+1$. Moreover, we have $2t_1+1\geq \lambda^{(1)}_{p_1+1}+2\geq \mu^{2,(1)}_{p_1+1}+2$ and $\mu^{2,(1)}_{p_1}=\lambda^{(1)}_{p_1}$  or $\lambda^{(1)}_{p_1}-2$. Note that $2t_1+1<\lambda^{(1)}_{p_1}$, then we consider the following two cases.

Case 1: If $2t_1+1<\mu^{2,(1)}_{p_1}$, then it is clear that $\mu^{2}\in\mathbb{C}^{<}_j(N_1,\ldots,N_{k-1};i|p_1,t_1)$.

Case 2: If $2t_1+1>\mu^{2,(1)}_{p_1}$, then we have $\mu^{2,(1)}_{p_1}= \lambda^{(1)}_{p_1}-2$, and so $\lambda^{(1)}_{p_1}-2<2t_1+1<\lambda^{(1)}_{p_1}$. This implies that $\lambda^{(1)}_{p_1}=2t_1+2$ and $\mu^{2,(1)}_{p_1}=2t_1$. Moreover, we have $2t_1+3>\mu^{2,(1)}_{p_1}+2$ and $\mu^{2,(1)}_{p_1-1}\geq 2t_1+4>2t_1+3$. So, $\mu^{2}\in\mathbb{C}^{<}_j(N_1,\ldots,N_{k-1};i|p_1-1,t_1+1)$. 

Thus, we complete the proof. \qed

Now, we are in a position to give a proof of Theorem \ref{equiv-main-thm}.

{\noindent \bf Proof of Theorem \ref{equiv-main-thm}.} Let $(\lambda,\tau,\eta)$ be a triplet in  $\mathbb{F}_j(k,i)$. We define $\pi=\Phi_{k,i}(\lambda,\tau,\eta)$ as follows. Assume that there are $N_r$ $r$-marked parts in $GG(\lambda)$, where $1\leq r\leq k-1$, then we have $\tau\in\mathbb{O}_{N_2}$ and $\eta\in \mathbb{I}_{N_1}$.  Set $\ell_1=\ell(\tau)$, $\ell_2=\ell(\eta)$ and $\ell=\ell_1+\ell_2$.

Note that there do not exist odd parts in $\lambda$, then we  see that $\lambda\in\mathbb{\hat{C}}_j(N_1,\ldots,N_{k-1};i|p)$ for $1\leq p\leq N_2$ and $\lambda\in\mathbb{C}^{<}_j(N_1,\ldots,N_{k-1};i|m)$ for $m\geq N_1$. Let $\lambda^{0}=\lambda$, $\tau^{0}=\tau$ and $\eta^{0}=\eta$. If $\ell\geq 1$, then set $b=0$ and repeat the following process until $b=\ell$.

\begin{itemize}
    \item[(E)]  If $\tau^{b}\neq \emptyset$, then assume that $1-2p_b$ is the smallest part in $\tau^b$, where $1\leq p_b\leq N_2$. Otherwise, if $\tau^b=\emptyset$, then set $p_b=0$. If $\eta^b\neq \emptyset$, then assume that $2m_b+1$ is the smallest part in $\eta^b$, where $m_b\geq N_1+\ell_2-\ell(\eta^b)$. Otherwise, if $\eta^b=\emptyset$, then set $m_b=0$.
    
     If $p_b=0$, then apply the insertion $I_{m_b}$ to $\lambda^b$ to get $\lambda^{b+1}$, that is, $\lambda^{b+1}=I_{m_b}(\lambda^{b})$.
    
    If $m_b=0$, then apply the reduction $R_{p_b}$ to $\lambda^b$ to get $\lambda^{b+1}$, that is,  $\lambda^{b+1}=R_{p_b}(\lambda^{b})$.
    
    If $p_b\neq 0$ and $m_b\neq 0$, then we have $\lambda^b\in\mathbb{\hat{C}}_j(N_{1_b},\ldots,N_{k-1};i|p_b)$ and $\lambda^b\in\mathbb{C}^{<}_j(N_{1_b},\ldots,N_{k-1};i|m_b)$, where $N_{1_b}=N_1+\ell_2-\ell(\eta^b)$. 
    Let $\lambda^{b,(1)}=(\lambda^{b,(1)}_1,\ldots,\lambda^{b,(1)}_{N_{1_b}})$  denote $1$-marked parts in ${GG}(\lambda^{b})$ and  let $\lambda^{b,(r)}=(\lambda^{b,(r)}_1,\ldots,\lambda^{b,(r)}_{N_{r}})$  denote $r$-marked parts in ${GG}(\lambda^{b})$, where   $2\leq r\leq k-1$. Assume that $p'_b$ is the smallest integer such that $2(m_b-p'_b)+1\geq \lambda^{b,(1)}_{p'_b+1}+2$. Set  $t'_b=m_b-p'_b$, then  we have $\lambda^b\in\mathbb{C}^{<}_j(N_1,\ldots,N_{k-1};i|p'_b,t'_b)$. Assume that  $\hat{R}_{p_b}(\lambda^b)=s_b$, there are four cases.
    
     \begin{itemize}

  \item[\rm{(1)}] If $2t'_b+1\leq \lambda^{b,(1)}_{s_b}-2$, then apply the insertion $I_{m_b}$ to $\lambda^b$ to get $\lambda^{b+1}$, that is,    $\lambda^{b+1}=I_{m_b}(\lambda^{b})$.

    \item[\rm{(2)}] If $2t'_b+1=\lambda^{b,(1)}_{s_b}-1$ and $\hat{\alpha}_{p_b}(\lambda^b)$ is of cluster $\hat{A}_1$ or $\hat{C}$, then apply the insertion $I_{m_b}$ to $\lambda^b$ to get $\lambda^{b+1}$, that is, $\lambda^{b+1}=I_{m_b}(\lambda^{b})$.

         \item[\rm{(3)}] If $2t'_b+1=\lambda^{b,(1)}_{s_b}-1$ and $\hat{\alpha}_{p_b}(\lambda^b)$ is of cluster $\hat{A}_2$ or $\hat{B}$, then apply the reduction $R_{p_b}$ to $\lambda^b$ to get $\lambda^{b+1}$, that is, $\lambda^{b+1}=R_{p_b}(\lambda^{b})$.

        \item[\rm{(4)}] If $2t'_b+1> \lambda^{b,(1)}_{s_b}+2$, then we apply the reduction $R_{p_b}$ to $\lambda^b$ to get $\lambda^{b+1}$, that is, $\lambda^{b+1}=R_{p_b}(\lambda^{b})$.

  \end{itemize}
  
  If $\lambda^{b+1}=R_{p_b}(\lambda^{b})$, then set $\eta^{b+1}=\eta^b$ and remove $1-2p_b$ from $\tau^b$ to obtain $\tau^{b+1}$.
  
   If $\lambda^{b+1}=I_{m_b}(\lambda^{b})$, then set $\tau^{b+1}=\tau^b$ and remove $2m_b+1$ from $\eta^b$ to obtain $\eta^{b+1}$.
  
  \item[(F)] Replace $b$ by $b+1$. If $b=\ell$, then we are done. If $b<\ell$, then go back to (E).
    
\end{itemize}
  By Theorem \ref{ssins} and Theorem \ref{success-2}, we see that the process above is well-defined.
    Finally, we set $\pi=\lambda^{\ell}$.  Clearly, $\pi\in\mathbb{C}_j(N_1+\ell_2,\ldots,N_{k-1};i)$ and $|\pi|=|\lambda|+|\tau|+|\eta|$. It implies that  $\pi\in\mathbb{C}_j(k;i)$ and $\ell(\pi)=\ell(\lambda)+\ell(\eta)$.

To show that $\Phi_{k,i}$ is a bijection, we need to give the inverse map $\Psi_{k,i}$ of $\Phi_{k,i}$. Let $\pi$ be a partition in $\mathbb{C}_j(k;i)$. Assume that there are $\ell\geq 0$ odd parts in $\pi$. Let $\pi^{\ell}=\pi$, $\tau^{\ell}=\emptyset$ and $\eta^{\ell}=\emptyset$. If $\ell\geq 1$, then set $b=\ell$ and repeat the following process until $b=0$.

\begin{itemize}
    \item[(E)]  Let $\pi^{b,(r)}=(\pi^{b,(r)}_1,\ldots,\pi^{b,(r)}_{N_{r_b}})$ denote the $r$-marked parts in ${GG}(\pi^{b})$, where $1\leq r\leq k-1$. Assume that  $2t'_b+1$ is  the largest odd part in $\pi^{b}$. There are two cases.
    
    \begin{itemize}

 \item[\rm{(1)}] If $2t'_b+1$ is a reduction part of $\pi^{b}$, then assume that the reduction index of $2t'_b+1$ is $p_b$. We apply the dilation $H_{p_b}$ to $\pi^b$ to get $\pi^{b-1}$, that is, $\pi^{b-1}=H_{p_b}(\pi^{b})$. Then set $\eta^{b-1}=\eta^b$, and add $1-2p_b$ as a new part to $\tau^{b}$ to obtain $\tau^{b-1}$.

     \item[\rm{(2)}] If $2t'_b+1$ is a insertion part of $\pi^{b}$, then assume that $\pi^{b,(1)}_{p'_b+1}=2t'_b+1$. Let $m_b=p'_b+t'_b$. We apply the separation $S_{m_b}$ to $\pi^b$ to get $\pi^{b-1}$, that is, $\pi^{b-1}=S_{m_b}(\pi^{b})$. Then set $\tau^{b-1}=\tau^b$ and add $2m_b+1$ as a new part to $\eta^{b}$ to obtain $\eta^{b-1}$.

\end{itemize}

\item[(F)] Replace $b$ by $b-1$. If $b=0$, then we are done. If $b>0$, then go back to (E).
    
\end{itemize}

Finally, we set $\lambda=\pi^{0}$, $\tau=\tau^0$ and $\eta=\eta^{0}$.  Clearly, $\lambda\in\mathbb{E}_j(k,i)$, $|\pi|=|\lambda|+|\tau|+|\eta|$ and $\ell(\pi)=\ell(\lambda)+\ell(\eta)$. Moreover,  we find that for $0\leq b\leq \ell$, $N_{1_b}=N_{1_\ell}-\ell(\eta^b)$ and $N_{r_b}=N_{r_\ell}$, where $2\leq r\leq k-1$. Setting $N_r=N_{r_0}$, where $1\leq r\leq k-1$,  then we have $\tau\in\mathbb{O}_{N_2}$. Assume that $\eta=(\eta_1,\ldots,\eta_{\ell_2})$, where $\eta_s=1+2m_{b_s}$, where $1\leq s\leq \ell_2$. For $1\leq s< \ell_2$, we see that
\[\begin{split}
&\eta_s=1+2m_{b_s}=1+2p'_{b_s}+2t'_{b_s}\geq 1+2p'_{b_s}+2(p'_{b_{s+1}}+1-p'_{b_s})+2t'_{b_{s+1}}\\
&\quad=1+2(p'_{b_{s+1}}+1)+2t'_{b_{s+1}}>1+2p'_{b_{s+1}}+2t'_{b_{s+1}}=1+2m_{b_{s+1}}=\eta_{s+1}.
\end{split}
\]
Moreover, we have $\eta_{\ell_2}=1+2m_{b_{\ell_2}}\geq 1+2N_{1_{b_{\ell_2}-1}}\geq 1+2N_{1_0}=1+2N_1$. So, $\eta_1>\eta_2>\cdots>\eta_{\ell_2}\geq 1+2N_1$, which implies $\eta\in\mathbb{I}_{N_1}$. Hence, $(\lambda,\tau,\eta)$ is a triplet in  $\mathbb{F}_j(k,i)$. It follows from Theorem \ref{deltagammathmbb-insertion-all} and Theorem \ref{reduction-1} that $\Psi_{k,i}$ is the inverse map of $\Phi_{k,i}$.  Thus, we complete the proof. \qed

\noindent{\bf Acknowledgments.}
This work was supported by
the National Science Foundation of
China  (Nos. 12101437 and 11901430).

\end{document}